\pdfoutput=1
\documentclass[reqno]{amsart}
\usepackage{graphicx,color}
\usepackage{empheq}
\usepackage{amsmath,amssymb}
\usepackage{amsthm}
\usepackage{mathrsfs}
\usepackage{bm}
\usepackage{mathtools}
\usepackage{subcaption}
\theoremstyle{definition}
\newtheorem{theorem}{Theorem}[section]
\newtheorem{lemma}[theorem]{Lemma}
\newtheorem{prop}[theorem]{Proposition}
\newtheorem{definition}[theorem]{Definition}
\newtheorem{conjecture}[theorem]{Conjecture}

\title{A proof of the Teichm\"{u}ller TQFT volume conjecture for $7_3$ knot}
\author{Soichiro Uemura}
\address{Graduate School of Mathematical Sciences\\
  the University of Tokyo, Tokyo, Japan}
\email{soichiro.uemura@ipmu.jp}
\date{\today}
\begin{document}
\begin{abstract}
In the generalized topological quantum field theory constructed by Andersen and Kashaev, invariants of 3-manifolds are defined given the combinatorial structure of a tetrahedral decomposition. Furthermore, a variant of the volume conjecture has been proposed in which the hyperbolic volume can be extracted from this invariant of the complementary space of the hyperbolic knot in an oriented $3$-dimensional closed manifold. We prove that the reformulated volume conjecture holds for the complementary space of the hyperbolic knot $7_3$ in $S^3$, given a specific tetrahedral decomposition.
\end{abstract}
\maketitle
\tableofcontents
\section{Introduction}
A Topological Quantum Field Theory (TQFT) is a theory that satisfies the axiom of Atiyah, Segal, and Witten \cite{MR1001453,MR0981378,MR0953828}.
It is defined as a functor from the cobordism category to the module category. The Jones-Witten theory is an example of constructing topological invariants of 3-manifolds based on the TQFT. This invariant is given for the principal $SU(2)$ bundle on an oriented compact 3-manifold $M$ by the path integral using the Chern-Simons functional described in the following equation \eqref{CS}.  

Let $P$ be a principal $G$ bundle on $M$ for a simple, connected, simply-connected compact Lie group $G$, and denote by $\mathcal{A}$ the space of connections on $P$, then $\mathcal{A}$ is identified with the space $\Omega^1(M,\mathfrak{g})$ of first-order differential forms on $M$ with values in the Lie ring $\mathfrak{g}$ of $G$.
In this case, the Chern-Simons functional $CS:\mathcal{A}\rightarrow\mathbb{R}$ is defined by the following equation, \begin{equation}
CS(A)=\int_M {\rm{Tr}}\left(A\wedge dA+\frac{2}{3}A\wedge A\wedge A\right),\quad A\in\mathcal{A}.\label{CS}
\end{equation}

The TQFT based on the path integral or the partition function using the Chern-Simons functional is called the Chern-Simons theory. 
The Jones-Witten theory is for $G=SU(2)$, and the path integral is formally 
expressed as 
\begin{equation*}
Z_k(M)=\int_{\mathcal{A}/\mathcal{G}}\exp\left({\frac{\sqrt{-1}k}{4\pi}CS(A)}\right)\mathcal{D}A, 
\end{equation*}
where $k$ is an integer called the level of the theory, and $\mathcal{G}$ denotes the gauge transformation group \cite{MR0990772}. 

Such topological invariants of 3-manifolds based on the path integral are called quantum invariants. In particular, the first mathematically rigorously defined quantum invariant called the WRT invariant is obtained from the TQFT constructed by Reshetikhin and Turaev \cite{MR1036112,MR1091619,MR1292673}. 

Also related to these quantum invariants, a polynomial knot invariant called the colored Jones polynomial is defined as the path integral on the Wilson loop. In particular, for hyperbolic knots in $S^3$, the volume conjecture, that the hyperbolic volume of the knot complement space appears in the asymptotic expansion of the colored Jones polynomial at some radical root of unity has been proposed by Kashaev, Murakami and Murakami \cite{MR1434238, MR1828373} and proved rigorously for several hyperbolic knots, but the conjecture is not completely solved. The volume conjecture suggests a close relationship between topology and hyperbolic geometry. Chen and Yang have also proposed similar volume conjectures for quantum invariants such as the WRT invariant and the Turaev-Viro invariant obtained by TQFT as well \cite{MR3827806}.

 In contrast, the Chern-Simons theory corresponding to the case where the structure group $G$ is noncompact, compared to the case where $G$ is compact, is challenging to analyze and poorly understood. Recently, Andersen and Kashaev have constructed a generalized TQFT that is expected to have a correspondence with the Chern-Simons theory of $G=SL(2,\mathbb{C})$ by using the quantum Teichm\"{u}ller theory, which has the feature of generating unitary representations of the centrally extended mapping class group of a punctured surface into the infinite-dimensional Hilbert space \cite{MR3227503}. 
 A series of conjectures (Conjecture \ref{Andersen-Kashaev}) analogous to the volume conjecture have been proposed for the invariant defined by this TQFT. This conjecture suggests a new relationship between hyperbolic geometry and topology. Theorems similar to this conjecture have been proved for $4_1$ and $5_2$ knots in $S^3$ \cite{MR3227503} and all twist knots in $S^3$ \cite{MR3945172} for a certain tetrahedrad decomposition.
 
In this paper, we compute and explicitly show the invariant based on this TQFT (hereafter referred to as the Teichm\"{u}ller TQFT) for the hyperbolic knot $7_3$ in $S^3$, which has not been investigated in previous studies. By rigorously evaluating the integral using the saddle point method, we prove a conjecture analogous to the volume conjecture for $7_3$ knot in $S^3$ in a reformulated form for a certain tetrahedral decomposition, and obtain the main result, Theorem \ref{uemura}.

This paper is organized as follows. Section \ref{A-K-TQFT} introduces the Teichm\"{u}ller TQFT proposed by Andersen and Kashaev. Section \ref{triangulation} obtains an ideal tetrahedral decomposition of the complementary space $S^3 \backslash 7_3$ and a tetrahedral decomposition of $S^3$ called one vertex H-triangulation such that a knot $7_3$ in $S^3$ is an edge of a single tetrahedron. 
Section \ref{ideal-triangulation} obtains the partition function defined in the Teichm\"{u}ller TQFT for the ideal tetrahedral decomposition of $S^3\backslash 7_3$. Section \ref{one-vertex_Htriangulation} obtains the partition function for one vertex H-triangulation. 
In Section \ref{volume-conjecture-strict-proof}, we prove Theorem \ref{uemura} analogous to the volume conjecture by evaluating the integral rigorously. In Appendix \ref{volume-conjecture}, an approximate numerical analysis also verifies Theorem \ref{uemura} (3).
\section*{Acknowledgments}
The author thanks Takahiro Kitayama and Masahito Yamazaki for valuable comments and discussions.

\section{Andersen-Kashaev's Teichm\"{u}ller TQFT}\label{A-K-TQFT}
This section summarizes the results of the Teichm\"{u}ller TQFT \cite{MR3227503} of Andersen and Kashaev.
\subsection{Oriented triangulated pseudo 3-manifolds}
Consider the disjoint union of a finite number of copies of a standard 3-simplex in $\mathbb{R}^3$. Each simplex is assumed to have totally ordered vertices. The order of the vertices induces the orientation of the edges. Some pairs of faces of codimension 1 of this disjoint union are glued and identified by an affine homeomorphism, called gluing homeomorphism, which preserves the order of the vertices and reverses the faces' orientations. The quotient space $X$, by this identification, is a CW-complex with oriented edges called an oriented triangulated pseudo 3-manifold. 
For $i\in\left\{0,1,2,3\right\}$, $\Delta_i(X)$ denotes the i-dimensional cell of $X$. For any $i>j$, we define 
\begin{equation*}
\Delta_i^j(X)=\left\{(a,b)\mid a\in\Delta_i(X),b\in\Delta_j(a)\right\}.
\end{equation*} 

We also define the natural map as 
\begin{equation*}
\phi_{i,j}:\Delta_i^j(X)\rightarrow\Delta_i(X),\ \ \phi^{i,j}:\Delta_i^j(X)\rightarrow\Delta_j(X).
\end{equation*} 

Also, the standard boundary map 
\begin{equation*}
\partial_i:\Delta_j(X)\rightarrow \Delta_{j-1}(X),\ \ 0\le i\le j
\end{equation*}
is defined such that for a $j$ dimensional simplex $S=[v_0,v_1,\ldots,v_j]$ in $\mathbb{R}^3$ with ordered vertices $v_0,v_1,\ldots,v_j$,   
\begin{equation*}
\partial_iS=[v_0,\ldots,v_{i-1},v_{i+1},\ldots,v_j],\ \ i\in\{0,\ldots,j\}.
\end{equation*} 
\subsection{Shaped 3-manifolds}
Let $X$ be an oriented triangulated pseudo 3-manifold. 
\begin{definition}
The shape structure on $X$ is an assignment of a positive real number 
called the dihedral angle to each edge of each tetrahedron  
\begin{equation*}
\alpha_X:\Delta_3^1(X)\rightarrow\mathbb{R}_+
\end{equation*}
such that the sum of the dihedral angles at the three edges extending from each vertex of each tetrahedron is $\pi$. 
An oriented triangulated pseudo 3-manifold with a shape structure is called a shaped pseudo 3-manifold. Let $S(X)$ be the set of all shape structures on $X$. \label{shape}
\end{definition}
By Definition \ref{shape}, for any tetrahedron, the dihedral angles at opposite edges are equal. Thus, for each tetrahedron, three dihedral angles are associated with three pairs of opposite edges whose sum is $\pi$. 
\begin{definition}
For each shape structure on X, the weight function
\begin{equation*}
\omega_X:\Delta_1(X)\rightarrow\mathbb{R}_+
\end{equation*}
is defined by mapping each edge $e$ of $X$ to the sum of the dihedral angles around it.  
That is, for any $e\in\Delta_1(X)$, we define 
\begin{equation*}
\omega_X(e)=\sum_{a\in(\phi^{3,1})^{-1}(e)} \alpha_X(a).
\end{equation*}
\end{definition}
\begin{definition}
An edge $e$ of shaped pseudo 3-manifold $X$ is called balanced if it is inside $X$ and $\omega_X(e)=2\pi$. 
An edge that is not balanced is called unbalanced. 
We say that a shaped pseudo 3-manifold $X$ is fully balanced if every edge of $X$ is balanced. 
\end{definition}
\subsection{$\mathbb{Z}/3\mathbb{Z}$ action on the set of opposite edge pairs of a tetrahedron}
Let $X$ be an oriented triangulated pseudo 3-manifold. The orientation on $X$ gives a cyclic ordering to the three edges extending from each tetrahedron vertex. Moreover, this cyclic order induces a cyclic order on the set of opposite pairs of tetrahedron edges.
Let $\Delta_3^{1/p}(X)$ be the set consisting of the opposite edge pairs of all tetrahedra. 
$\Delta_3^{1/p}(X)$ is the quotient set for $\Delta_3^1(X)$ by the equivalence relation generated by the opposite edge pairs of all tetrahedra, and let the corresponding quotient map be
\begin{equation*}
    p:\Delta_3^1(X)\rightarrow\Delta_3^{1/p}(X).
\end{equation*}

Skew-symmetric function
\begin{equation*}
\epsilon_{a,b}\in\{0,\pm 1\},\quad \epsilon_{a,b}=-\epsilon_{b,a},\quad a,b\in\Delta_3^{1/p}(X)
\end{equation*}
is set $\epsilon_{a,b}=0$ if the tetrahedra in which $a$ and $b$ exist are different, and $\epsilon_{a,b}=1$ if the tetrahedra in which $a$ and $b$ exist coincide and $a$ is ahead of $b$ under the cyclic order.
\subsection{Leveled shaped 3-manifolds}
\begin{definition}
A leveled shaped 3-manifold is a pair $(X,l_X)$ consisting of a shaped pseudo 3-manifold $X$ and a real number $l_X\in\mathbb{R}$ called the level. The leveled shaped structure on an oriented triangulated pseudo 3-manifold $X$ is a pair $(\alpha_X,l_X)$ of a shape structure on $X$ and a level. Let $LS(X)$ denote the set of all leveled shaped structures on $X$. 
\end{definition}
\begin{definition}\label{gauge-equivalent}
Two leveled shaped pseudo 3-manifolds $(X,l_X)$ and $(Y,l_Y)$ are said to be gauge equivalent if there exists an isomorphism of cellular structures of $X$ and $Y$ 
\begin{equation*}
h:X\rightarrow Y    
\end{equation*}
and a function
\begin{equation*}
g:\Delta_1(X)\rightarrow \mathbb{R}    
\end{equation*}
which satisfies the following conditions. 
\begin{equation*}
\Delta_1(\partial X)\subset g^{-1}(0)
\end{equation*}
\begin{equation*}
\alpha_Y(h(a))=\alpha_X(a)+\pi \sum_{b\in\Delta_3^1(X)}\epsilon_{p(a),p(b)}g(\phi^{3,1}(b)),\quad\forall a\in\Delta_3^1(X)    
\end{equation*}
\begin{equation*}
l_Y=l_X+\sum_{e\in\Delta_1(X)}g(e)\sum_{a\in (\phi^{3,1})^{-1}(e)}\left(\frac{1}{3}-\frac{\alpha_X(a)}{\pi}\right)
\end{equation*}
\end{definition}
If $X$ and $Y$ are gauge equivalent, weights on edges are gauge invariant in the sense that 
\begin{equation*}
\omega_X=\omega_Y\circ h
\end{equation*}
holds.

\begin{definition}
Two leveled shape structures $(\alpha_X, l_X)$ and $(\alpha_X^\prime, l_X^\prime)$ on a pseudo 3-manifold $X$ are called based gauge equivalent if they are gauge equivalent with $h:X\rightarrow X$ as the identity map in Definition \ref{gauge-equivalent}. 
\end{definition}
The (based) gauge equivalence relation on leveled shaped pseudo 3-manifolds induces a (based) gauge equivalence relation on shaped pseudo 3-manifolds under the map which forgets the level. Let $LS_r(X)$ denote the set of gauge equivalence classes of based leveled shape structures on $X$ and $S_r(X)$ denote the set of gauge equivalence classes of based shape structures on $X$. 
\begin{definition}
A generalized shape structure on $X$ is an assignment of real numbers to each edge of each tetrahedron such that the sum of the real numbers assigned to the three edges extending from each vertex of each tetrahedron is $\pi$. Let $\widetilde{S}(X)$ denote the set of all generalized shape structures on $X$ and $\widetilde{LS}(X)$ denote the set of all generalized leveled shape structures on $X$. The leveled generalized shape structure and its gauge equivalence are defined in the same way as the leveled shaped structure and its gauge equivalence. We denote the space of based gauge equivalence classes of generalized shape structures by $\widetilde{S}_r(X)$ and the space of based gauge equivalence classes of leveled generalized shape structures by $\widetilde{LS}_r(X)$.
\end{definition}

$S_r(X)$ is an open convex subset of $\widetilde{S}_r(X)$.

Let the following map denote the map which assigns the weight function $\omega_X:\Delta_1(X)\rightarrow \mathbb{R}$ to the generalized shape structure $\alpha_X\in \widetilde{S}(X)$,  
\begin{equation*}
\widetilde{\Omega}_X:\widetilde{S}(X)\rightarrow\mathbb{R}^{\Delta_1(X)}.
\end{equation*}

Since this map is gauge equivalent, it induces the following unique map.
\begin{equation*}
\widetilde{\Omega}_{X,r}:\widetilde{S}_r(X)\rightarrow\mathbb{R}^{\Delta_1(X)}.
\end{equation*}

Let $N_0(X)$ be a sufficiently small tubular neighborhood of $\Delta_0(X)$. $\partial N_0(X)$ is a two-dimensional surface. It can be disconnected and with a boundary if $\partial X\neq\emptyset$.
\subsection{3-2 Pachner move}
Let $X$ be a shaped pseudo 3-manifold.
Let $e$ be a balanced edge of $X$ shared by three different tetrahedra $t_1,t_2$, and $t_3$.
Let $S$ be a shaped pseudo 3-submanifold of $X$ consisting of tetrahedra $t_1,t_2$, and $t_3$. There exists another triangulation $S_e$ of $S$ consisting of two tetrahedra $t_4,t_5$ so that the triangulations of $\partial S$ and $\partial S_e$ can be made to coincide. 
This is obtained by removing the edge $e$, then $\Delta_1(S_e)=\Delta_1(S)\backslash{e}$. 
Furthermore, there is a unique shape structure on $S_e$ that induces the same weight as that induced by the shape structure on $S$. 
Let $(\alpha_i,\beta_i,\gamma_i)$ be the dihedral angles of $t_i$ ( for $i=1,2,3$, let $\alpha_i$ be the dihedral angle on $e$). The following relation holds. 
\begin{equation*}
\alpha_4=\beta_2+\gamma_1,\quad\alpha_5=\beta_1+\gamma_2
\end{equation*}
\begin{equation*}
\beta_4=\beta_1+\gamma_3,\quad\beta_5=\beta_3+\gamma_1\\
\end{equation*}
\begin{equation*}
\gamma_4=\beta_3+\gamma_2,\quad\gamma_5=\beta_2+\gamma_3    
\end{equation*}

Since the edge $e$ is balanced, i.e.,  $\alpha_1+\alpha_2+\alpha_3=2\pi$, we can see that the sum of the dihedral angles of $t_4$ and $t_5$ is $\pi$, respectively, taking into consideration the above relation. Furthermore, since the dihedral angles of $t_1,t_2$, and $t_3$ are positive, the dihedral angles of $t_4$ and $t_5$ are also positive by the above equations. Conversely, given the dihedral angles of $t_4$ and $t_5$, it does not necessarily follow that the solutions for the positive dihedral angles of $t_1,t_2$ and $t_3$ are obtained. However, under common dihedral angles of $t_4,t_5$, if multiple solutions of positive dihedral angles of $t_1,t_2,t_3$ are obtained, they are gauge equivalent and  $\alpha_1+\alpha_2+\alpha_3=2\pi$. (Figure \ref{fig:3-2_Pachner}) . 
\begin{figure}[htbp]
\centering
\includegraphics[width=0.7\textwidth]{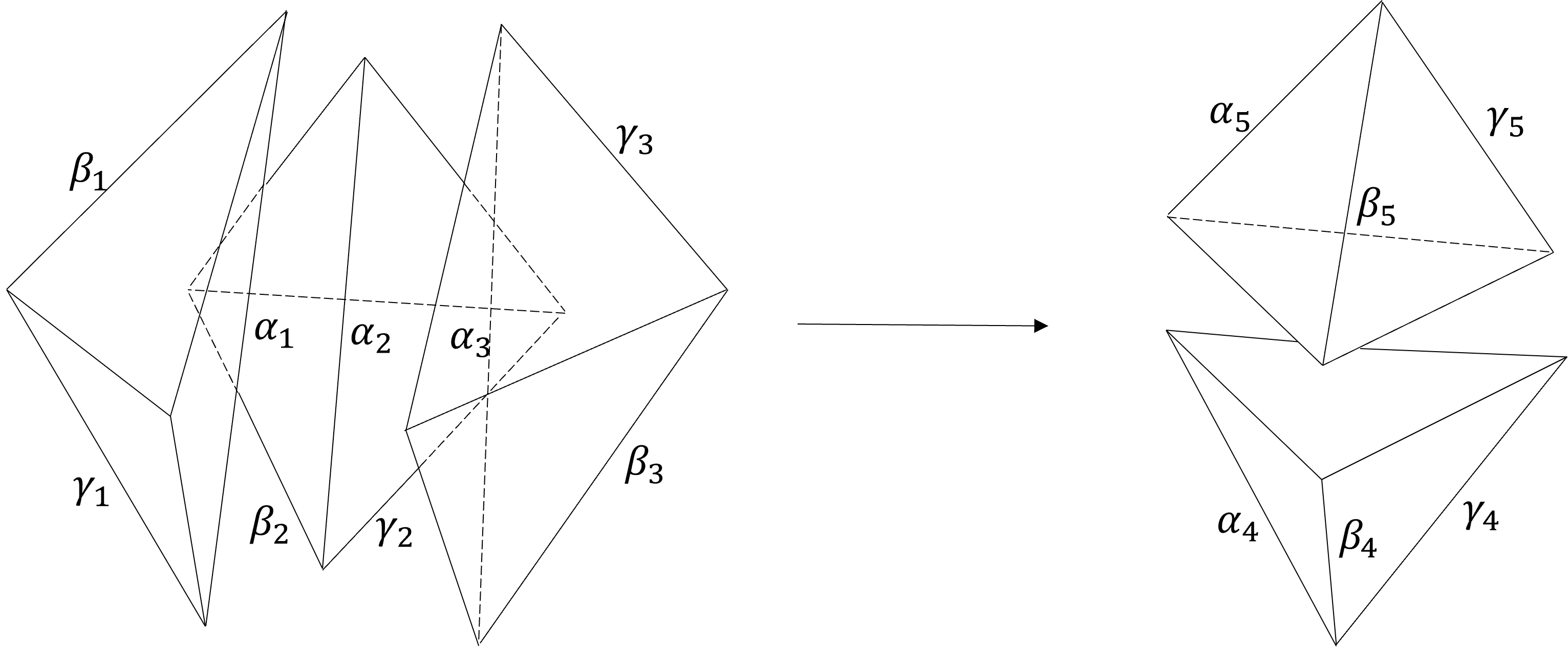}
\caption{3-2 Pachner move.}
\label{fig:3-2_Pachner}
\end{figure}
\begin{definition}
If some shaped pseudo 3-manifold $Y$ is obtained from $X$ by replacing $S$ by $S_e$, we say that it is obtained from $X$ by a shaped 3-2 Pachner move along $e$. 
\end{definition}
A leveled shaped pseudo 3-manifold $(Y,l_Y)$ is said to be obtained by a leveled shaped 3-2 Pachner move if it is obtained from a leveled shaped pseudo 3-manifold $(X,l_X)$ such that the following condition is satisfied.
The condition is that there exists $e\in\Delta_1(X)$ such that $Y=X_e$, and 
\begin{equation*}
l_Y=l_X+\frac{1}{12}\sum_{a\in(\phi^{3,1})^{-1}(e)}\sum_{b\in\Delta_3^1(X)}\epsilon_{p(a),p(b)}\alpha_X(b).
\end{equation*}
This action is defined such that the level-dependent invariant is invariant to leveled shaped 3-2 Pachner moves.
\begin{definition}
A (leveled) shaped pseudo 3-manifold $X$ is called a Pachner refinement of a (leveled) shaped pseudo 3-manifold $Y$ if it satisfies the following condition.
The condition is that there exists a finite sequence of (leveled) shaped pseudo 3-manifolds 
\begin{equation*}
X=X_1,X_2,\ldots,X_n=Y
\end{equation*}
such that for any $i\in\{1,\ldots,n-1\}$, $X_{i+1}$ is obtained from $X_i$ by a (leveled) shaped 3-2 Pachner move.

Two (leveled) shaped pseudo 3-manifolds, $X$ and $Y$, are said to be equivalent if there exist Pachner refinements $X^\prime$, $Y^\prime$ of $X$ and $Y$ such that $X^\prime$ and $Y^\prime$ are gauge equivalent. 
\end{definition}

\begin{definition}
An oriented triangulated pseudo 3-manifold $X$ is said to 
be admissible if  
\begin{equation*}
    S_{r}(X)\neq\emptyset
\end{equation*}
and
\begin{equation*}
H_2(X-\Delta_0(X),\mathbb{Z})=0
\end{equation*} 
\end{definition}

The Andersen-Kashaev TQFT is not defined on all leveled shaped pseudo 3-manifolds, and the invariant of the TQFT is guaranteed to be well-defined only on admissible (leveled) shaped psudo 3-manifolds.
 This is because the positivity of the dihedral angles is necessary to regularize the distribution corresponding to the tetrahedra described below, and $H_2(X-\Delta_0(X),\mathbb{Z})=0$ guarantees that the product of the distributions for all tetrahedra constituting $X$ is determined, and we can push forward the product necessary to define the TQFT invariant.
\begin{definition}
Two admissible (leveled) shaped pseudo 3-manifolds $X$ and $Y$ are said 
to be admissibly equivalent if there exists a gauge equivalence $h^\prime: X^\prime\rightarrow Y^\prime$ for some Pachner refinements $X^\prime$, $Y^\prime$ of $X$ and $Y$ such that the following conditions are satisfied.
\begin{equation*}
\Delta_1(X^\prime)=\Delta_1(X)\cup D_X, \quad \Delta_1(Y^\prime)=\Delta_1(Y)\cup D_Y
\end{equation*}
and
\begin{equation*}
h^\prime(S_r(X^\prime)\cap\widetilde{\Omega}_{X^\prime,r}^{-1}(P_{D_X}))\cap\widetilde{\Omega}_{Y^\prime,r}^{-1}(P_{D_Y})\neq\emptyset,    
\end{equation*}
where
\begin{equation*}
P_{D_X}\coloneqq \left\{\omega:\Delta_1(X^\prime)\rightarrow\mathbb{R}\mid\forall e\in D_X, \omega(e)=2\pi \right\},
\end{equation*}
and $P_{D_Y}$ is defined similarly.
\end{definition}
\subsection{Categroid}
\begin{definition}
A categroid $\mathcal{C}$ consists of a family of objects $\mathrm{Obj} (\mathcal{C})$ and a set of morphisms $\mathrm{Hom}_{\mathcal{C}}(A,B)$ for any two objects  $A$, $B$ in $\mathrm{Obj} (\mathcal{C})$ such that the following two properties are satisfied. 

1. For any three objects $A$, $B$, $C\in\mathrm{Obj}(\mathcal{C})$, 
there exists a subset called composable morphisms 
\begin{equation*}
K_{A,B,C}^{\mathcal{C}}\subset \mathrm{Hom}_{\mathcal{C}}(A,B)\times \mathrm{Hom}_{\mathcal{C}}(B,C) 
\end{equation*}
and a composite map such that the composition of morphisms is associative,
\begin{equation*}
\circ: K_{A,B,C}^{\mathcal{C}}\rightarrow \mathrm{Hom}_{\mathcal{C}}(A,C).
\end{equation*}

2. For any object $A\in\mathrm{Obj}(\mathcal{C})$, there exists an identity morphism $\mathrm{id}_A\in\mathrm{Hom}_{\mathcal{C}}(A,A)$ which can be composite with any morphism $f\in\mathrm{Hom}_{\mathcal{C}}(A,B)$, $g\in\mathrm{Hom}_{\mathcal{C}}(B,A)$ such that the following equation satisfies.
\begin{equation*}
\mathrm{id}_A\circ f = f \quad \mathrm{and}\quad g\circ\mathrm{id}_A =g. 
\end{equation*}
\end{definition}
\subsection{A categroid of admissible leveled shaped pseudo 3-manifolds}
The equivalence class of leveled shaped pseudo 3-manifolds is a morphism of the cobordism category $\mathcal{B}$, where the object is a triangulated surface, and the composition of morphisms is gluing the corresponding parts of the boundary by the CW-homeomorphism that preserves the orientation of the edges and reverses the orientation of the faces, which involves an obvious composition of dihedral angles and a sum of levels. Depending on how the boundary is partitioned, even the same leveled shaped pseudo 3-manifolds can be interpreted as different morphisms of $\mathcal{B}$. However, there is a standard way of partitioning the boundary as described below.

For a tetrahedron $T=[v_0,v_1,v_2,v_3]$ in $\mathbb{R}^3$ with ordered vertices $v_0,v_1,v_2,v_3$, its sign is defined as
\begin{equation*}
{\rm sign}(T)={\rm sign}(\det(v_1-v_0,v_2-v_0,v_3-v_0)),
\end{equation*}
and the sign of the face is defined as 
\begin{equation*}
{\rm sign}(\partial_i T)=(-1)^i{\rm sign}(T),\quad i\in{0,\ldots,3}.
\end{equation*}

For a pseudo 3-manifold $X$, the sign of faces of tetrahedra constituting $X$ induces the following sign function on the faces of the boundary of $X$. 
\begin{equation*}
{\rm sign}_X:\Delta_2(\partial X)\rightarrow \left\{\pm 1\right\}
\end{equation*}

This sign function decomposes the boundary of $X$ into two components $\partial_{-}X$, $\partial_{+}X$ consisting of an equal number of triangles.

In the following (the equivalence class of) a leveled shaped pseudo 3-manifold $X$ is considered as a $\mathcal{B}$-morphism between the objects $\partial_{-}X$ and $\partial_{+}X$. In other words
\begin{equation*}
    X\in{\rm Hom }_{\mathcal{B}}(\partial_{-}X,\partial_{+}X). 
\end{equation*}

The TQFT of Andersen and Kashaev is not defined on the entire $\mathcal{B}$-category, but on the sub-categroid formed by admissible equivalence classes of admissible morphisms.
\begin{definition}
The categroid $\mathcal{B}_{a}$ of admissible leveled shaped pseudo 3-manifolds is a sub-categroid of the category formed by leveled shaped pseudo 3-manifolds, whose morphisms consist of admissible equivalence classes of admissible leveled shaped pseudo 3-manifolds.
The composition map in this sub-categroid is induced from the category $\mathcal{B}$, and
the composable morphisms are
\begin{equation*}
\begin{split}
 K_{A,B,C}^{\mathcal{B}_a}&=
 \{(X_1,X_2)\in\mathrm{Hom}_{\mathcal{B}_a}(A,B)\times \mathrm{Hom}_{\mathcal{B}_a}(B,C)\mid S_r(X_1\circ X_2)\neq\emptyset,\\ &\quad H_2(X_1\circ X_2 -\Delta_0(X_1\circ X_2),\mathbb{Z})=0 \}.     
\end{split}
\end{equation*}
\end{definition}
\subsection{The TQFT functor}
Andersen, Kashaev's TQFT constructs a family of functors $\{F_{\hbar}\}_{\hbar\in\mathbb{R}}$  from a cobordism categroid $\mathcal{B}_a$ consisting of admissible leveled shaped pseudo 3-manifolds to a categroid $\mathcal{D}$ of tempered distributions as described below.

The space $\mathcal{S}^\prime (\mathbb{R}^n)$ of (complex) tempered distributions is the space of continuous linear functionals on the (complex) Schwartz space $\mathcal{S}(\mathbb{R}^n)$.  
By the Schwartz representation theorem, any tempered distribution can be represented by differentiating a continuous function with polynomial growth a finite number of times. Therefore, we can consider a tempered distribution as a function of $\mathbb{R}^n$. For any $\phi\in\mathcal{S}^\prime (\mathbb{R}^n),x\in\mathbb{R}^n$, we denote $\phi(x)\equiv\langle x|\phi\rangle$. This notation should be considered in the usual sense of distributions. For example,
\begin{equation*}
\phi(f)=\int_{\mathbb{R}^n}\phi(x)f(x)dx
\end{equation*}

This formula shows that $\mathcal{S}(\mathbb{R}^n)\subset\mathcal{S}^\prime(\mathbb{R}^n)$.
\begin{definition}
The categroid $\mathcal{D}$ has a finite set as its object, and for two finite sets $n$ and $m$, the set of morphisms from $n$ to $m$ is
 \begin{equation*}
 {\rm Hom}_{\mathcal{D}}(n,m)=\mathcal{S}^\prime(\mathbb{R}^{n\sqcup m}).   
 \end{equation*}
\end{definition}
Let $\mathcal{L}(S(\mathbb{R}^n),\mathcal{S}^\prime(\mathbb{R}^m))$ denote the space of continuous linear maps from $\mathcal{S}(\mathbb{R}^n)$ to $\mathcal{S}^\prime(\mathbb{R}^m)$.

For any $\phi\in L(\mathcal{S}(\mathbb{R}^n)$, $\mathcal{S}^\prime(\mathbb{R}^m)), f\in\mathcal{S}(\mathbb{R}^n)$, $g\in\mathcal{S}(\mathbb{R}^m)$, there exists an isomorphism 
\begin{equation}\label{Schwartz-isom}
\widetilde{\cdot}:\mathcal{L}(\mathcal{S}(\mathbb{R}^n), \mathcal{S}^\prime(\mathbb{R}^m))\rightarrow\mathcal{S}^\prime(\mathbb{R}^{n\sqcup m})
\end{equation}
defined by 
\begin{equation*}
\phi(f)(g)=\widetilde{\phi}(f\otimes g).
\end{equation*}

The composition partially defined in this categroid is defined as follows. Let $n,m,l$ be three finite sets and $A\in{\rm Hom}_{\mathcal{D}}(n,m)$, $B\in{\rm Hom}_{\mathcal{D}}(m,l)$.
By theorem 6.1.2 in \cite{MR1996773}, there exist pullback maps
\begin{equation*}
\pi^{\ast}_{n,m}:\mathcal{S}^\prime(\mathbb{R}^{n\sqcup m})\rightarrow\mathcal{S}^{\prime}(\mathbb{R}^{n\sqcup m\sqcup l}) ,\quad
\pi^{\ast}_{m,l}:\mathcal{S}^\prime(\mathbb{R}^{m\sqcup l})\rightarrow\mathcal{S}^{\prime}(\mathbb{R}^{n\sqcup m\sqcup l}) .
\end{equation*}
By theorem IX.45 in \cite{MR0493420}, 
\begin{equation*}
\pi^\ast_{n,m}(A)\pi^\ast_{m,l}(B)\in\mathcal{S}^\prime(\mathbb{R}^{n\sqcup m \sqcup l})
\end{equation*}
is well-defined if $\pi_{n,m}^\ast(A)$ and $\pi_{m,l}^\ast(B)$ satisfy the following transversality condition.
\begin{equation}\label{wave_front_set}
(WF(\pi_{n,m}^\ast(A))\oplus WF(\pi_{m,l}^\ast(B)))\cap Z_{n\sqcup m\sqcup l}=\emptyset,
\end{equation}
where $Z_{n\sqcup m\sqcup l}$ is the zero section of $T^\ast(\mathbb{R}^{n\sqcup m\sqcup l})$. Furthermore, if $\pi^\ast_{n,m}(A)\pi^\ast_{m,l}(B)$ extends continuously to $\mathcal{S}^\prime(\mathbb{R}^{n\sqcup m\sqcup l })_{m}$, we obtain the following well-defined element.
\begin{equation*}
(\pi_{n,l})_\ast(\pi^\ast_{n,m}(A)\pi^\ast_{m,l}(B))\in\mathcal{S}^\prime(\mathbb{R}^{n \sqcup l}),
\end{equation*}
where $\mathcal{S}^\prime(\mathbb{R}^n)_m$ is defined as follows.
\begin{definition}
$\mathcal{S}(\mathbb{R}^n)_m$ is the set of all $\phi\in C^\infty(\mathbb{R}^n)$ such that  
\begin{equation*}
\sup_{x\in\mathbb{R}^n}|x^\beta\partial^\alpha(\phi)(x)|<\infty
\end{equation*} 
for all multi indices $\alpha,\beta$ such that for $n-m<i\le n$ if $\alpha_i=0$, then $\beta_i=0$. 
Also, $\mathcal{S}^\prime(\mathbb{R}^n)_m$ is a continuous dual of $\mathcal{S}(\mathbb{R}^n)_m$ regarding these seminorms.  
\end{definition}
\begin{definition}
If $A\in{\rm Hom}_D(n,m)$ and $B\in{\rm Hom}_D(m,l)$ satisfy the 
equation (\ref{wave_front_set}) and $\pi^\ast_{n,m}(A)\pi^\ast_{m,l}(B)$ continuously extends to $\mathcal{S}^\prime (\mathbb{R}^{n\sqcup m\sqcup l})_m$, we define 
\begin{equation*}
AB=(\pi_{n,l})_\ast(\pi_{n,m}^\ast(A)\pi_{m,l}^\ast(B))\in{\rm Hom}_{\mathcal{D}}(n,l).
\end{equation*}
\end{definition}
Therefore, the composable morphisms in the categroid $\mathcal{D}$ are 
\begin{equation*}
\begin{split}
K_{n,m,l}^{\mathcal{D}}=&\{(A,B)\in\mathcal{S}^\prime(\mathbb{R}^{n\sqcup m})\times\mathcal{S}^\prime(\mathbb{R}^{m\sqcup l})\mid (\mathrm{WF}(\pi_{n,m}^\ast(A))\oplus WF(\pi_{m,l}^\ast(B)))\cap Z_{n\sqcup m\sqcup l} \\
& =\emptyset, \pi^\ast_{n,m}(A)\pi^\ast_{m,l}(B)\in\mathcal{S}^\prime(\mathbb{R}^{n\sqcup m\sqcup l })_{m}\}.
\end{split}
\end{equation*}

For any $A\in\mathcal{L}(\mathcal{S}(\mathbb{R}^n),\mathcal{S}^\prime(\mathbb{R}^m))$, there 
exists an unique adjoint $A^\ast\in\mathcal{L}(\mathcal{S}(\mathbb{R}^m),\mathcal{S}^\prime(\mathbb{R}^n))$ defined as follows.
For any $f\in\mathcal{S}(\mathbb{R}^m), g\in\mathcal{S}(\mathbb{R}^n)$, 
\begin{equation*}
A^\ast(f)(g)=\overline{A(\overline{g})(\overline{f})}.
\end{equation*}
\begin{definition}
The functor $F: \mathcal{B}_a\rightarrow D$ is called 
a $\ast-$functor if 
\begin{equation*}
F(X^\ast)=F(X)^\ast, 
\end{equation*}
where $X^\ast$ is $X$ with an opposite orientation to $X$, and  $F(X)^\ast$ is the adjoint of $F(X)$.
\end{definition}
Faddeev's quantum dilogarithm plays an important role in constructing 
the functor of Andersen-Kashaev's TQFT.
\begin{definition}
Faddeev's quantum dilogarithm \cite{MR1345554} is a function whose arguments are two complex variables $z$ and $\mathsf{b}$ defined for $| \Im z |<\frac{1}{2}| \mathsf{b}+\mathsf{b}^{-1}|$, 
\begin{equation*}
\Phi_{\mathsf{b}}(z)\coloneqq\exp{\left(\int_{C}\frac{e^{-2izw}dw}{4\sinh{w{\mathsf{b}}}\sinh{(\frac{w}{\mathsf{b}})}w}\right)},
\end{equation*}
where $C$ runs along the real axis and deviates into the upper half-plane in the neighborhood of the origin, and it extends to the meromorphic function on $z\in\mathbb{C}$ by the functional equation
\begin{equation*}
\Phi_{\mathsf{b}}(z-i{\mathsf{b}}^{\pm 1}/2)=(1+e^{2\pi {\mathsf{b}}^{\pm 1}z})\Phi_{\mathsf{b}}(z+i{\mathsf{b}}^{\pm 1}/2).
\end{equation*}
\end{definition}
$\Phi_{\mathsf{b}}(z)$ depends on ${\mathsf{b}}$ only through $\hbar$
defined by
\begin{equation*}
\hbar\coloneqq({\mathsf{b}+\mathsf{b}^{-1}})^{-2}.
\end{equation*}
$\Phi_{\mathsf{b}}$ has the following properties.
\begin{prop}[ \cite{MR3227503}, Appendix A ]\label{property-of-Phi_b}
(1) 
For any ${\mathsf{b}}\in\mathbb{R}_{>0}$, for any $z\in\mathbb{R}+i\left(\frac{-1}{2\sqrt{\hbar}},\frac{1}{2\sqrt{\hbar}}\right)$, 
\begin{equation*}
\Phi_{\mathsf{b}}(z)\Phi_{\mathsf{b}}(-z)=e^{i\frac{\pi}{12}(\mathsf{b}^2+\mathsf{b}^{-2})}e^{i\pi z^2}.  \end{equation*}
(2) 
For any $\mathsf{b}\in\mathbb{R}_{>0}$, for any $z\in\mathbb{R}+i\left(\frac{-1}{2\sqrt{\hbar}},\frac{1}{2\sqrt{\hbar}}\right)$,
\begin{equation*}
\overline{\Phi_{\mathsf{b}}(z)}=\frac{1}{\Phi_{\mathsf{b}}(\overline{z})}.
\end{equation*}
(3) (Behavior at infinity) For any $\mathsf{b}\in\mathbb{R}_{>0}$,
\begin{equation*}
\begin{split}
\Phi_{\mathsf{b}}(z)&\underset{\Re(z)\rightarrow-\infty}{\sim}1,\\
\Phi_{\mathsf{b}}(z)&\underset{\Re(z)\rightarrow\infty}{\sim}e^{i\frac{\pi}{12}(\mathsf{b}^2+\mathsf{b}^{-2})}e^{i\pi z^2}.
\end{split}
\end{equation*}

In particular, for any $\mathsf{b}\in\mathbb{R}_{>0}$, for any 
$d\in\left(\frac{-1}{2\sqrt{\hbar}},\frac{1}{2\sqrt{\hbar}}\right)$,
\begin{equation*}
\begin{split}
\left|\Phi_{\mathsf{b}}(x+id)\right|&\underset{\mathbb{R}\ni x\rightarrow-\infty}{\sim}1,\\
\left|\Phi_{\mathsf{b}}(x+id)\right|&\underset{\mathbb
{R}\ni x\rightarrow +\infty}{\sim}e^{-2\pi xd}.
\end{split}
\end{equation*}
\end{prop}
\begin{theorem}[Andersen, Kashaev]
For any $\hbar\in\mathbb{R}_{+}$, there exists a unique $\ast-$ functor satisfying the following conditions. 
For any $A\in{\rm Ob}(\mathcal{B}_a)$, $F_\hbar(A)=\Delta_2(A)$, and
for any admissible leveled shaped pseudo 3-manifold $(X, l_X)$, the corresponding morphism in $\mathcal{D}$ is
\begin{equation*}
F_\hbar(X,l_X)=Z_\hbar(X)e^{i\pi \frac{l_X}{4\hbar}}\in\mathcal{S}^\prime(\mathbb{R}^{\Delta_2(\partial X)}),
\end{equation*}
where $Z_\hbar(X)$ is defined such that for a tetrahedron $T$ with ${\rm sign}(T)=1$
\begin{equation*}
Z_\hbar(T)(x)=\delta(x_0+x_2-x_1)\frac{\exp\left(2\pi i(x_3-x_2)(x_0+\frac{\alpha_3}{2i\sqrt{\hbar}}
)+\pi i\frac{\varphi_T}{4\hbar}\right)}{\Phi_{\mathsf{b}}\left(x_3-x_2+\frac{1-\alpha_1}{2i\sqrt{\hbar}}\right)}.
\end{equation*}
Here $\delta(t)$ is Dirac's delta function and
\begin{equation*}
\varphi_T\coloneqq\alpha_1\alpha_3+\frac{\alpha_1-\alpha_3}{3}-\frac{2\hbar+1}{6},\quad \alpha_i\coloneqq\frac{1}{\pi }\alpha_T(\partial_0\partial_i T),\quad i\in\{1,2,3\},
\end{equation*}
\begin{equation*}
x_i\coloneqq x(\partial_i(T)),\quad x:\Delta_2(\partial T)\to\mathbb{R}.
\end{equation*}
\end{theorem}
$Z_\hbar(X)$ and $Z_\hbar(T)$ are called partition functions and described in detail in Subsection \ref{partition-func-of-tetrahedron}.
For an admissible pseudo 3-manifold $X$, Andersen-Kashaev's TQFT functor gives the following well-defined function, 
\begin{equation*}
F_\hbar : LS_r(X)\rightarrow \mathcal{S}^\prime (\mathbb{R}^{\partial X}).
\end{equation*}

If $\partial X=\emptyset$, then $\mathcal{S}^\prime(\mathbb{R}^{\partial X})=\mathbb{C}$, and we obtain a complex-valued function on $LS_r(X)$.

In particular, the value of the functor $F_\hbar$ on any fully balanced admissible leveled shaped 3-manifold is a complex number, and it is a topological invariant in the sense that if two fully balanced admissible leveled shaped 3-manifolds are admissibly equivalent, $F_\hbar$ assigns the same complex number to them. 
\subsection{Tetrahedral operators of quantum Teichm\"{u}ller theory}
The operators $\mathsf{q}_i$ and $\mathsf{p}_i$ which act on $\mathcal{S}(\mathbb{R}^n)$ are defined as follows. For any $f\in\mathcal{S}(\mathbb{R}^n)$,
\begin{equation*}
\mathsf{q}_i(f)(t)=t_if(t),\quad \mathsf{p}_i(f)(t)=\frac{1}{2\pi i}\frac{\partial}{\partial t_i}(f)(t),\quad \forall t\in\mathbb{R}^n.
\end{equation*}
These operators are known to extend continuously to $\mathcal{S}^\prime(\mathbb{R}^n)$ and satisfy the following Heisenberg's commutation relations,
\begin{equation*}
[\mathsf{p}_i,\mathsf{p}_j]=[\mathsf{q}_i,\mathsf{q}_j]=0,\quad [\mathsf{p}_i,\mathsf{q}_j]=\frac{1}{2\pi i}\delta_{i,j}.
\end{equation*}

Fix ${\rm{b}}\in\mathbb{C}$ such that ${\Re}({\mathsf{b}})\neq 0$. 
By the spectral theorem, the following operators
\begin{equation*}
\mathsf{u}_i=e^{2\pi\mathsf{bq}_i},\quad\mathsf{v}_i=e^{2\pi\mathsf{bp}_i}
\end{equation*}
can be defined.

The commutation relations between $\mathsf{u}_i$ and $\mathsf{v}_j$ are
\begin{equation*}
[\mathsf{u}_i,\mathsf{u}_j]=[\mathsf{v}_i,\mathsf{v}_j]=0,\quad \mathsf{u}_i\mathsf{v}_j=e^{i2\pi\mathsf{b}^2\delta_{i,j}}\mathsf{v}_j\mathsf{u}_i.
\end{equation*}

According to \cite{MR1607296}, consider the following operations on $\vec{\mathsf{w}}_i=(\mathsf{u}_i,\mathsf{v}_i), i=1,2$.
\begin{equation*}
\vec{\mathsf{w}}_1\cdot\vec{\mathsf{w}}_2\coloneqq(\mathsf{u}_1\mathsf{u}_2, \mathsf{u}_1\mathsf{v}_2+\mathsf{v}_1)
\end{equation*}
\begin{equation*}
\vec{\mathsf{w}}_1\ast\vec{\mathsf{w}}_2\coloneqq(\mathsf{v}_1\mathsf{u}_2(\mathsf{u}_1\mathsf{v}_2+\mathsf{v}_1)^{-1},\mathsf{v}_2(\mathsf{u}_1\mathsf{v}_2+\mathsf{v}_1)^{-1})
\end{equation*}
\begin{theorem}[\cite{MR1607296}]
Let $\psi(z)$ be a solution of the following functional equation
\begin{equation}\label{functional-equation}
\psi\left(z+\frac{i{\mathsf{b}}}{2}\right)=\psi\left(z-\frac{i\mathsf{b}}{2}\right)(1+e^{2\pi\mathsf{b}z}),\quad z\in\mathbb{C}.
\end{equation}
Then the operator 
\begin{equation}\label{T-operator}
\mathsf{T}=\mathsf{T}_{12}\coloneqq e^{2\pi i\mathsf{p}_1\mathsf{q}_2}\psi(\mathsf{q}_1+\mathsf{p}_2-\mathsf{q}_2)=\psi(\mathsf{q}_1-\mathsf{p}_1+\mathsf{p}_2)e^{2\pi i\mathsf{p}_1\mathsf{q}_2}
\end{equation}
is an element of $\mathcal{L}(\mathcal{S}(\mathbb{R}^4),\mathcal{S}(\mathbb{R}^4))$, and satisfies the following equations,
\begin{equation*}
\vec{\mathsf{w}}_1\cdot\vec{\mathsf{w}}_2\mathsf{T}=\mathsf{T}\vec{\mathsf{w}}_1,\quad
\vec{\mathsf{w}}_1\ast\vec{\mathsf{w}}_2\mathsf{T}=\mathsf{T}\vec{\mathsf{w}}_2.
\end{equation*}
\end{theorem}
One solution of the equation (\ref{functional-equation}) is given by Faddeev's quantum dilogarithm as follows.
\begin{equation*}
\psi(z)=\overline{\Phi}_\mathsf{b}(z)\coloneqq\frac{1}{\Phi_{\mathsf{b}}(z)}.
\end{equation*}
In the following sections, $\mathsf{b}$ is taken so that 
\begin{equation*}
\hbar\coloneqq\left(\mathsf{b}+\mathsf{b}^{-1}\right)^{-2}\in\mathbb{R}_+ .
\end{equation*}

\subsection{Charged tetrahedral operators}
For any positive real numbers $a$ and $c$ such that $b\coloneqq\frac{1}{2}-a-c$ is also positive, the charged $\mathsf{T}-$ operators are defined as 
\begin{equation*}
\mathsf{T}(a,c)\coloneqq e^{-\pi ic_{\mathsf{b}}^2\frac{4(a-c)+1}{6}}e^{4\pi ic_{\mathsf{b}}(c\mathsf{q}_2-a\mathsf{q}_1)}\mathsf{T}e^{-4\pi ic_{\mathsf{b}}(a\mathsf{p}_2+c\mathsf{q}_2)}
\end{equation*}
\begin{equation*}
\overline{\mathsf{T}}(a,c)\coloneqq e^{\pi ic_{\mathsf{b}}^2\frac{4(a-c)+1}{6}}e^{-4\pi ic_{\mathsf{b}}(a\mathsf{p}_2+c\mathsf{q}_2)}\overline{\mathsf{T}}e^{4\pi ic_{\mathsf{b}}(c\mathsf{q}_2-a\mathsf{q}_1)},
\end{equation*}
where $\overline{\mathsf{T}}\coloneqq\mathsf{T}^{-1}$ and  
\begin{equation*}
c_{\mathsf{b}}\coloneqq i\frac{(\mathsf{b}+\mathsf{b}^{-1})}{2}.
\end{equation*}
Then $\mathsf{T}(a,c):\mathcal{S}(\mathbb{R}^2)\rightarrow\mathcal{S}(\mathbb{R}^2)$ and $\overline{\mathsf{T}}(a,c):\mathcal{S}({\mathbb{R}^2})\rightarrow\mathcal{S}(\mathbb{R}^2)$.
Substituting the equation (\ref{T-operator}), we obtain
\begin{equation*}
\mathsf{T}(a,c)=e^{2\pi i\mathsf{p}_1\mathsf{q}_2}\psi_{a,c}(\mathsf{q}_1-\mathsf{q}_2+\mathsf{p}_2),
\end{equation*}
where 
\begin{equation*}
\psi_{a,c}(x)\coloneqq\psi(x-2c_{\mathsf{b}}(a+c))e^{-4\pi ic_{\mathsf{b}}a(x-c_{\mathsf{b}}(a+c))}e^{-\pi ic_{\mathsf{b}}^2\frac{4(a-c)+1}{6}}.
\end{equation*}
The following formula holds for $\mathsf{T}(a,c)\in\mathcal{S}^\prime(\mathbb{R}^4)$.
\begin{equation*}
\langle x_0,x_2|\mathsf{T}(a,c)|x_1,x_3\rangle =\delta(x_0+x_2-x_1)\widetilde{\psi}_{a,c}^\prime(x_3-x_2)e^{2\pi ix_0(x_3-x_2)},
\end{equation*}
where
\begin{equation*}
\widetilde{\psi}_{a,c}^\prime(x)\coloneqq e^{-\pi ix^2}\widetilde{\psi}_{a,c}(x),\quad \widetilde{\psi}_{a,c}(x)\coloneqq \int_{\mathbb{R}}\psi_{a,c}(y)e^{-2\pi ixy}dy.
\end{equation*}

The conditions imposed on $a$ and $c$ guarantee that the Fourier integral is absolutely convergent.

The Fourier transformation formula for Faddeev's quantum dilogarithm leads to the following identity.
\begin{equation*}
\widetilde{\psi}_{a,c}^\prime(x)=e^{-\frac{\pi i}{12}}\psi_{c,b}(x).
\end{equation*}
Concerning complex conjugation, 
\begin{equation*}
\overline{\psi_{a,c}(x)}=e^{-\frac{\pi i}{6}}e^{\pi ix^2}\psi_{c,a}(-x)=e^{-\frac{\pi i}{12}}\widetilde{\psi}_{b,c}(-x).
\end{equation*}
In combination with these, 
\begin{equation*}
\overline{\widetilde{\psi}^\prime_{a,c}(x)}=e^{\frac{\pi i}{12}}\overline{\psi_{c,b}(x)}=e^{-\frac{\pi i}{12}}e^{\pi ix^2}\psi_{b,c}(-x).
\end{equation*}

From the above, the following formula for $\overline{\mathsf{T}}(a,c)$ is obtained, 
\begin{eqnarray*}
\langle x,y|\overline{\mathsf{T}}(a,c)|u,v\rangle&=&\overline{\langle u,v|\mathsf{T}(a,c)|x,y\rangle}\\
&=&\delta(u+v-x)\overline{\widetilde{\psi}_{a,c}^\prime(y-v)}e^{-2\pi iu(y-v)}\\
&=&\delta(u+v-x)\psi_{b,c}(v-y)e^{-\frac{\pi i}{12}}e^{\pi i(v-y)^2}e^{-2\pi iu(y-v)}.
\end{eqnarray*}
\subsection{The partition function of a shaped tetrahedron}
\label{partition-func-of-tetrahedron}
Let $T$ be a shaped tetrahedron with ordered vertices $v_i\ ( i=0,1,2,3 )$ in $\mathbb{R}^3$.

We define the partition function 
$Z_\hbar(T)\in\mathcal{S}^\prime(\mathbb{R}^{\Delta_2(\partial T)})$
by the following formula.
\begin{equation}
\langle x|Z_\hbar(T)\rangle=\left\{
  \begin{aligned}
  &\langle x_0,x_2\lvert \mathsf{T}\left(c\left(v_0v_1\right),c\left(v_0v_3\right)\right)\rvert x_1,x_3\rangle\quad\quad {\rm{if}}\quad {\rm sign}(T)=1\\
  &\langle x_1,x_3\lvert \overline{\mathsf{T}}\left(c\left(v_0v_1\right),c\left(v_0v_3\right)\right)\rvert x_0,x_2\rangle\quad\quad {\rm{if}}\quad {\rm sign}(T)=-1
  \end{aligned}
\right., 
\end{equation}
where
\[
x_i=x(\partial _i T),\quad i\in\{0,1,2,3\}
\]
\[
c\coloneqq\frac{1}{2\pi}\alpha_T:\Delta_1(T)\rightarrow\mathbb{R}_+.
\]
\subsection{One vertex H-triangulation}
Let $(M,K)$ be a pair of oriented closed 3-manifold $M$ and a knot $K$ in $M$. A one vertex H-triangulation of $(M,K)$ is a tetrahedral decomposition of $M$ with one vertex and which has one edge representing the knot $K$. If $M=S^3$, we can obtain an example of one vertex H-triangulation from the projective diagram of $K$.  

\subsection{Notations and conditions}
We denote an oriented triangulated pseudo 3-manifold $X$ which consists of $N$ tetrahedra with ordered vertices $0,1,2,3$ by $X=(T_1, \ldots, T_N, \sim)$, where $\sim$ means the equivalence relation generated by gluing between faces which preserve the orientation of edges.

Let $\mathscr{S}_X$ denote the set corresponding to the shape structure $S(X)$ on $X$ defined by 
\begin{equation*}
\begin{split}
\mathscr{S}_X\coloneqq & \left\{ \alpha_X=(2\pi a_1,2\pi b_1,2\pi c_1,\ldots,2\pi a_N,2\pi b_N,2\pi c_N)\in(0,\pi)^{3N}\mid\right.\\
& \quad \left.\forall k\in\{1,\ldots,N\},a_k+b_k+c_k=\frac{1}{2}\right\},
\end{split}
\end{equation*}
 where $2\pi a_k$ (respectively, $2\pi b_k,2\pi c_k$) denotes the value of the dihedral angle on $\overrightarrow{01}$ (respectively, $\overrightarrow{02},\ \overrightarrow{03}$) and its opposite edge of a tetrahedron $T_k$. $\mathscr{S}_X$ is also called a shape structure.

Let $\overline{\mathscr{S}_X}$ denote the closure of $\mathscr{S}_X$ such that $a_k,b_k$, and $c_k$ have values in $[0,\frac{1}{2}]$, and the weight function is defined in an extended way. 

A fully balanced shape structure on $X$ is called an angle structure and is defined by $\mathscr{A}_X\coloneqq\{\alpha_X\in \mathscr{S}_X\mid
\forall e\in \Delta_1(X), \omega_{X}(e)=2\pi\}$ and the extended angle structure on $X$ is defined by
$\overline{\mathscr{A}_X}\coloneqq\{\alpha_X\in \overline{\mathscr{S}_X}\mid
\forall e\in \Delta_1(X), \omega_{X}(e)=2\pi\}$.

A shaped pseudo 3-manifold can be represented as a pair $(X,\alpha_X)$ of an underlying oriented triangulated pseudo 3-manifold $X$ and a shape structure $\alpha_X\in\mathscr{S}_X$. $\omega_{X,\alpha_X}$ denotes the weight function on $(X,\alpha_X)$. 

This paper discusses under the condition $\mathsf{b}\in\mathbb{R}_{>0}$.

\subsection{The conjecture analogous to the volume conjecture}
The following conjecture has been proposed.
\begin{conjecture}[Andersen, Kashaev]\label{Andersen-Kashaev}
Let $M$ be an oriented compact closed 3-manifold. For any hyperbolic knot $K$ in $M$, there exists a smooth function $J_{M,K}(\hbar,x)$ on $\mathbb{R}_{> 0}\times\mathbb{R}$ satisfying the following properties.

(1) For any fully balanced shaped ideal triangulation $X$ of the complementary space of $K$ in $M$, there exists a real linear combination $\lambda$ of gauge-invariant dihedral angles and a real second-order polynomial $\phi$ of (not necessarily gauge-invariant) dihedral angles such that 
\[
Z_{\hbar}(X)=e^{i\frac{\phi}{\hbar}}\int_{\mathbb{R}}J_{M,K}(\hbar,x)e^{-\frac{x\lambda}{\sqrt{\hbar}}}dx.
\]

(2) For any one vertex shaped H-triangulation $Y$ of the pair $(M,K)$, there exists a real second-order polynomial $\varphi$ of dihedral angles such that
\[
\lim_{\omega_{Y}\rightarrow\tau}\Phi_{\rm{b}}\left(\frac{\pi-\omega_{Y}(K)}{2\pi i\sqrt{\hbar}}\right)Z_{\hbar}(Y)=e^{i\frac{\varphi}{\hbar}-i\frac{\pi}{12}}J_{M,K}(\hbar,0),
\]
where $\tau:\Delta_1(Y)\rightarrow\mathbb{R}$ takes the value $0$ on the edge representing the knot $K$ and $2\pi$ on all other edges. 

(3) The hyperbolic volume of the complementary space of $K$ in $M$ is obtained by the following limit,
\[
\lim_{\hbar\rightarrow 0}2\pi\hbar\log|J_{M,K}(\hbar,0)|=-{\rm Vol}(M\backslash K).
\]
\end{conjecture}

Concerning this conjecture, the $(S^3,4_1)$ and $(S^3,5_2)$ cases are proven for a certain triangulation \cite{MR3227503}. 
A similar reformulated theorem is also proven for all twist knots under the condition $\mathsf{b}>0$ and with the specific triangulation \cite{MR3945172}.
In this paper, we aim to prove this conjecture in the case of the pair of $S^3$ and $7_3$ knot, which has yet to proven in previous studies, using a similar method to the proof of \cite{MR3945172}. 
Results similar to (1) and (2) regarding the form of the partition function are confirmed for a specific tetrahedral decomposition and its equivalence class by the calculation of partition functions in Section \ref{ideal-triangulation} and \ref{one-vertex_Htriangulation}. 
As for (3), we prove it by rigorously evaluating the integral using a similar method to the proof of \cite{MR3945172} for twist knots, and we describe the proof in Section \ref{volume-conjecture-strict-proof}.
In other words, the main theorem of this paper is as follows.
\begin{theorem}\label{uemura}
For the hyperbolic knot $7_3$ in $S^3$, there exists a function $J_{S^3,7_3}(\hbar,x)$ on $\mathbb{R}_{> 0}\times\mathbb{C}$ satisfying the following properties. 

(1) For an ideal tetrahedral decomposition $X$ of the complementary space of $7_3$ knot in $S^3$ and for any angle structure $\alpha_X\in\mathscr{A}_X$ on $X$, there exist gauge invariant real linear combinations of dihedral angles $\lambda(\alpha_X)$, $\mu(\alpha_X)$, and a (not necessarily gauge invariant) quadratic polynomial of dihedral angles $\phi(\alpha_X)$ such that  
\[
Z_{\hbar}(X,\alpha_X)=e^{i\frac{\phi(\alpha_X)}{\hbar}}\int_{\mathbb{R}+\frac{i\mu(\alpha_X)}{\sqrt{\hbar}}}J_{S^3,7_3}(\hbar,x)e^{-\frac{x\lambda(\alpha_X)}{\sqrt{\hbar}}}dx.
\]

(2) For a one vertex H-triangulation $Y$ of $(S^3,7_3)$, let $Z$ be a tetrahedron containing an edge $\Vec{K}$ representing $7_3$ knot. For all $\mathsf{b}>0$, all $\tau\in\overline{\mathscr{S}_Z}\times\mathscr{S}_{Y\backslash Z}$ such that $\omega_{Y,\tau}$ takes the value $0$ on the edge $\Vec{K}$ and $2\pi$ on the other edges, there exists a real quadratic polynomial $\varphi(\tau)$ of dihedral angles such that 

\[
\lim_{\alpha_Y\rightarrow\tau,\alpha_Y\in\mathscr{S}_Y}\Phi_{\mathsf{b}}\left(\frac{\pi-\omega_{Y,\alpha_Y}(\Vec{K})}{2\pi i\sqrt{\hbar}}\right)Z_{\hbar}(Y,\alpha_Y)=e^{i\frac{\varphi(\tau)}{\hbar}-i\frac{\pi}{12}}J_{S^3,7_3}(\hbar,0).
\]

(3) The following limit obtains the the hyperbolic volume of the complementary space of $7_3$ knot in $S^3$,
\[
\lim_{\hbar\rightarrow 0^+}2\pi\hbar\log|J_{S^3,7_3}(\hbar,0)|=-{\rm Vol}(S^3\backslash 7_3).
\]
\end{theorem}
We note that we choose the contour of the integral in Theorem\ref{uemura} (1) different from the Conjecture \ref{Andersen-Kashaev} (1) in order to evaluate the integral rigorously. 
Regarding Theorem\ref{uemura} (3), we can also numerically confirm its establishment by evaluating the integral using the saddle-point method shown in Appendix \ref{volume-conjecture}.
\section{Derivation of a one vertex H-triangulation for $(S^3,7_3)$ and an ideal tetrahedral decomposition of the complementary space $S^3\backslash 7_3$}\label{triangulation}
In this section, we obtain a one vertex H-triangulation for a pair of $S^3$ and its contained knot $7_3$ and an ideal tetrahedral decomposition for the complementary space $S^3\backslash 7_3$ according to the methods described in \cite{MR718149},\cite{MR3486430}.
\begin{prop}\label{one_veretex_H-triantulation}
An example of one vertex H-triangulation $Y=(T_1,\cdots,T_6,\sim)$ for $(S^3,7_3)$ is given in Figure \ref{fig:one_vertex_Htriangulation}.
To explain the way to view this figure, the top left figure about the tetrahedron $T_1^{+}$ is a top view of the tetrahedron with the bottom face $\partial_0(T_1^{+})$ opposite to the vertex $0$. The variable $x$ assigned to the bottom face $\partial_0(T_1^{+})$ is marked outside, and the variables assigned to the other faces are marked inside the corresponding triangles. The same is true for the other tetrahedra. 
\begin{figure}[tbh]
\centering
\includegraphics[width=\textwidth]{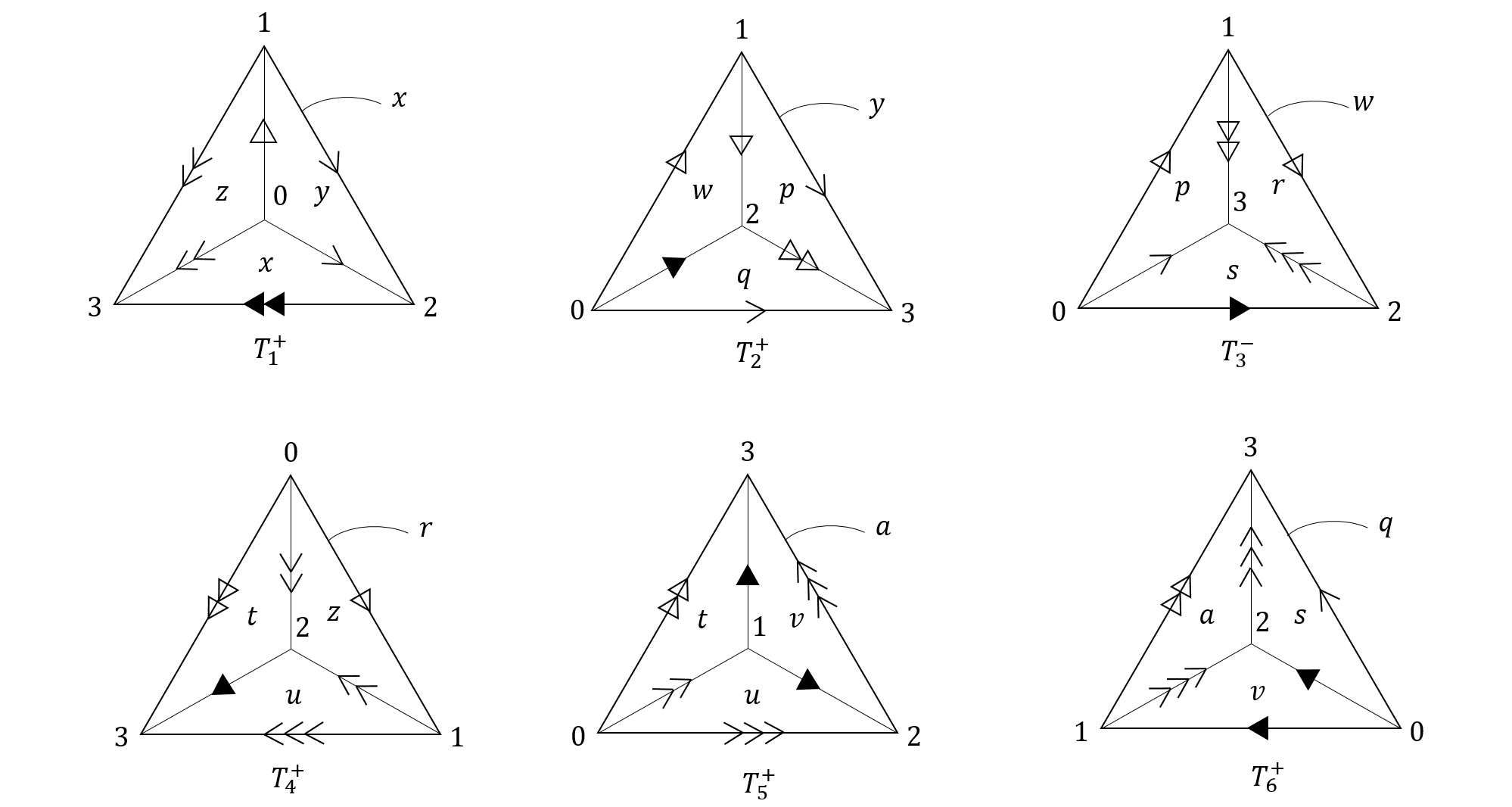}
\caption{One vertex H-triangulation for $(S^3,7_3)$.}
\label{fig:one_vertex_Htriangulation}
\end{figure}
\begin{figure}[tbh]
\centering
\includegraphics[width=0.5\textwidth]{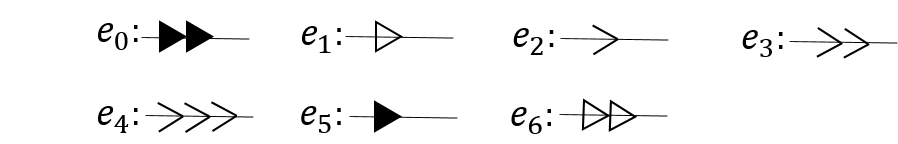}
\caption{Arrow symbols corresponding to edges in one vertex H-triangulation for ($S^3$, $7_3$).}
\label{fig:one_vertex_Htriangulation_edge}
\end{figure}

Let $\Delta_1(Y)=\{e_0, e_1, \cdots, e_6\}$, and the inverse image of each edge before gluing is marked with an arrow symbol. Figure \ref{fig:one_vertex_Htriangulation_edge} shows the correspondence. The edge corresponding to the knot is $e_0$. 
\end{prop}
\begin{proof}
Let $S^3$ be a standard one-point compactification of $\mathbb{R}^3$, and fix the embedding $S^2\subset S^3$ as the closure of the standard embedding $\mathbb{R}^2\subset \mathbb{R}^3$, $(x,y)\mapsto (x,y,0)$. Let $K$ be a knot $7_3$ in $S^3$, take a small positive number $\epsilon$, and let $K \subset \mathbb{R}^2\times[-\epsilon,\epsilon]\subset\mathbb{R}^3$.
Let the projective diagram $D$ be the image of $K$ by the orthogonal projection $\mathbb{R}^3\rightarrow \mathbb{R}^2$. Let $K$ be cellularly decomposed, and small circles in $D$ represent the images of its vertices by the orthogonal projection.
 Figure \ref{fig:knot-diagram} shows $D$.
\begin{figure}[htbp]
\centering
\includegraphics[width=0.3\textwidth]{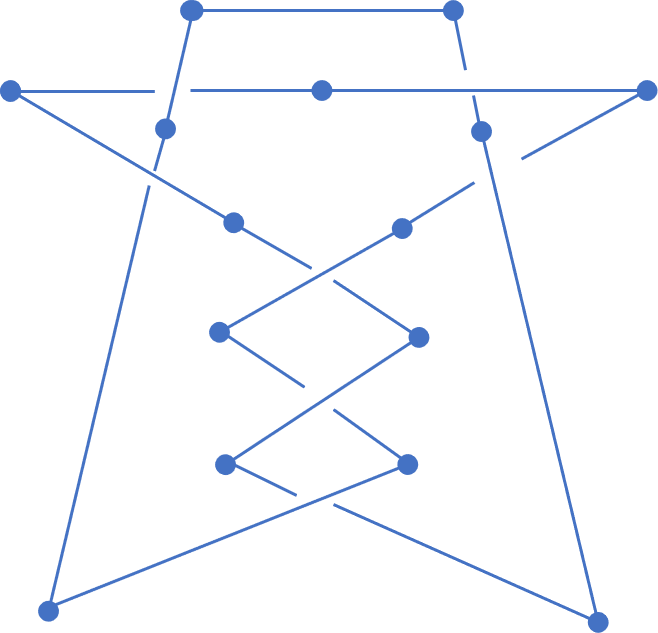}
\caption{A projective diagram of $7_3$ knot with the cellular structure.}
\label{fig:knot-diagram}
\end{figure}
The cellular decomposition of $K$ can be extended to a cellular decomposition of $S^3$ by adding new edges with the same vertices.
Here the edges are added so that each intersection point of $D$ is bounded by the orthogonal projective image of the newly added four edges, as in Figure \ref{fig:shaded-knot-diagram}, and each higher dimensional cell is given by a tetrahedral cell with the natural cellular structure contained in $\mathbb{R}^2\times[-\epsilon,\epsilon]$ such that it is projected onto a shaded rectangle containing the intersection point of $D$.
\begin{figure}[htbp]
\centering
\includegraphics[width=0.3\textwidth]{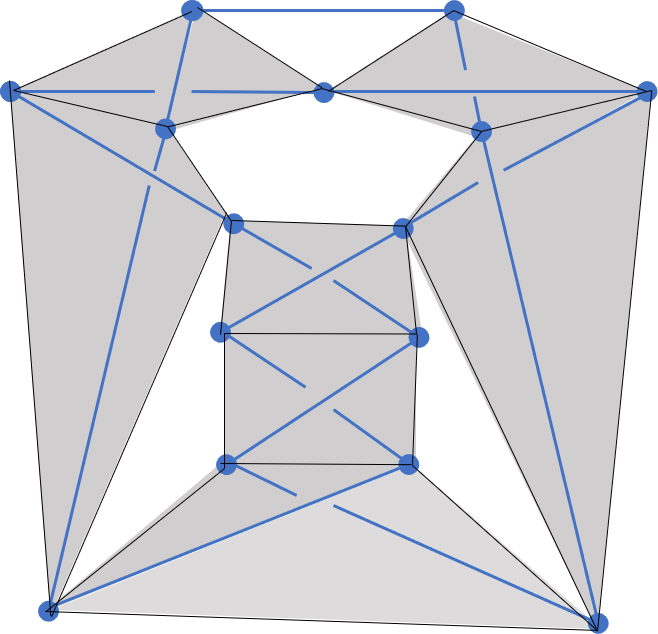}
\caption{An Induced cellular decompositon of $S^3$ such that shaded tetrahedral cells contain the intersection points.}
\label{fig:shaded-knot-diagram}
\end{figure}
The other higher dimensional cells are given by two 3-dimensional cells $\Tilde{c}_\pm$ which are respectively the common part of the 3-dimensional balls $B_+$, $B_-$ defined by the closure of upper and lower half spaces in $S^3$, i.e., $B_\pm=\overline{\left\{(x,y,z)\mid \pm z\ge 0 \right\}}$ and the complementary space of tetrahedral cells in $S^3$.
From the cellular decomposition of $S^3$ constructed in this way, we create a new cellular complex by an isotopy which starts from the identity map and ends up in the projection to the quotient space by an equivalence relation that identifies all points on $K$ except one edge and collapses each tetrahedron to one edge, as in Figure \ref{fig:isotopy}.

\begin{figure}[htbp]
\centering
\includegraphics[width=0.7\textwidth]{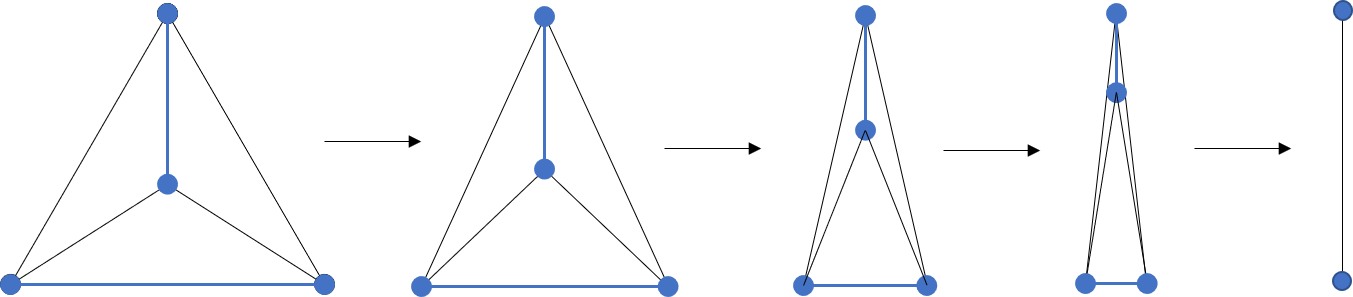}
\caption{Isotopy which collapses a tetrahedron to one edge.}
\label{fig:isotopy}
\end{figure}
The resulting cellular complex consists of two $3$-dimensional cells $c_\pm$, which are the projective images of $\Tilde{c}_\pm$, and the complementary region of the shaded rectangles containing the intersection points of the diagram $D$ gives the 2-skeleton.

In Figure \ref{fig:oriented-shaded-diagram}, the edge of $K$ that is not collapsed to one point is the uppermost horizontal line segment, and the unshaded region, i.e., three triangular cells, one pentagonal cell, and one hexagonal cell corresponding to the outer region, gives the $2$-dimensional cells. The same kinds of arrows attached to the edges represent the edges that are identified, including their orientations if tetrahedra are collapsed into line segments. 
\begin{figure}[htbp]
\centering
\includegraphics[width=0.3\textwidth]{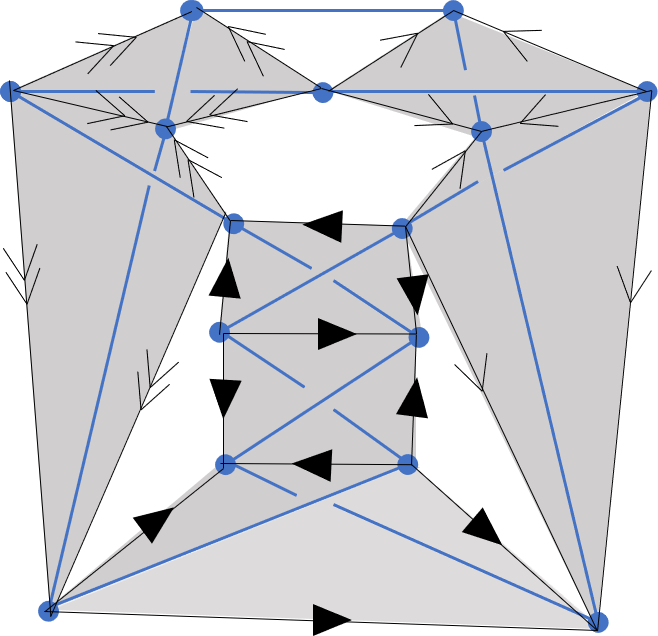}
\caption{Edges marked with the same arrow symbol are considered identical if tetrahedra are collapsed into line segments.}
\label{fig:oriented-shaded-diagram}
\end{figure}

By gluing two 3-dimensional cells $c_\pm$ along the hexagonal 2-dimensional cell corresponding to the outer region in the projective diagram of the knot, we obtain a cellular complex given by a single 3-dimensional cell whose boundaries are composed of remaining 2-dimensional cells as in Figure \ref{glued-complex}.
Here each 2-dimensional cell appears twice with different orientations corresponding to $c_+$ and $c_-$, and in Figure \ref{glued-complex}, 2-dimensional cells are glued together with the same letters as those marked with $\prime$ on the right shoulder. Figure \ref{glued-complex-2-decomposition} shows the cellular complex given by the cellular decomposition of the 2-dimensional cells represented by $2$ and $2^\prime$ in Figure \ref{glued-complex}. Figure \ref{glued-complex-cut-T0} shows the tetrahedron $T_0$ and the remaining cellular complex obtained by splitting the cellular complex with a new two-dimensional cell bounded by the dashed line in Figure \ref{glued-complex-2-decomposition}.

\begin{figure}[htbp]
\begin{tabular}{ccc}
\begin{subfigure}{0.33\textwidth}
\centering
\includegraphics[height=0.20\vsize]{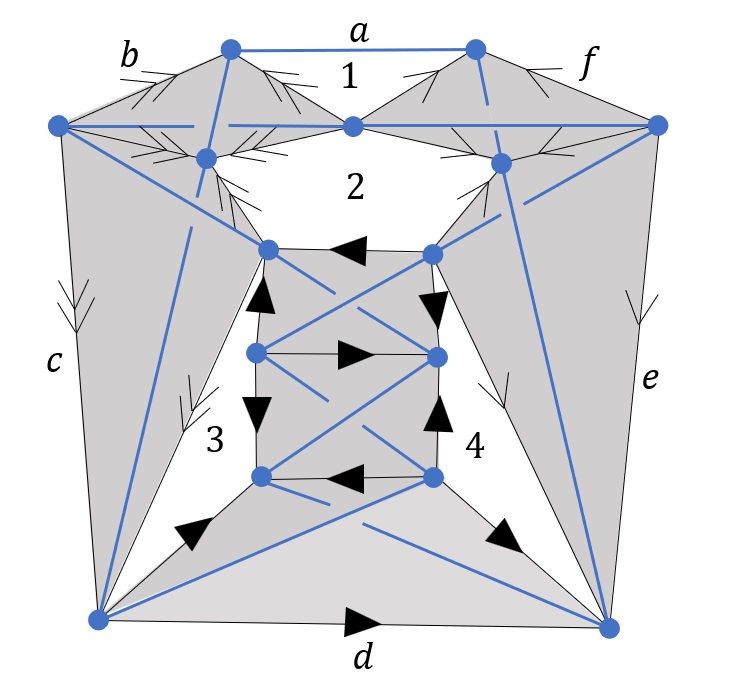}
\subcaption{}\label{a}
\end{subfigure}
& 
\begin{subfigure}{0.33\textwidth}
\centering
\includegraphics[height= 0.20\vsize]{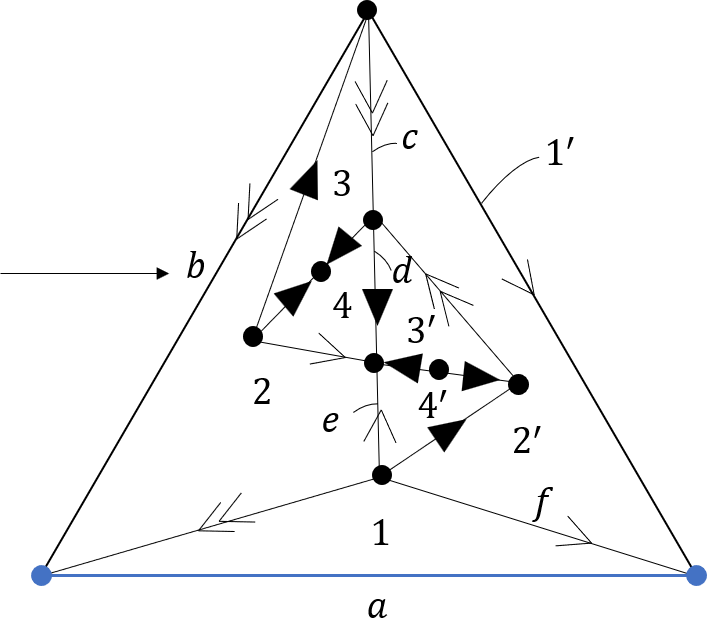}
\subcaption{}
\label{glued-complex} 
\end{subfigure}
& 
\begin{subfigure}{0.33\textwidth}
\centering
\includegraphics[height = 0.20\vsize]{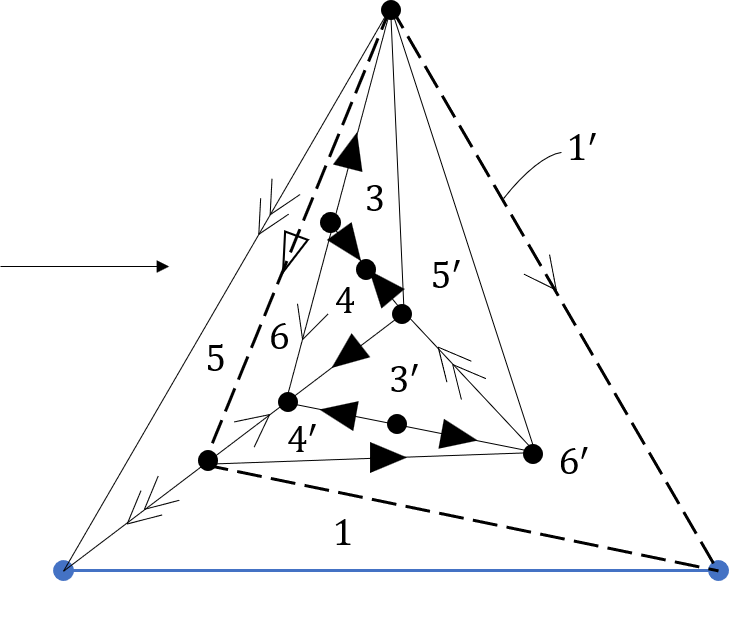}
\subcaption{}
\label{glued-complex-2-decomposition}
\end{subfigure}
\\
\multicolumn{3}{c}{
\begin{subfigure}{0.99\textwidth}
\centering
\includegraphics[height= 0.20\vsize]{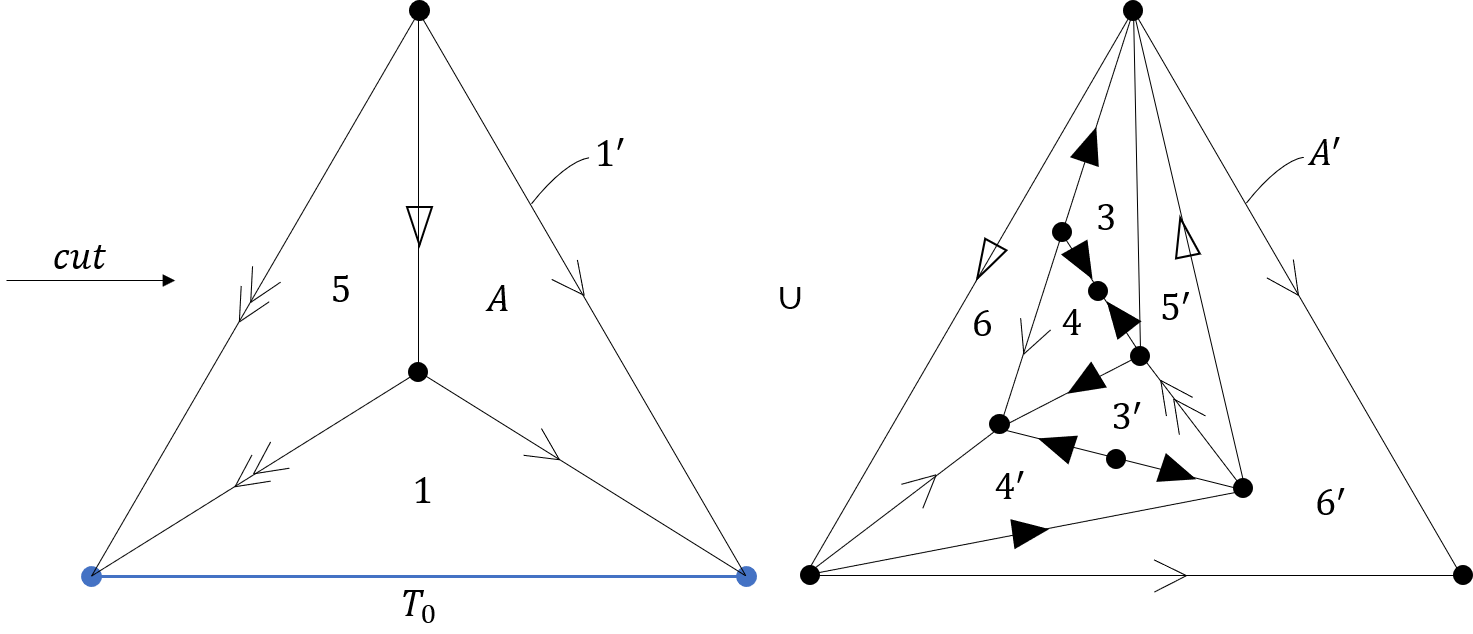}
\subcaption{}
\label{glued-complex-cut-T0}
\end{subfigure}}
\\
\end{tabular}
\caption{The cellular complex obtained by gluing $c_+$ and $c_-$ with the outer hexagon in Figure \ref{fig:oriented-shaded-diagram}.}
\end{figure}

Figure \ref{fig:T0_remainder} shows the cellular decomposition of $2$-dimensional cells represented by the symbol $6$, $6^\prime$ in the non-tetrahedral cellular complex obtained in Figure \ref{glued-complex-cut-T0} and Figure \ref{fig:glued-complex-cut-T1} represents the tetrahedron $T_1$ and the cellular complex obtained by splitting the tetrahedron $T_1$ by a new $2$-dimensional cell whose boundary is the dashed line in Figure \ref{fig:T0_remainder}.

\begin{figure}[htbp]
\begin{tabular}{c}
 \begin{minipage}[t]{0.96\hsize} 
\centering
\includegraphics[height= 0.20\vsize]{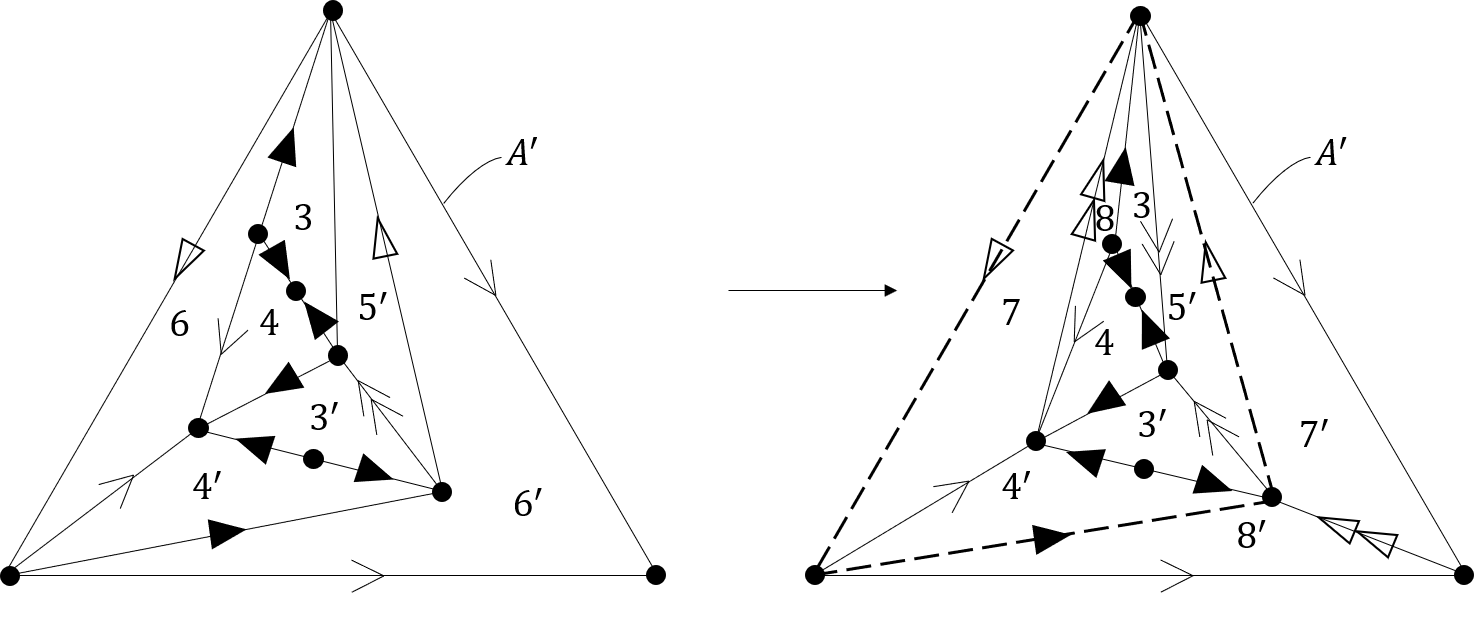}
\subcaption{}
\label{fig:T0_remainder}
\end{minipage} \\
 \begin{minipage}[t]{0.96\hsize} 
\centering
\includegraphics[height= 0.20\vsize]{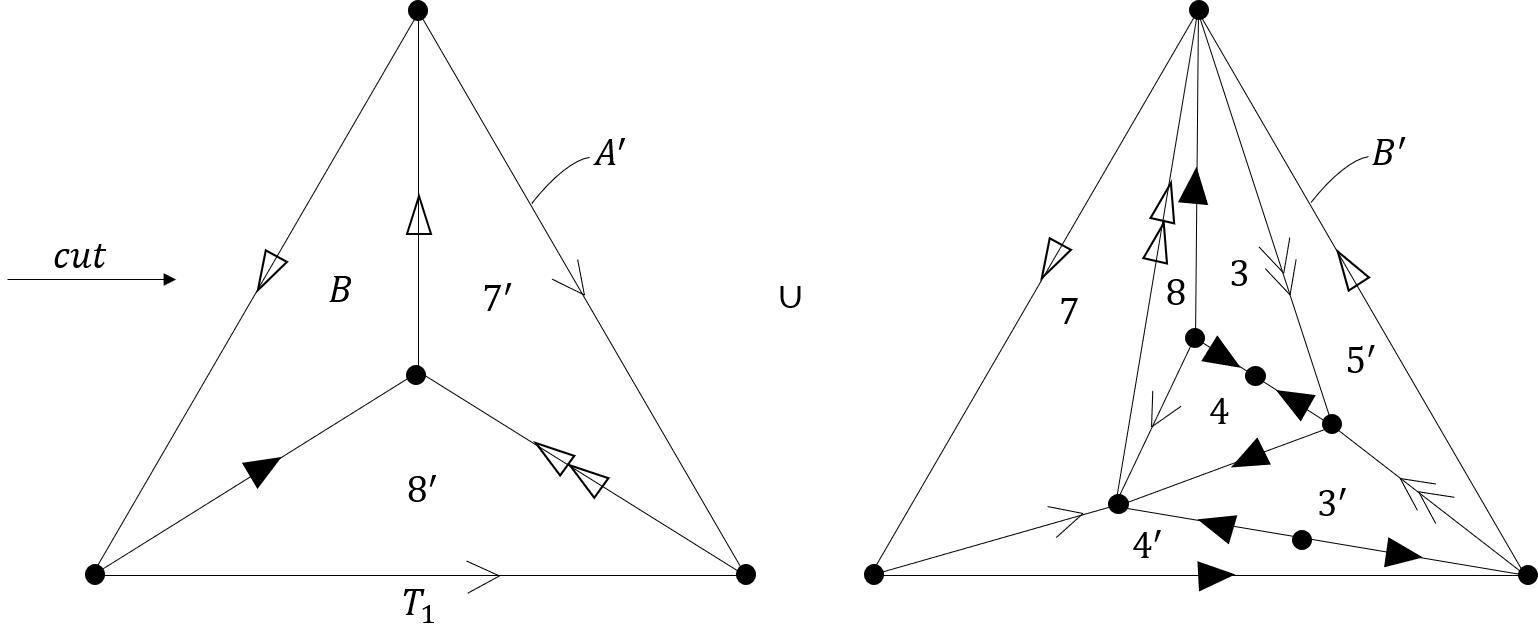}
\subcaption{}
\label{fig:glued-complex-cut-T1}
\end{minipage}
\end{tabular}
\caption{Cellular decomposition of the cellular complex obtained in Figure \ref{glued-complex-cut-T0}.}
\end{figure}

Similarly, Figure \ref{fig:T1-remainder} shows the cellular decomposition of $2$-dimensional cells represented by the symbol $4$,$4^\prime$ in the non-tetrahedral cellular complex obtained in Figure \ref{fig:glued-complex-cut-T1}, and Figure \ref{fig:T2-cut} represents the tetrahedron $T_2$ and the cellular complex obtained by splitting the tetrahedron $T_2$ by a new $2$-dimensional cell whose boundary is the dashed line in Figure \ref{fig:T1-remainder}.

\begin{figure}[htbp]
\begin{tabular}{c}
 \begin{minipage}[t]{0.96\hsize} 
\centering
\includegraphics[height= 0.20\vsize]{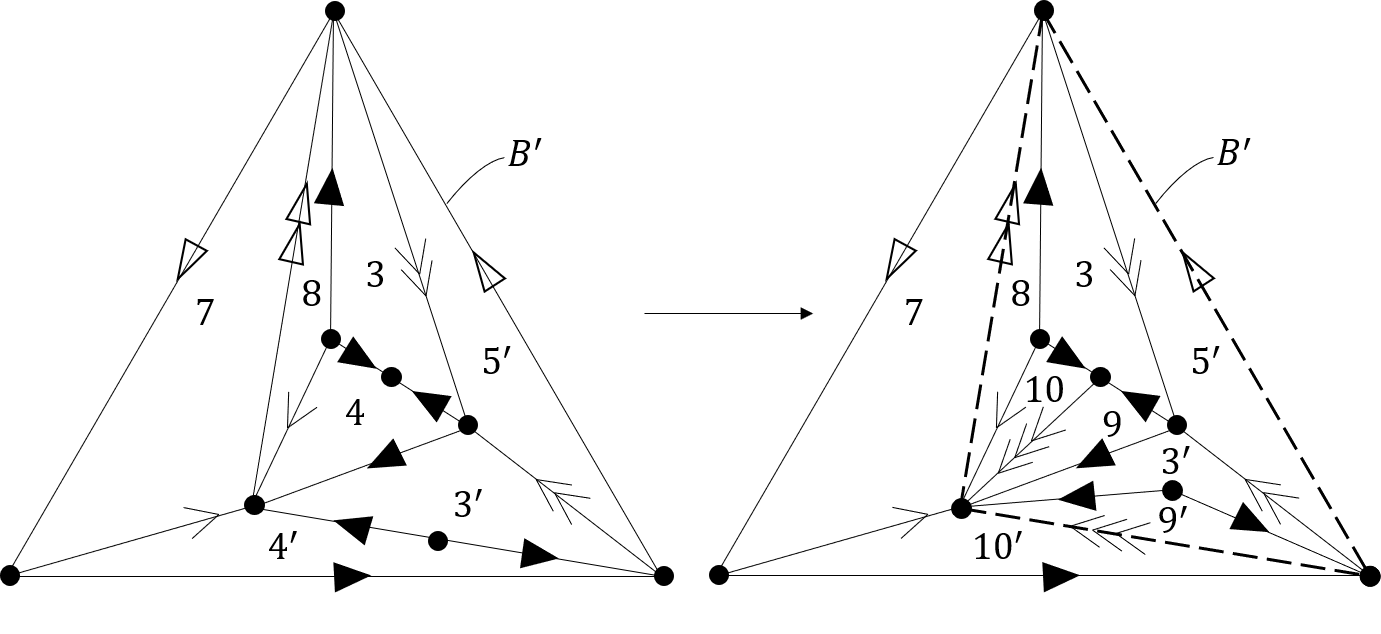}
\subcaption{}
\label{fig:T1-remainder}
\end{minipage} \\
 \begin{minipage}[t]{0.96\hsize} 
\centering
\includegraphics[height= 0.20\vsize]{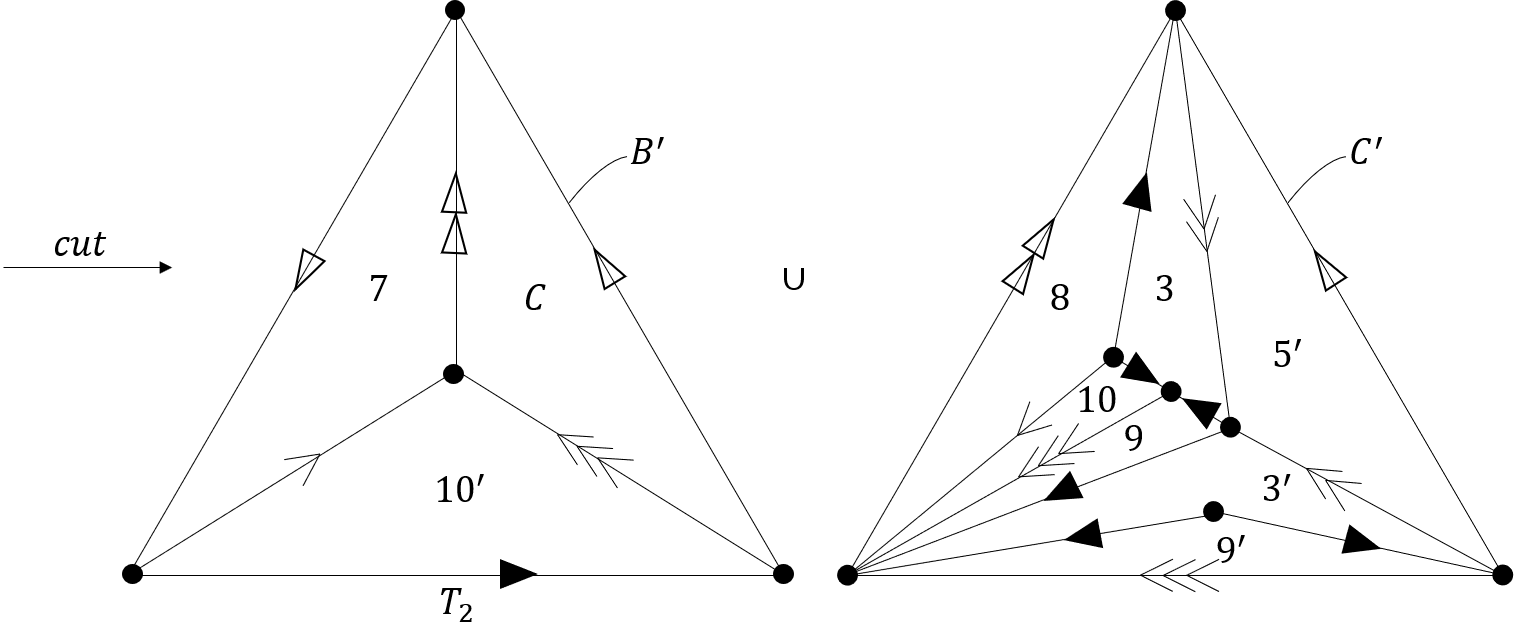}
\subcaption{}
\label{fig:T2-cut}
\end{minipage}
\end{tabular}
\caption{Cellular decomposition of the cellular complex obtained in Figure \ref{fig:glued-complex-cut-T1}.}
\end{figure}

Figure \ref{fig:T2-remainder} shows the cellular decomposition of the $2$-dimensional cells represented by the symbol $3$, $3^\prime$ in the non-tetrahedral cellular complex obtained in Figure\ref{fig:T2-cut}, and Figure \ref{fig:cut-T3-T6} shows the tetrahedra $T_3,\cdots,T_6$ obtained by splitting them by the $2$-dimensional cells whose boundaries are the dashed lines in Figure \ref{fig:T2-remainder}. 

\begin{figure}[htbp]
\begin{tabular}{c}
 \begin{minipage}[t]{\hsize} 
\centering
\includegraphics[height= 0.20\vsize]{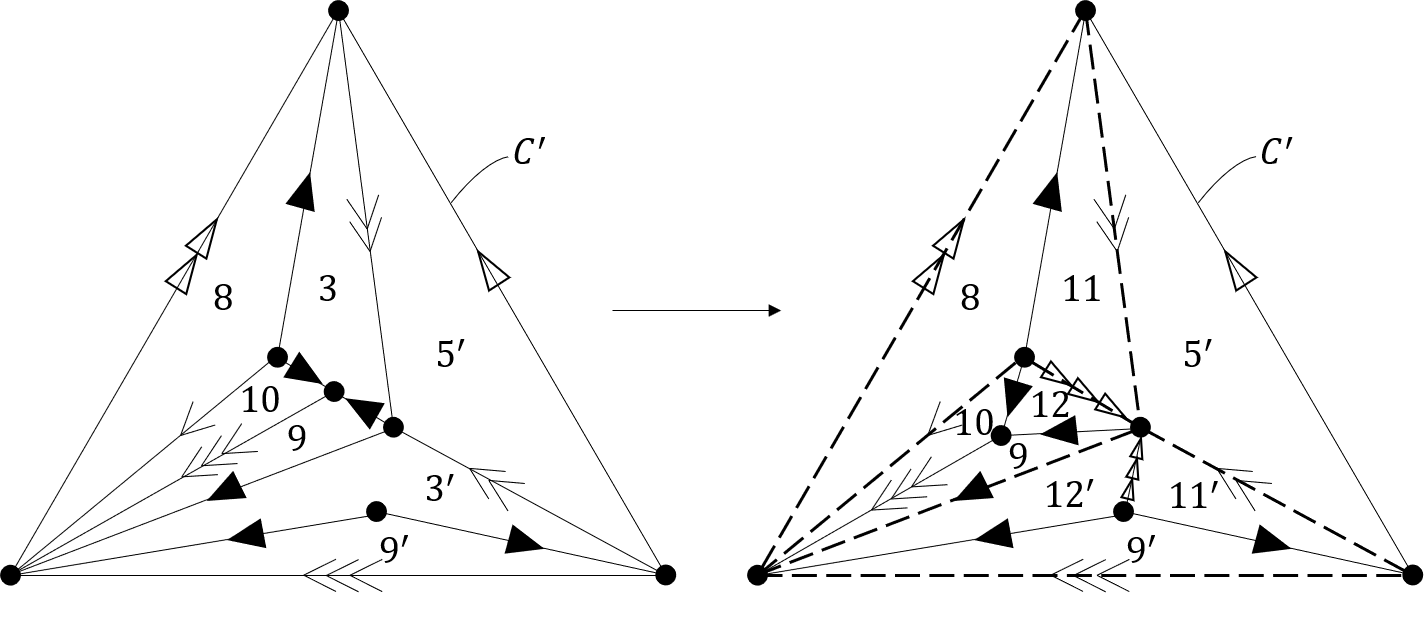}
\subcaption{}
\label{fig:T2-remainder}
\end{minipage} \\
 \begin{minipage}[t]{\hsize} 
\centering
\includegraphics[height= 0.15\vsize]{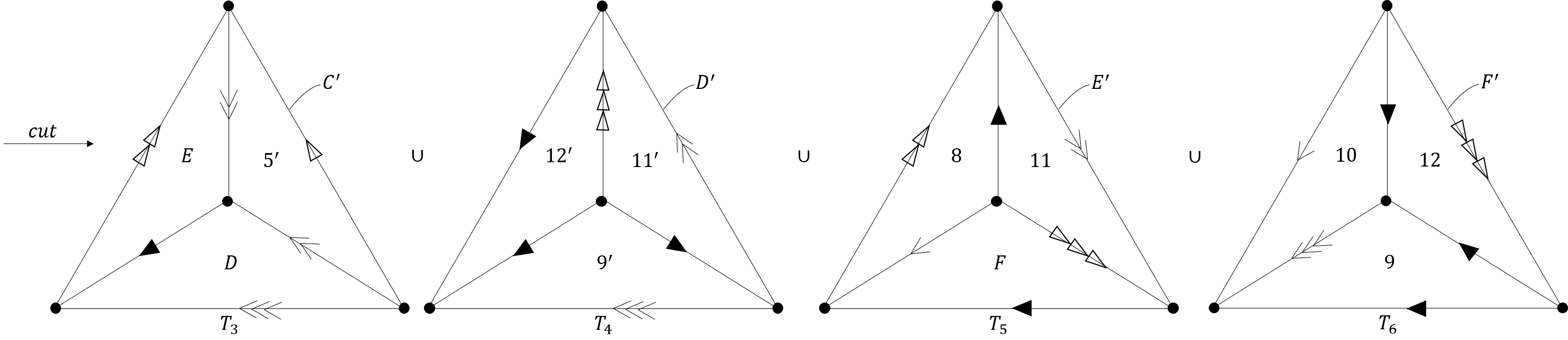}
\subcaption{}
\label{fig:cut-T3-T6}
\end{minipage}
\end{tabular}
\caption{Tetrahedral decomposition for the non-tetrahedral cellular complex  obtained in Figure \ref{fig:T2-cut}.}
\end{figure}

Furthermore, for the cellular complex consisting of tetrahedra $T_4$, $T_5$, and $T_6$ glued together with their corresponding $2$-dimensional cells, another cellular decomposition consisting of two tetrahedra $T_7,T_8$ is obtained by applying the Pachner $3-2$ move as shown in Figure \ref{fig:7_3-Pachner}.

\begin{figure}[htbp]
    \centering
    \includegraphics[height=0.19\vsize]{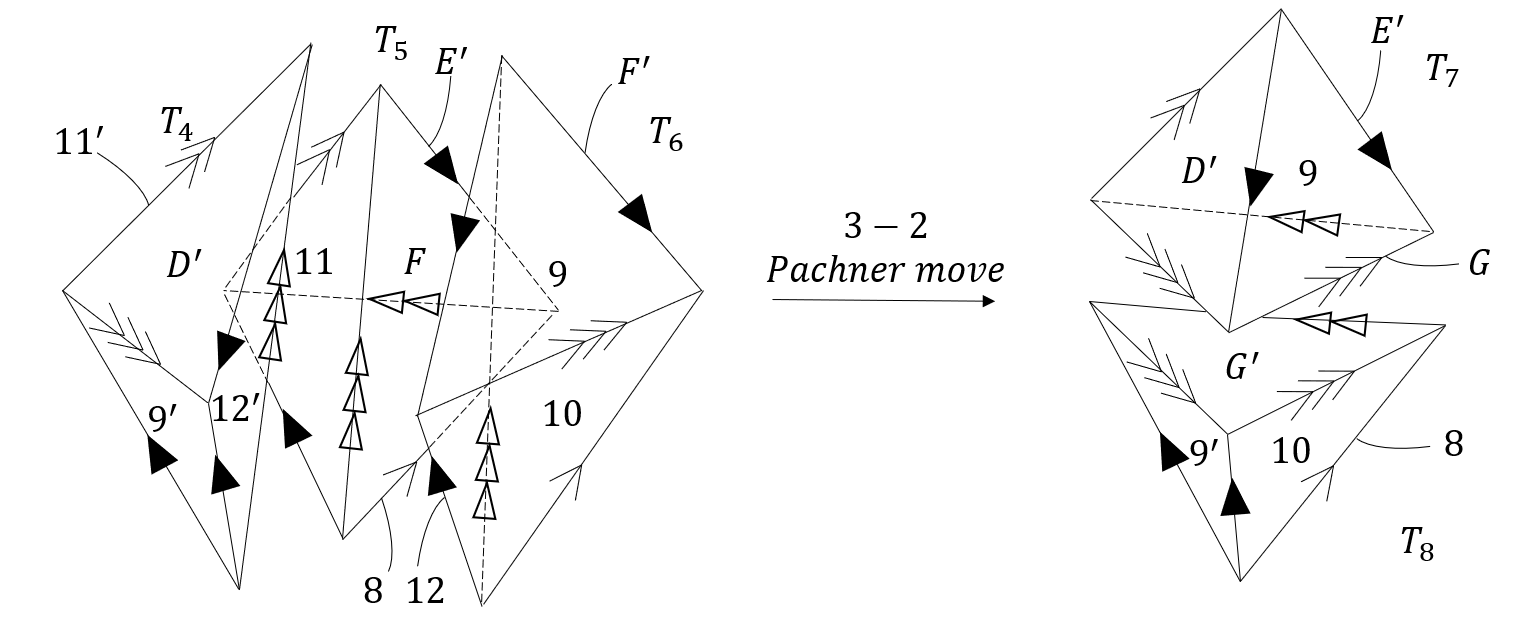}
    \caption{Pachner 3-2 move for tetrahedra $T_4,T_5$,and $T_6$.}
    \label{fig:7_3-Pachner}
\end{figure}

The tetrahedra $T_0,T_1,T_2,T_3,T_7$, and $T_8$ obtained in the above process give the one vertex H-triangulation.
\end{proof}
\begin{prop}\label{thm:ideal_triangulation}
An example of ideal tetrahedral decompositions of $S^3\backslash 7_3$, $X=(T_1,\cdots,T_5,\sim)$, is given in Figure \ref{fig:ideal_triangulation}.
Here $\Delta_1(X)=\{ e_1, e_2, \cdots, e_5 \}$ and the edges of the inverse image of each edge by $\sim$ are marked with the same arrow symbol. The correspondence between the arrow symbols and the edges $e_1,e_2,\cdots,e_5$ is shown in Figure \ref{fig:ideal_triangulation_edge}.
\begin{figure}[htbp]
\centering
\includegraphics[width=\textwidth]{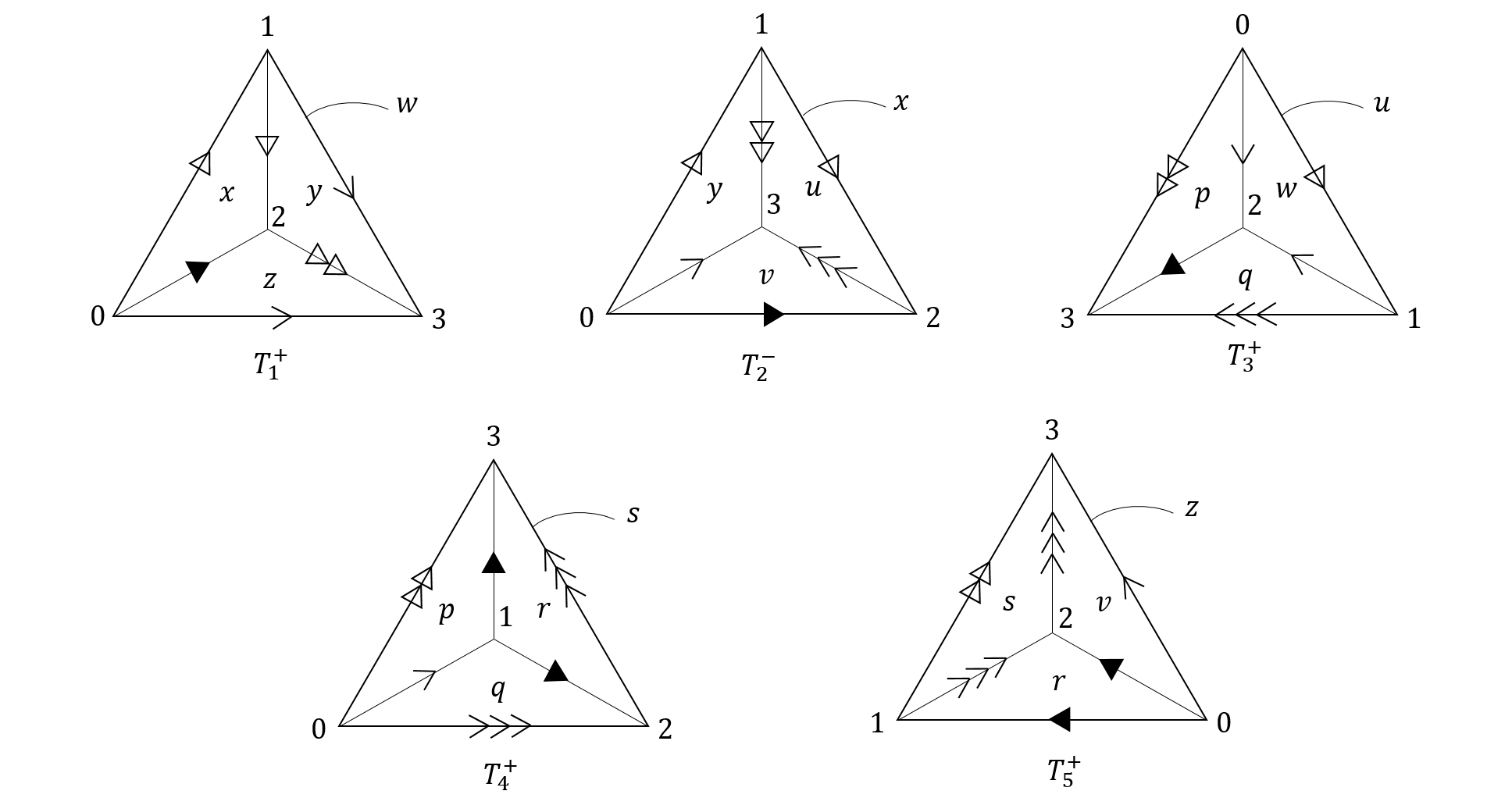}
\caption{An Ideal tetrahedral decomposition for the complementary space of $7_3$ knot.}
\label{fig:ideal_triangulation}
\end{figure}
\begin{figure}[htbp]
\centering
\includegraphics[width=0.4\textwidth]{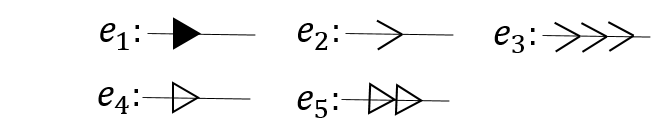}
\caption{Arrow symbols corresponding to edges in the ideal tetrahedral decomposition for the complementary space of $7_3$ knot.}
\label{fig:ideal_triangulation_edge}
\end{figure}
\end{prop}
\begin{proof}
 For the one vertex H-triangulation obtained by Theorem \ref{one_veretex_H-triantulation}, this decomposition is obtained by collapsing the edge $e_0$ corresponding to the knot to one point and the tetrahedron $T_1^{+}$ by gluing the faces $\partial_2(T_1^{+})$ and $\partial_3(T_1 ^{+})$. In this case, we glue the faces which are assigned the variables $y$ and $z$ in Figure \ref{fig:one_vertex_Htriangulation}. 
\end{proof}
\section{Calculation of the partition function for the ideal tetrahedral decomposition of the complementary space of $7_3$ knot.}
\label{ideal-triangulation}
In this section, we calculate the partition function of the ideal tetrahedral decomposition $X$ for the knot complement $S^3\backslash 7_3$ obtained by Proposition \ref{thm:ideal_triangulation} in Section \ref{triangulation}.

In general, the second-order homology group with an integer coefficient of the complementary space of the knot embedded in $S^3$ is shown to be trivial, either by using the Mayer-Vietoris exact sequence \cite{MR1472978} or the Alexander duality theorem. Then $H_2(X-\Delta_0(X))=0$. Therefore, $X$ is admissible. 

Let $\alpha_X=(2\pi a_1,2\pi b_1,2\pi c_1,\cdots,2\pi a_5, 2\pi b_5,2\pi c_5)\in\mathscr{S}_X$.

From Figure \ref{fig:ideal_triangulation}, the condition that $\alpha_X\in\mathscr{A}_X$, i.e., the weight of each edge is $2\pi$, is
\begin{align}
e_1: &\quad & b_1+b_2+a_3+b_4+c_4+a_5+b_5&=1\label{it1}\\
e_2:&\quad & b_1+c_1+c_2+b_3+c_3+a_4+c_5&=1\label{it2}\\
e_3:&\quad & a_2+b_3+a_4+b_4+a_5+c_5&=1\label{it3}\\
e_4:&\quad & a_1+c_1+a_2+c_2+a_3&=1\label{it4}\\
e_5:&\quad & a_1+b_2+c_3+c_4+b_5&=1\label{it5}.
\end{align}

Using $a_i+b_i+c_i=\frac{1}{2}$ for $i=1,2,\ldots,5$, 
\begin{empheq}[left={\eqref{it1},\eqref{it2},\eqref{it3},\eqref{it4},\eqref{it5}\iff\empheqlbrace}]{align}
b_1+b_2+a_3&=a_4+c_5 \label{it1'}\\
a_2&=2b_1+c_1 \label{it2''}\\
a_2+b_3&=c_4+b_5 \label{it3'}\\
a_3&=b_1+b_2\label{it4'}.
\end{empheq}
\begin{lemma}\label{non-empty}
$\mathscr{A}_X$ is non-empty.
\end{lemma}
\begin{proof}
$(a_1,b_1,c_1) =(\frac{1}{4},\frac{1}{8},\frac{1}{8})$, 
$(a_2,b_2,c_2)=(\frac{3}{8},\frac{1}{16},\frac{1}{16})$,  
$(a_3,b_3,c_3)=(\frac{3}{16},\frac{1}{8},\frac{3}{16})$, 
$(a_4,b_4,c_4)=(\frac{1}{8},\frac{1}{16},\frac{5}{16})$, 
$(a_5,b_5,c_5)=(\frac{1}{16},\frac{3}{16},\frac{1}{4})$
gives a non-empty element of $\mathscr{A}_X$.
\end{proof}
In the following, we calculate the partition function $Z_{\hbar}(X,\alpha_X)$ of the fully balanced shaped ideal triangulation $(X,\alpha_X)$. 
\begin{theorem}\label{volume-conjecture-1}
The partition function of the fully balanced shaped ideal triangulation $(X,\alpha_X)$ is expressed as follows.
\begin{equation*}
Z_\hbar(X,\alpha_X)=e^{-\frac{\pi}{3}i}e^{i\frac{\phi(\alpha_X)}{\hbar}}
\int_{\mathbb{R}+\frac{i\mu(\alpha_X)}{\sqrt{\hbar}}}J_X(\hbar,x)e^{-\frac{x\lambda(\alpha_X)}{\sqrt{\hbar}}}dx,
\end{equation*}
where $\lambda(\alpha_X)$ and $\mu(\alpha_X)$ are gauge invariant real linear combinations of dihedral angles, $\phi(\alpha_X)$ is a real quadratic polynomial of dihedral angles, and 
\begin{equation*}
\begin{split}
J_X:(\hbar,x)\mapsto
\int_{\mathscr{Y}^\prime} d\mathbf{y}^\prime e^{2i\pi({\mathbf{y}^\prime}^{\rm T}Q\mathbf{y}^\prime-\frac{3}{2}x^2)}e^{\frac{1}{\sqrt{\hbar}}{\mathbf{y}^\prime}^{\rm T}\mathscr{W}}\frac{\Phi_{\mathsf{b}}(Z^\prime)}{\Phi_{\mathsf{b}}(Y^\prime)\Phi_{\mathsf{b}}(W^\prime)\Phi_{\mathsf{b}}(W^\prime-x)\Phi_{\mathsf{b}}(W^\prime+x)}
\end{split}
\end{equation*}
where
\begin{eqnarray*}
\mathscr{Y}^\prime&\coloneqq&(\mathbb{R}-2c_{\mathsf{b}}(c_1+b_1))\times(\mathbb{R}+2c_{\mathsf{b}}(b_2+c_2))\times(\mathbb{R}-2c_{\mathsf{b}}(c_3+b_3))\\
&=&\left(\mathbb{R}-\frac{i}{2\sqrt{\hbar}}(1-2a_1)\right)\times\left(\mathbb{R}+\frac{i}{2\sqrt{\hbar}}(1-2a_2)\right)\times\left(\mathbb{R}-\frac{i}{2\sqrt{\hbar}}(1-2a_3)\right)
\end{eqnarray*}
and let
\begin{equation*}
\mathbf{y}^\prime=\begin{bmatrix}
Y^\prime\\
Z^\prime\\
W^\prime
\end{bmatrix},\quad
\mathscr{W}=\begin{bmatrix}
-\pi\\
0\\
\pi
\end{bmatrix},\quad
Q=\begin{bmatrix}
1 & -\frac{1}{2} & 0\\
-\frac{1}{2} & 0 & \frac{1}{2}\\
0 & \frac{1}{2} &\frac{1}{2}
\end{bmatrix}.
\end{equation*}
\end{theorem}
\begin{proof}
\begin{equation*}
\begin{split}
Z_\hbar(X,\alpha_X)&=\int_{\mathbb{R}^{10}}dxdydzdwdudvdpdqdrds\ \langle y,w\lvert \mathsf{T}(a_1,c_1)\rvert z,x\rangle \langle v,x\lvert \overline{\mathsf{T}}(a_2,c_2)\rvert u,y\rangle \\ 
&\quad\quad\times \langle q,u\lvert \mathsf{T}(a_3,c_3)\rvert p,w\rangle
\langle r,p\lvert \mathsf{T}(a_4,c_4)\rvert s,q\rangle \langle s,z\lvert \mathsf{T}(a_5,c_5)\rvert v,r\rangle \\
&=\int_{\mathbb{R}^{10}}dxdydzdwdudvdpdqdrds\ \delta(y+w-z)e^{-\frac{\pi}{12}i}\psi_{c_1,b_1}(x-w)e^{2\pi iy(x-w)}\\
& \quad\quad\times  \delta(u+y-v)
\psi_{b_2,c_2}(y-x)e^{-\frac{\pi}{12}i}e^{\pi i(y-x)^2}e^{-2\pi iu(x-y)}\delta(q+u-p)e^{-\frac{\pi}{12}i}\\ 
&\quad\quad\times \psi_{c_3,b_3}(w-u)
e^{2\pi iq(w-u)}
\delta(r+p-s)e^{-\frac{\pi}{12}i}\psi_{c_4,b_4}(q-p)e^{2\pi i r(q-p)}\\
&\quad\quad\times\delta(s+z-v)e^{-\frac{\pi}{12}i}\psi_{c_5,b_5}(r-z)e^{2\pi is(r-z)}.
\end{split}
\end{equation*}
Iterated integration with respect to the variables $q$, $p$, $s$, $v$, and $z$, in turn, gives
\begin{equation*}
\begin{split}
Z_\hbar(X,\alpha_X)&=\int_{\mathbb{R}^5}dxdydwdudr\ e^{-\frac{5\pi}{12}i}\psi_{c_1,b_1}(x-w)e^{2\pi iy(x-w)}\psi_{b_2,c_2}(y-x)e^{\pi i(y-x)^2}\\
&\quad\quad\times e^{-2\pi iu(x-y)}\psi_{c_3,b_3}(w-u)
 e^{-2\pi i(r+w)(w-u)}\psi_{c_4,b_4}(-u)e^{-2\pi iru}\\
& \quad\quad\times \psi_{c_5,b_5}(r-y-w)e^{2\pi i(u-w)(r-y-w)}.
\end{split}
\end{equation*}
Transformation of variables,
$\left\{
\begin{aligned}
&A\coloneqq x-w\\
&B\coloneqq y-x\\
&C\coloneqq w-u\\
&D\coloneqq -u\\
&E\coloneqq r-y-w
\end{aligned},
\right.$
gives 
$\left\{
\begin{aligned}
&x=A+C-D\\
&y=A+B+C-D\\
&w=C-D\\
&u=-D\\
&r=A+B+2C-2D+E
\end{aligned}
\right.$
and its Jacobian is 
$
\begin{vmatrix}
1 & 0 & 1& -1&0\\
1 & 1& 1 & -1&0 \\
0 & 0 & 1 & -1&0 \\
0 & 0 & 0 & -1&0\\
1 & 1 & 2 & -2&1
\end{vmatrix}
=-1$. Then 
\begin{equation*}
\begin{split}
Z_{\hbar}(X,\alpha_X)
&=\int_{\mathbb{R}^5}dAdBdCdDdE\ e^{-\frac{5\pi}{12}i}\psi_{c_1,b_1}(A)\psi_{b_2,c_2}(B)\psi_{c_3,b_3}(C)\psi_{c_4,b_4}(D)\psi_{c_5,b_5}(E)\\
&\quad\quad\times e^{\pi iB^2}e^{2\pi i(A^2+AB-3C^2+5CD-2CE-2D^2+DE-BC)}.
\end{split}
\end{equation*}
Since
\begin{equation*}
\int_{\mathbb{R}}dE\ \psi_{c_5,b_5}(E)e^{2\pi iE(-2C+D)}=\widetilde{\psi}_{c_5,b_5}(2C-D)=e^{\pi i(2C-D)^2}e^{-\frac{\pi}{12}i}\psi_{b_5,a_5}(2C-D) , 
\end{equation*}
\begin{equation*}
\begin{split}
Z_{\hbar}(X,\alpha_X)
&=\int_{\mathbb{R}^4}dAdBdCdDe^{-\frac{\pi}{2}i}\psi_{c_1,b_1}(A)\psi_{b_2,c_2}(B)\psi_{c_3,b_3}(C)
\psi_{c_4,b_4}(D)\psi_{b_5,a_5}(2C-D)\\
&\quad\quad\times e^{\pi iB^2}e^{2\pi i(A^2+AB-C^2+3CD-BC)}e^{-3\pi iD^2}\\
&=\int_{\mathbb{R}^4}dAdBdCdD  e^{-\frac{\pi}{2}i}\psi(A-2c_{\mathsf{b}}(b_1+c_1))e^{-4\pi ic_{\mathsf{b}}c_1(A-c_{\mathsf{b}}(c_1+b_1))}\\
&\quad\quad\times e^{-\pi ic_{\mathsf{b}}^2\frac{4(c_1-b_1)+1}{6}}
\psi(B-2c_{\mathsf{b}}(b_2+c_2))e^{-4\pi ic_{\mathsf{b}}b_2(B-c_{\mathsf{b}}(b_2+c_2))}\\
&\quad\quad\times e^{-\pi ic_{\mathsf{b}}^2\frac{4(b_2-c_2)+1}{6}}\psi(C-2c_{\mathsf{b}}(c_3+b_3))e^{-4\pi ic_{\mathsf{b}}c_3(C-c_{\mathsf{b}}(c_3+b_3))}\\
&\quad\quad\times e^{-\pi ic_{\mathsf{b}}^2\frac{4(c_3-b_3)+1}{6}}\psi(D-2c_{\mathsf{b}}(c_4+b_4))e^{-4\pi ic_{\mathsf{b}}c_4(D-c_{\mathsf{b}}(c_4+b_4))}\\
&\quad\quad\times e^{-\pi ic_{\mathsf{b}}^2\frac{4(c_4-b_4)+1}{6}}\psi(2C-D-2c_{\mathsf{b}}(b_5+a_5))e^{-4\pi ic_{\mathsf{b}}b_5(2C-D-c_{\mathsf{b}}(b_5+a_5))}\\
&\quad\quad\times e^{-\pi ic_{\mathsf{b}}^2\frac{4(b_5-a_5)+1}{6}}e^{\pi iB^2}e^{2\pi i(A^2+AB-C^2+3CD-BC)}e^{-3\pi iD^2}.
\end{split}
\end{equation*}
Transformation of variables 
\begin{equation*}
\begin{aligned}
&\Tilde{A}\coloneqq A-2c_{\mathsf{b}}(c_1+b_1)\\
&\Tilde{B}\coloneqq B-2c_{\mathsf{b}}(b_2+c_2)\\
&\Tilde{C}\coloneqq C-2c_{\mathsf{b}}(c_3+b_3)\\
&\Tilde{D}\coloneqq D-2c_{\mathsf{b}}(c_4+b_4)
\end{aligned}
\end{equation*}
gives, using the equations \eqref{it1'}, \eqref{it4'}, and $a_i+b_i+c_i=\frac{1}{2}\quad( i=1, \ldots, 5 )$,
\begin{equation*} 
2C-D-2c_{\mathsf{b}}(a_5+b_5)=2(\Tilde{C}+2c_{\mathsf{b}}(b_3+c_3))-\Tilde{D}-2c_{\mathsf{b}}(b_4+c_4)-2c_{\mathsf{b}}(a_5+b_5)=2\Tilde{C}-\Tilde{D}.
\end{equation*}
Concerning
\begin{equation*}
\begin{split}
&\quad e^{-4\pi ic_{\mathsf{b}}c_1 A}e^{-4\pi ic_{\mathsf{b}}b_2B}e^{-4\pi ic_{\mathsf{b}}c_3C}e^{-4\pi ic_{\mathsf{b}}c_4D}
e^{-4\pi ic_{\mathsf{b}}b_5(2C-D)}e^{\pi iB^2}e^{2\pi i(A^2+AB-C^2+3CD-BC)}
\\
&\quad \times e^{-3\pi iD^2}\\
&= e^{-4\pi ic_{\mathsf{b}}c_1(\Tilde{A}+2c_{\mathsf{b}}(c_1+b_1))}e^{-4\pi ic_{\mathsf{b}}b_2(\Tilde{B}+2c_{\mathsf{b}}(b_2+c_2))}e^{-4\pi ic_{\mathsf{b}}c_3(\Tilde{C}+2c_{\mathsf{b}}(c_3+b_3))}e^{-4\pi ic_{\mathsf{b}}c_4(\Tilde{D}+2c_{\mathsf{b}}(c_4+b_4))}\\
&\quad\times e^{-4\pi ic_{\mathsf{b}}b_5(2\Tilde{C}-\Tilde{D}+2c_{\mathsf{b}}(a_5+b_5))}e^{\pi i\left\{\Tilde{B}+2c_{\mathsf{b}}(b_2+c_2)\right\}^2}e^{2\pi i\left\{\Tilde{A}+2c_{\mathsf{b}}(c_1+b_1)\right\}^2}\\
&\quad\times e^{2\pi i\left\{\Tilde{A}+2c_{\mathsf{b}}(c_1+b_1)\right\}\left\{\Tilde{B}+2c_{\mathsf{b}}(b_2+c_2)\right\}}e^{-2\pi i\left\{\Tilde{C}+2c_{\mathsf{b}}(c_3+b_3)\right\}^2}e^{6\pi i \left\{\Tilde{C}+2c_{\mathsf{b}}(c_3+b_3)\right\}\left\{\Tilde{D}+2c_{\mathsf{b}}(c_4+b_4)\right\}}\\
&\quad\times e^{-2\pi i\left\{\Tilde{B}+2c_{\mathsf{b}}(b_2+c_2)\right\}\left\{\Tilde{C}+2c_{\mathsf{b}}(c_3+b_3)\right\}}e^{-3\pi i\left\{\Tilde{D}+2c_{\mathsf{b}}(c_4+b_4)\right\}^2},
\end{split}
\end{equation*}
the first-order terms concerning $\Tilde{A}$ with constant coefficients in the exponent of $e$ are 
\begin{equation*}
\begin{split}
&\quad 4\pi ic_{\mathsf{b}}\Tilde{A}\left(-c_1+2(b_1+c_1)+b_2+c_2\right)=4\pi ic_{\mathsf{b}}\Tilde{A}\left(-c_1+2\left(\frac{1}{2}-a_1\right)+\frac{1}{2}-a_2\right)\\
&
=4\pi ic_{\mathsf{b}}\left(\frac{3}{2}-2a_1-c_1-2b_1-c_1\right)\quad (\because\eqref{it2''})\\
&=4\pi ic_{\mathsf{b}}\Tilde{A}\left(\frac{3}{2}-2(a_1+b_1+c_1)\right)=2\pi ic_{\mathsf{b}}\Tilde{A},
\end{split}
\end{equation*}
the first-order terms concerning $\Tilde{B}$ with constant coefficients in the exponent of $e$ are
\begin{equation*}
\begin{split}
&\quad 4\pi ic_{\mathsf{b}}\Tilde{B}(-b_2+b_2+c_2+b_1+c_1-(b_3+c_3))=4\pi ic_{\mathsf{b}}\left(c_2+\frac{1}{2}-a_1-\left(\frac{1}{2}-a_3\right)\right)\\
&=4\pi ic_{\mathsf{b}}\Tilde{B}(c_2-a_1+a_3)=4\pi ic_{\mathsf{b}}\Tilde{B}(c_2-a_1+b_1+b_2)\quad(\because\eqref{it4'})\\
&=4\pi ic_{\mathsf{b}}\Tilde{B}\left(\frac{1}{2}-a_2-a_1+b_1\right)=4\pi ic_{\mathsf{b}}\Tilde{B}\left(\frac{1}{2}-2b_1-c_1-a_1+b_1\right)\quad(\because\eqref{it2''})\\
&=4\pi ic_{\mathsf{b}}\Tilde{B}\left(\frac{1}{2}-a_1-b_1-c_1\right)=0,
\end{split}
\end{equation*}
the first-order terms concerning $\Tilde{C}$ with constant coefficients in the exponent of $e$ are
\begin{equation*}
\begin{split}
&\quad 4\pi ic_{\mathsf{b}}\Tilde{C}(-c_3-2b_5-2(b_3+c_3)+3(c_4+b_4)-(b_2+c_2))\\
&=4\pi ic_{\mathsf{b}}\Tilde{C}\left(-c_3-2b_5-2\left(\frac{1}{2}-a_3\right)+3\left(\frac{1}{2}-a_4\right)-\left(\frac{1}{2}-a_2\right)\right)\\
&=4\pi ic_{\mathsf{b}}\Tilde{C}(-c_3-2b_5+2a_3-3a_4+a_2),
\end{split}
\end{equation*}
the first-order terms concerning $\Tilde{D}$ with constant coefficients in the exponent of $e$ are
\begin{equation*}
\begin{split}
&\quad 4\pi ic_{\mathsf{b}}\Tilde{D}(-c_4+b_5+3(b_3+c_3)-3(c_4+b_4))\\
&=4\pi ic_{\mathsf{b}}\Tilde{D}\left(-c_4+b_5+3\left(\frac{1}{2}-a_3\right)-3\left(\frac{1}{2}-a_4\right)\right)\\
&=4\pi ic_{\mathsf{b}}\Tilde{D}(-c_4+b_5-3a_3+3a_4).
\end{split}
\end{equation*}
Since
\begin{equation*}
\begin{split}
&\quad -c_3-2b_5+2a_3-3a_4+a_2-c_4+b_5-3a_3+3a_4\\
&=-c_3-c_4-b_5-a_3+a_2=-c_3-(a_2+b_3)-a_3+a_2\quad(\because\eqref{it3'})\\
&=-a_3-b_3-c_3=-\frac{1}{2},
\end{split}
\end{equation*}
let
\begin{equation*}
\lambda\coloneqq -c_3-2b_5+2a_3-3a_4+a_2 ,    
\end{equation*} 
then
\begin{equation*}
-c_4+b_5-3a_3+3a_4=-\lambda-\frac{1}{2} .     
\end{equation*} 
Therefore, 
\begin{equation*}
\begin{split}
Z_{\hbar}(X,\alpha_X)&=\int_{\mathscr{Y}}d\Tilde{A}d\Tilde{B}d\Tilde{C}d\Tilde{D}\ e^{-\frac{\pi}{2}i}\frac{1}{\Phi_{\mathsf{b}}(\Tilde{A})\Phi_{\mathsf{b}}(\Tilde{B})\Phi_{\mathsf{b}}(\Tilde{C})\Phi_{\mathsf{b}}(\Tilde{D})\Phi_{\mathsf{b}}(2\Tilde{C}-\Tilde{D})}e^{2\pi ic_{\mathsf{b}}\Tilde{A}}\\
&\quad\times e^{4\pi ic_{\mathsf{b}}\lambda\Tilde{C}}e^{4\pi ic_{\mathsf{b}}(-\lambda-\frac{1}{2})\Tilde{D}} e^{\pi i\Tilde{B}^2}e^{2\pi i(\Tilde{A}^2+\Tilde{A}\Tilde{B}-\Tilde{C}^2+3\Tilde{C}\Tilde{D}-\Tilde{B}\Tilde{C})}e^{-3\pi i\Tilde{D}^2}\times\mbox{\textcircled{\scriptsize 1}},
\end{split}
\end{equation*}
where
\begin{equation*}
\begin{split}
\mbox{\textcircled{\scriptsize 1}}&\coloneqq e^{4\pi ic_{\mathsf{b}}^2c_1(c_1+b_1)}e^{-\pi ic_{\mathsf{b}}^2\frac{4(c_1-b_1)+1}{6}}e^{4\pi ic_{\mathsf{b}}^2b_2(b_2+c_2)}e^{-\pi ic_{\mathsf{b}}^2\frac{4(b_2-c_2)+1}{6}}
 e^{4\pi ic_{\mathsf{b}}^2c_3(c_3+b_3)}\\
 &\quad \times e^{-\pi ic_{\mathsf{b}}^2\frac{4(c_3-b_3)+1}{6}}e^{4\pi ic_{\mathsf{b}}^2c_4(c_4+b_4)}
  e^{-\pi ic_{\mathsf{b}}^2\frac{4(c_4-b_4)+1}{6}}e^{4\pi ic_{\mathsf{b}}^2b_5(b_5+a_5)}e^{-\pi ic_{\mathsf{b}}^2\frac{4(b_5-a_5)+1}{6}}\\
&\quad \times e^{-8\pi ic_{\mathsf{b}}^2c_1(c_1+b_1)}e^{-8\pi ic_{\mathsf{b}}^2b_2(b_2+c_2)}e^{-8\pi ic_{b}^2c_3(c_3+b_3)}e^{-8\pi ic_{\mathsf{b}}^2c_4(c_4+b_4)}e^{-8\pi c_{\mathsf{b}}^2b_5(a_5+b_5)}\\
&\quad \times e^{4\pi ic_{\mathsf{b}}^2(b_2+c_2)^2}
e^{8\pi ic_{\mathsf{b}}^2(c_1+b_1)^2}e^{8\pi ic_{\mathsf{b}}^2(b_1+c_1)(b_2+c_2)}
e^{-8\pi ic_{\mathsf{b}}^2(b_3+c_3)^2}e^{24\pi ic_{\mathsf{b}}^2(b_3+c_3)(b_4+c_4)}\\
&\quad \times e^{-8\pi ic_{\mathsf{b}}^2(b_2+c_2)(b_3+c_3)}
e^{-12\pi ic_{\mathsf{b}}^2(c_4+b_4)^2}
\end{split}
\end{equation*}
and let 
\begin{equation*}
\mathscr{Y}\coloneqq (\mathbb{R}-2c_{\mathsf{b}}(c_1+b_1))\times(\mathbb{R}-2c_{\mathsf{b}}(b_2+c_2))\times(\mathbb{R}-2c_{\mathsf{b}}(c_3+b_3))\times(\mathbb{R}-2c_{\mathsf{b}}(c_4+b_4)).
\end{equation*}

By Proposition \ref{property-of-Phi_b}(1), since 
\begin{equation*}
\frac{e^{i\pi\Tilde{B}^2}}{\Phi_{\mathsf{b}}(\Tilde{B})}=e^{i\pi\frac{1+2c_{\mathsf{b}}^2}{6}}\Phi_{\mathsf{b}}(-\Tilde{B}),
\end{equation*} 
\begin{equation}
\begin{split}
Z_{\hbar}(X,\alpha_X)&=e^{-\frac{\pi}{3}i}\int_{\mathscr{Y}}d\Tilde{A}d\Tilde{B}d\Tilde{C}d\Tilde{D}\frac{1}{\Phi_{\mathsf{b}}(\Tilde{A})}e^{\frac{i}{3}\pi c_{\mathsf{b}}^2}\Phi_{\mathsf{b}}(-\Tilde{B})\frac{1}{\Phi_{\mathsf{b}}(\Tilde{C})\Phi_{\mathsf{b}}(\Tilde{D})\Phi_{\mathsf{b}}(2\Tilde{C}-\Tilde{D})}\\
&\quad\times e^{2\pi ic_{\mathsf{b}}\Tilde{A}}e^{4\pi ic_{\mathsf{b}}\lambda\Tilde{C}} e^{4\pi ic_{\mathsf{b}}(-\lambda-\frac{1}{2})\Tilde{D}}e^{2\pi i(\Tilde{A}^2+\Tilde{A}\Tilde{B}-\Tilde{C}^2+3\Tilde{C}\Tilde{D}-\Tilde{B}\Tilde{C})}e^{-3\pi i\Tilde{D}^2}\times\mbox{\textcircled{\scriptsize 1}}. 
\end{split}\label{partition-fun-ideal-for-volume}
\end{equation}
Let
\begin{equation*}
\Tilde{B}^\prime\coloneqq -\Tilde{B},
\end{equation*} 
then
\begin{equation*}
\begin{split}
Z_{\hbar}(X,\alpha_X)&=e^{-\frac{\pi}{3}i}\int_{\Tilde{\mathscr{Y}}} d\Tilde{A}d\Tilde{B}^\prime d\Tilde{C}d\Tilde{D}\frac{1}{\Phi_{\mathsf{b}}(\Tilde{A})}e^{\frac{i}{3}\pi c_{\mathsf{b}}^2}\Phi_{\mathsf{b}}(\Tilde{B}^\prime)\frac{1}{\Phi_{\mathsf{b}}(\Tilde{C})\Phi_{\mathsf{b}}(\Tilde{D})\Phi_{\mathsf{b}}(2\Tilde{C}-\Tilde{D})}\\
&\quad\times e^{2\pi ic_{\mathsf{b}}\Tilde{A}}e^{4\pi ic_{\mathsf{b}}\lambda\Tilde{C}}
 e^{4\pi ic_{\mathsf{b}}(-\lambda-\frac{1}{2})\Tilde{D}}e^{2\pi i(\Tilde{A}^2-\Tilde{A}\Tilde{B}^\prime-\Tilde{C}^2+3\Tilde{C}\Tilde{D}+\Tilde{B}^\prime\Tilde{C})}e^{-3\pi i\Tilde{D}^2}\times\mbox{\textcircled{\scriptsize 1}},
\end{split}
\end{equation*}
where
\begin{equation*}
\Tilde{\mathscr{Y}}\coloneqq (\mathbb{R}-2c_{\mathsf{b}}(c_1+b_1))\times(\mathbb{R}+2c_{\mathsf{b}}(b_2+c_2))\times(\mathbb{R}-2c_{\mathsf{b}}(c_3+b_3))\times(\mathbb{R}-2c_{\mathsf{b}}(c_4+b_4)).
\end{equation*}

Transformation of variables, 
$\left\{
\begin{aligned}
&x\coloneqq \Tilde{C}-\Tilde{D}\\
&Y\coloneqq \Tilde{A}\\
&Z\coloneqq \Tilde{B}^\prime\\
&W\coloneqq \Tilde{C}\\
\end{aligned}
\right.$, 
gives
$\left\{
\begin{aligned}
&\Tilde{A}=Y\\
&\Tilde{B}^\prime=Z\\
&\Tilde{C}=W\\
&\Tilde{D}=W-x\\
\end{aligned}
\right.$
and its Jacobian is 
$
\begin{vmatrix}
0 & 1 & 0 &0\\
0 & 0& 1 & 0\\
0 & 0 & 0 & 1\\
-1 & 0 & 0 & 1 
\end{vmatrix}
=1$. Therefore,  
\begin{equation*}
\begin{split}
Z_{\hbar}(X,\alpha_X)&=e^{-\frac{\pi}{3}i}\int_{\Tilde{\mathscr{Y}}^\prime} dx dY dZ dW e^{\frac{i}{3}\pi c_{\mathsf{b}}^2}\frac{\Phi_{\mathsf{b}}(Z)}{\Phi_{\mathsf{b}}(Y)\Phi_{\rm b}(W)\Phi_{\mathsf{b}}(W-x)\Phi_{\mathsf{b}}(W+x)}\\
&\quad\times  e^{4\pi ic_{\mathsf{b}}(\lambda+\frac{1}{2}) x}
e^{2\pi ic_{\mathsf{b}}(Y-W)}e^{\pi i(-3x^2+2Y^2+W^2-2YZ+2ZW)}\times\mbox{\textcircled{\scriptsize 1}}, 
\end{split}
\end{equation*}
where
\begin{equation*}
\Tilde{\mathscr{Y}}^\prime\coloneqq(\mathbb{R}+2c_{\mathsf{b}}(c_4+b_4-c_3-b_3))\times(\mathbb{R}-2c_{\mathsf{b}}(c_1+b_1))\times(\mathbb{R}+2c_{\mathsf{b}}(b_2+c_2))\times(\mathbb{R}-2c_{\mathsf{b}}(c_3+b_3)),
\end{equation*}
and $(x,Y,Z,W)\in\Tilde{\mathscr{Y}}^\prime$.

Since the integrand of $J_{X}(\hbar,X)$ rapidly decreases if $\alpha_X\in\mathscr{A}_X$ by the property of Proposition \ref{property-of-Phi_b} (3), the value of $J_{X}(\hbar,x)$ does not depend on the choice of $\alpha_X\in\mathscr{A}_X$ by the Bochner-Martinelli formula \cite{MR635928}, which generalizes Cauchy's integral theorem.

Then 
\begin{equation*}
\begin{split}
Z_{\hbar}(X,\alpha_X)&=e^{-\frac{\pi}{3}i}e^{\frac{i}{3}\pi c_{\mathsf{b}}^2}\times\mbox{\textcircled{\scriptsize 1}}\times\int_{\mathbb{R}+\frac{i\mu(\alpha_X)}{\sqrt{\hbar}}} dx J_X(\hbar,x) e^{-\frac{x}{\sqrt{\hbar}}\lambda(\alpha_X)}, \\
\end{split}
\end{equation*}
where
\begin{eqnarray*}
\mu(\alpha_X)&=&a_3-a_4\\
\lambda(\alpha_X)&=&-2\pi(-c_4+b_5-3a_3+3a_4).\\
\end{eqnarray*}

Let
\begin{equation*}
\lambda^\prime\coloneqq\lambda+\frac{1}{2}.
\end{equation*}
We check the gauge invariance of $\lambda^\prime$.  
Let $(\Tilde{X},\alpha_{\Tilde{X}})$ be a shaped pseudo $3$-manifold which is gauge equivalent to $(X,\alpha_X)$ and
\begin{equation*}
g: \Delta_{1}(X)\longrightarrow \mathbb{R}
\end{equation*}
be the map in Definition \ref{gauge-equivalent} associated with the gauge equivalence.
Let 
$\alpha_{\Tilde{X}}=(2\pi\Tilde{a}_1,2\pi\Tilde{b}_1,2\pi\Tilde{c}_1,\ldots,2\pi\Tilde{a}_5,2\pi\Tilde{b}_5,2\pi\Tilde{c}_5)$. Since
\begin{equation*}
\begin{split}
&-\Tilde{c}_3-2\Tilde{b}_5+2\Tilde{a}_3-3\Tilde{a}_4+\Tilde{a}_2\\
=&-\left(c_3+\frac{1}{2}(g(e_4)-g(e_3)+g(e_1)-g(e_2))\right)-2\left(b_5+\frac{1}{2}(g(e_2)-g(e_1)+g(e_3)-g(e_3))\right)\\
&+2\left(a_3+\frac{1}{2}(g(e_2)-g(e_2)+g(e_3)-g(e_5))\right)-3\left(a_4+\frac{1}{2}(g(e_1)-g(e_1)+g(e_3)-g(e_5))\right)\\
&+\left(a_2+\frac{1}{2}(g(e_2)-g(e_5)+g(e_4)-g(e_1))\right)\\
=&-c_3-2b_5+2a_3-3a_4+a_2, 
\end{split}
\end{equation*}
$\lambda$ is gauge invariant. Therefore $\lambda^\prime=\lambda+\frac{1}{2}$ and  $\lambda(\alpha_X)$ are also gauge invariant.\\
Since
\begin{equation*}
\begin{split}
&\quad \Tilde{a}_3-\Tilde{a}_4\\
&=\left(a_3+\frac{1}{2}\left(g(e_2)-g(e_2)+g(e_3)-g(e_5)\right)\right)-\left(a_4+\frac{1}{2}\left(g(e_1)-g(e_1)+g(e_3)-g(e_5)\right)\right)\\
&=a_3-a_4,
\end{split}
\end{equation*} 
$\mu(\alpha_X)$ is gauge invariant.
\end{proof}
\section{Calculation of the partition function for one vertex H-triangulation of $(S^3, 7_3)$}
\label{one-vertex_Htriangulation}

In this section, we calculate the partition function for one vertex H-triangulation $Y$ of ($S^3$, $7_3$) obtained from Proposition \ref{one_veretex_H-triantulation} in Section \ref{triangulation}.
Since $Y-\Delta_0(Y)$ is homeomorphic to the space $S^3$ minus one point, 
\begin{equation*}
H_2(Y-\Delta_0(Y),\mathbb{Z})\cong H_2(S^{3}\backslash\text{point},\mathbb{Z})\cong H_2(\mathbb{R}^3,\mathbb{Z})\cong H_2(\text{point},\mathbb{Z}) =0.
\end{equation*}

Therefore, $Y$ is admissible.

Let $\alpha_Y=(2\pi a_1,2\pi b_1,2\pi c_1,\ldots,2\pi a_6,2\pi b_6,2\pi c_6)\in\mathscr{S}_Y$.

From Figure \ref{fig:one_vertex_Htriangulation}, the weight of each edge divided by $2\pi$ is obtained as follows, 
\begin{align*}
e_0:&\quad & a_1\\
e_1:&\quad & a_1+a_2+c_2+a_3+c_3+a_4\\
e_2:&\quad & b_1+c_1+b_2+c_2+c_3+c_6\\
e_3:&\quad & b_1+c_1+b_4+c_4+a_5\\
e_4:&\quad & a_3+b_4+a_5+b_5+a_6+c_6\\
e_5:&\quad & b_2+b_3+a_4+b_5+c_5+a_6+b_6\\
e_6:&\quad & a_2+b_3+c_4+c_5+b_6 .
\end{align*}

Let $\tau=(2\pi a_1^\tau,2\pi b_1^\tau,2\pi c_1^\tau,\cdots,2\pi a_6^\tau,2\pi b_6^\tau, 2\pi c_6^\tau)\in\overline{\mathscr{S}_{T_1}}\times\mathscr{S}_{Y\backslash T_1}$ be the extended shape structure which gives the weight function $\omega_{Y,\tau}:\Delta_1(Y)\longrightarrow\mathbb{R}$ such that each weight on $e_1,e_2,\cdots,e_6$ is $2\pi$, and the weight on $e_0$ is $0$.

Then 
\begin{align}
a_1^\tau+a_2^\tau+c_2^\tau+a_3^\tau+c_3^\tau+a_4^\tau=1\label{da1}\\
b_1^\tau+c_1^\tau+b_2^\tau+c_2^\tau+c_3^\tau+c_6^\tau=1\label{da2}\\
b_1^\tau+c_1^\tau+b_4^\tau+c_4^\tau+a_5^\tau=1\label{da3}\\
a_3^\tau+b_4^\tau+a_5^\tau+b_5^\tau+a_6^\tau+c_6^\tau=1\label{da4}\\
b_2^\tau+b_3^\tau+a_4^\tau+b_5^\tau+c_5^\tau+a_6^\tau+b_6^\tau=1\label{da5}\\
a_2^\tau+b_3^\tau+c_4^\tau+c_5^\tau+b_6^\tau=1\label{da6}\\
a_1^\tau=0.\label{da7}
\end{align}

Using $a_i^\tau+b_i^\tau+c_i^\tau=\frac{1}{2}$ for $i=1,2,\ldots,6$,
\begin{empheq}[left={\eqref{da1},\eqref{da2},\eqref{da3},\eqref{da4},\eqref{da5},\eqref{da6},\eqref{da7}\iff\empheqlbrace}]{align}
&a_3^\tau+b_4^\tau=c_5^\tau+b_6^\tau \label{da1''}\\
&a_4^\tau=b_2^\tau+b_3^\tau=a_5^\tau \label{da2''}\\
&a_4^\tau=c_6^\tau \label{da3''}\\
&a_2^\tau=c_3^\tau+a_4^\tau \label{da4''}\\
&a_1^\tau=0.\label{da5''}
\end{empheq}

The following holds in the same way as in \cite{MR3945172} Lemma 6.2.
\begin{lemma}\label{upper-bound-of-Phi_b}
Let $\hbar >0$, $\delta\in(0,\frac{1}{4})$. Then 
$M_{\delta,\hbar}\coloneqq\max_{z\in\mathbb{R}+\frac{i}{\sqrt{\hbar}}[\delta,\frac{1}{2}-\delta]}|\Phi_{\mathsf{b}}(z)|$ is finite.
\end{lemma}
\begin{proof}
Let $\hbar > 0$, $\delta\in(0,\frac{1}{4})$. Suppose $M_{\delta,\hbar}=\infty$. Then there exists a sequence of points $(z_n)_{n\in\mathbb{N}}\in\left(\mathbb{R}+\frac{i}{\sqrt{\hbar}}[\delta,\frac{1}{2}-\delta]\right)^{\mathbb{N}}$ such that $\left|\Phi_{\mathsf{b}}(z_n)\right|\underset{n\rightarrow\infty}{\rightarrow}\infty$.
If $(\Re(z_n))_{n\in\mathbb{N}}$ is bounded, $(z_n)_{n\in\mathbb{N}}$ is on the compact set and by the continuity of $\left|\Phi_{\mathsf{b}}\right|$, $(|\Phi_{\mathsf{b}}(z_n)|)_{n\in\mathbb{N}}$ is bounded, which is a contradiction. 
Therefore there exists a subsequence tending to $-\infty$ or $\infty$ in $(\Re(z_n))_{n\in\mathbb{N}}$, and the images of this sequence by $\left|\Phi_{\mathsf{b}}\right|$ must tend to $\infty$, which contradicts Proposition \ref{property-of-Phi_b}(3).
\end{proof}
Below we calculate the partition function $Z_\hbar(Y,\alpha_Y)$ of $(Y,\alpha_Y)$. 
\begin{theorem}\label{volume-conjecture-2}
The following holds for the partition function $Z_\hbar(Y,\alpha_Y)$ of $(Y,\alpha_Y)$.
\begin{equation*}
\lim_{\alpha_Y\rightarrow \tau}\Phi_{\mathsf{b}}\left(\frac{\pi- \omega_{Y,\alpha_Y}(e_0)}{2\pi i\sqrt{\hbar}}\right)Z_\hbar(Y,\alpha_Y)=e^{-\frac{5}{12}\pi i}e^{-\frac{i\varphi(\tau)}{\hbar}}J_{X}(\hbar,0),
\end{equation*}
where $\varphi(\tau)$ is a real quadratic polynomial of dihedral angles and $J_{X}(\hbar,x)$ is the function defined in Section \ref{ideal-triangulation}.
\end{theorem}
\begin{proof}
\begin{equation*}
\begin{split}
Z_\hbar(Y,\alpha_Y)&=\int_{\mathbb{R}^{12}}dxdydzdwdpdqdrdsdtdudvda\ \langle x,z\lvert \mathsf{T}(a_1,c_1)\rvert x,y\rangle \langle p,y\lvert \mathsf{T}(a_2,c_2)\rvert q,w\rangle \\ 
& \quad\quad\times \langle s,w\lvert \overline{\mathsf{T}}(a_3,c_3)\rvert r,p\rangle\langle u,r\lvert \mathsf{T}(a_4,c_4)\rvert t,z\rangle \langle v,t\lvert \mathsf{T}(a_5,c_5)\rvert a,u\rangle \\
& \quad\quad\times \langle a,q\lvert \mathsf{T}(a_6,c_6)\rvert s,v\rangle \\
&=\int_{\mathbb{R}^{12}}dxdydzdwdpdqdrdsdtdudvda\ \delta(z)\psi_{c_1,b_1}(y-z)e^{2\pi ix(y-z)}e^{-\frac{\pi}{12}i}\\
&\quad\quad\times \delta(p+y-q)\psi_{c_2,b_2}(w-y) e^{2\pi ip(w-y)}e^{-\frac{\pi}{12}i}\delta(r+p-s)\psi_{b_3,c_3}(p-w)\\
&\quad\quad\times e^{\pi i (p-w)^2}e^{-2\pi ir(w-p)}e^{-\frac{\pi}{12}i}\delta(u+r-t)
 \psi_{c_4,b_4}(z-r)e^{2\pi i u(z-r)}e^{-\frac{\pi}{12}i}\\
 & \quad\quad\times \delta(v+t-a)\psi_{c_5,b_5}(u-t)e^{2\pi i v(u-t)}e^{-\frac{\pi}{12}i}\delta(a+q-s)\psi_{c_6,b_6}(v-q)\\
&\quad\quad\times e^{2\pi i a(v-q)}e^{-\frac{\pi}{12}i}.
\end{split}
\end{equation*}

Iterated integration with respect to the variables $z, u, t, a, s$, and $q$, in turn, gives
\begin{equation*}
\begin{split}
Z_\hbar(Y,\alpha_Y)&=\int_{\mathbb{R}^{6}}dxdydwdpdrdv\ e^{-\frac{\pi}{2}}\psi_{c_1,b_1}(y)\psi_{c_2,b_2}(w-y)\psi_{b_3,c_3}(p-w)\psi_{c_4,b_4}(-r)\\
&\quad\quad\times \psi_{c_5,b_5}(-r)\psi_{c_6,b_6}(v-p-y)
 e^{\pi i(p^2+w^2+2y^2+2xy-2rw+2rv-2yv)}.
\end{split} 
\end{equation*}
Since
\begin{equation*}
\delta(y)=\int_{\mathbb{R}}\ e^{2\pi ixy}dx,   
\end{equation*}
\begin{equation*}
\begin{split}
Z_\hbar(Y,\alpha_Y)&=\int_{\mathbb{R}^5}dydwdpdrdv\ e^{-\frac{\pi}{2}i}\delta(y)
\psi_{c_1,b_1}(y)\psi_{c_2,b_2}(w-y)\psi_{b_3,c_3}(p-w)\psi_{c_4,b_4}(-r)\\
&\quad\quad\times\psi_{c_5,b_5}(-r)\psi_{c_6,b_6}(v-p-y)e^{\pi i(p^2+w^2+2y^2-2rw+2rv-2yv)}\\
&=\int_{\mathbb{R}^4}dwdpdrdv\ e^{-\frac{\pi}{2}i}\psi_{c_1,b_1}(0)\psi_{c_2,b_2}(w)\psi_{b_3,c_3}(p-w)\psi_{c_4,b_4}(-r)\psi_{c_5,b_5}(-r)\\
&\quad\quad\times \psi_{c_6,b_6}(v-p)e^{\pi i(p^2+w^2-2rw+2rv)}.
\end{split}    
\end{equation*}
Transformation of variables, 
$\left\{
\begin{aligned}
&A\coloneqq w\\
&B\coloneqq -r\\
&C\coloneqq v-p\\
&D\coloneqq p-w\\
\end{aligned}
\right.$,
gives
$\left\{
\begin{aligned}
&w=A\\
&p=A+D\\
&r=-B\\
&v=A+C+D\\
\end{aligned}
\right.$
and its Jacobian is 
$
\begin{vmatrix}
1 & 0 & 0 & 0\\
1 & 0& 0 & 1 \\
0 & -1 & 0 & 0 \\
1 & 0 & 1 & 1
\end{vmatrix}
=-1
$. Therefore, 
\begin{equation*}
\begin{split}
Z_{\hbar}(Y,\alpha_Y)&=\int_{\mathbb{R}^4}dAdBdCdD\ e^{-\frac{\pi}{2}i}
\psi_{c_1,b_1}(0)\psi_{c_2,b_2}(A)\psi_{b_3,c_3}(D)\psi_{c_4,b_4}(B)
\psi_{c_5,b_5}(B)\\
&\quad\quad\times \psi_{c_6,b_6}(C)e^{\pi i(2A^2-2BC-2BD+2AD+D^2)}.
\end{split}
\end{equation*}
Since
\begin{equation*}
\int_{\mathbb{R}}dC\ \psi_{c_6,b_6}(C)e^{-2\pi iBC}=\widetilde{\psi}_{c_6,b_6}(B)=e^{\pi iB^2}e^{-\frac{\pi}{12}i}\psi_{b_6,a_6}(B),
\end{equation*}
\begin{equation*}
\begin{split}
Z_{\hbar}(Y,\alpha_Y)&=\int_{\mathbb{R}^3}dAdBdD\ e^{-\frac{7}{12}\pi i}\psi_{c_1,b_1}(0)\psi_{c_2,b_2}(A)\psi_{b_3,c_3}(D)\psi_{c_4,b_4}(B)\psi_{c_5,b_5}(B)\\ &\quad\quad\times \psi_{b_6,a_6}(B)e^{\pi i(2A^2+B^2-2BD+D^2+2AD)}\\
&=e^{-\frac{7}{12}\pi i}\int_{\mathbb{R}^3}dAdBdD \psi(-2c_{\mathsf{b}}(c_1+b_1))e^{-4\pi i c_{\mathsf{b}}c_1(-c_{\mathsf{b}}(c_1+b_1))}e^{-\pi ic_{\mathsf{b}}^2\frac{(4(c_1-b_1)+1)}{6}}\\
&\quad\times \psi(A-2c_{\mathsf{b}}(c_2+b_2))e^{-4\pi ic_{\mathsf{b}}c_2(A-c_{\mathsf{b}}(c_2+b_2))}e^{-\pi i c_{\mathsf{b}}^2\frac{(4(c_2-b_2)+1)}{6}}\psi(D-2c_{\mathsf{b}}(b_3+c_3))\\
&\quad\times e^{-4\pi ic_{\mathsf{b}}b_3(D-c_{\mathsf{b}}(b_3+c_3))}e^{-\pi ic_{\mathsf{b}}^2\frac{(4(b_3-c_3)+1)}{6}}\psi(B-2c_{\mathsf{b}}(c_4+b_4))e^{-4\pi ic_{\mathsf{b}}c_4(B-c_{\mathsf{b}}(c_4+b_4))}\\
&\quad\times e^{-\pi i c_{\mathsf{b}}^2\frac{(4(c_4-b_4)+1)}{6}}\psi(B-2c_{\mathsf{b}}(c_5+b_5))e^{-4\pi ic_{\mathsf{b}}c_5(B-c_{\mathsf{b}}(c_5+b_5))}e^{-\pi ic_{\mathsf{b}}^2\frac{4(c_5-b_5)+1}{6}}\\
&\quad\times \psi(B-2c_{\mathsf{b}}(b_6+a_6))e^{-4\pi ic_{\mathsf{b}}b_6(B-c_{\mathsf{b}}(b_6+a_6))}e^{-\pi ic_{\mathsf{b}}^2\frac{4(b_6-a_6)+1}{6}}\\
&\quad\times e^{\pi i(2A^2+B^2-2BD+D^2+2AD)}\\
&=e^{-\frac{7}{12}\pi i}\int_{\mathbb{R}^3}dAdBdD\frac{1}{\Phi_{\mathsf{b}}(-2c_{\mathsf{b}}(c_1+b_1))}\frac{1}{\Phi_{\mathsf{b}}(A-2c_{\mathsf{b}}(c_2+b_2))}
e^{-4\pi ic_{\mathsf{b}}c_2A}\\
&\quad\times \frac{1}{\Phi_{\mathsf{b}}(B-2c_{\mathsf{b}}(b_4+c_4))}e^{-4\pi ic_{\mathsf{b}}c_4B}\frac{1}{\Phi_{\mathsf{b}}(B-2c_{\mathsf{b}}(b_5+c_5))}
e^{-4\pi ic_{\mathsf{b}}c_5B}\\
&\quad \times \frac{1}{\Phi_{\mathsf{b}}(B-2c_{\mathsf{b}}(a_6+b_6))}
e^{-4\pi ic_{\mathsf{b}}b_6B}\frac{1}{\Phi_{\mathsf{b}}(D-2c_{\mathsf{b}}(b_3+c_3))} e^{-4\pi ic_{\mathsf{b}}b_3D} \\
& \quad \times e^{\pi i(2A^2+B^2-2BD+D^2+2AD)}\times\mbox{\textcircled{\scriptsize 2}},
\end{split}
\end{equation*}
where
\begin{equation*}
\begin{split}
\mbox{\textcircled{\scriptsize 2}}&\coloneqq e^{4\pi i c_{\mathsf{b}}^2c_1(c_1+b_1)}
e^{-\pi ic_{\mathsf{b}}^2\frac{(4(c_1-b_1)+1)}{6}}
e^{4\pi ic_{\mathsf{b}}^2c_2(c_2+b_2)}e^{-\pi i c_{\mathsf{b}}^2\frac{(4(c_2-b_2)+1)}{6}}e^{4\pi ic_{\mathsf{b}}^2b_3(b_3+c_3)}\\
&\quad\times e^{-\pi ic_{\mathsf{b}}^2\frac{(4(b_3-c_3)+1)}{6}}e^{4\pi ic_{\mathsf{b}}^2c_4(c_4+b_4)}e^{-\pi i c_{\mathsf{b}}^2\frac{(4(c_4-b_4)+1)}{6}}  e^{4\pi ic_{\mathsf{b}}^2c_5(c_5+b_5)}\\
&\quad \times e^{-\pi i c_{\mathsf{b}}^2\frac{(4(c_5-b_5)+1)}{6}}e^{4\pi ic_{\mathsf{b}}^2b_6(b_6+a_6)}e^{-\pi i c_{\mathsf{b}}^2\frac{(4(b_6-a_6)+1)}{6}} .
\end{split}
\end{equation*}
Since
\begin{equation*}
\Phi_{\mathsf{b}}\left(\frac{\pi-2\pi a_1}{2\pi i\sqrt{\hbar}}\right)=\Phi_{\mathsf{b}}(2c_{\mathsf{b}}a_1-c_{\mathsf{b}})=\Phi_{\mathsf{b}}(-2c_{\mathsf{b}}(b_1+c_1)),
\end{equation*} 
\begin{equation*}
\begin{split}
&\quad \Phi_{\mathsf{b}}\left(\frac{\pi-2\pi a_1}{2\pi i\sqrt{\hbar}}\right)Z_\hbar(Y,\alpha_Y)\\
&=e^{-\frac{7}{12}\pi i}\int_{\mathbb{R}^3}dAdBdD \frac{1}{\Phi_{\mathsf{b}}(A-2c_{\mathsf{b}}(c_2+b_2))}
e^{-4\pi ic_{\mathsf{b}}c_2A}\frac{1}{\Phi_{\mathsf{b}}(B-2c_{\mathsf{b}}(b_4+c_4))}e^{-4\pi ic_{\mathsf{b}}c_4B}\\
&\quad\times\frac{1}{\Phi_{\mathsf{b}}(B-2c_{\mathsf{b}}(b_5+c_5))}
e^{-4\pi ic_{\mathsf{b}}c_5B}\frac{1}{\Phi_{\mathsf{b}}(B-2c_{\mathsf{b}}(a_6+b_6))}e^{-4\pi ic_{\mathsf{b}}b_6B}\frac{1}{\Phi_{\mathsf{b}}(D-2c_{\mathsf{b}}(b_3+c_3))}\\
&\quad\times e^{-4\pi ic_{\mathsf{b}}b_3D} e^{\pi i(2A^2+B^2-2BD+D^2+2AD)}\times\mbox{\textcircled{\scriptsize 2}}.
\end{split}
\end{equation*}

Take $\delta>0$ such that there exists a neighborhood $\mathfrak{U}$ of $\tau$ in $\overline{\mathscr{S}_{T_1}}\times\mathscr{S}_{Y\backslash T_1}$ such that for each $\alpha_Y\in\mathfrak{U}\cap \mathscr{S}_Y$, the values of $15$ variables $a_2,b_2,c_2,\ldots,a_6,b_6,c_6$ of $\alpha_Y$ are in $(\delta,\frac{1}{2}-\delta)$.
Then for every $\alpha_Y\in\mathfrak{U}\cap\mathscr{S}_Y$, for any $j\in\{2,\ldots,6\}$, any $t\in\mathbb{R}$,
\begin{equation*}
\left|e^{-4\pi ic_{\mathsf{b}}c_jt}\Phi_{\mathsf{b}}\left(t\pm 2c_{\mathsf{b}}(b_j+c_j)\right)^{\pm 1}\right|=\left|e^{\frac{2\pi}{\sqrt{\hbar}}c_jt}\Phi_{\mathsf{b}}\left(t\pm\frac{i}{\sqrt{\hbar}}(b_j+c_j)\right)^{\pm 1}\right| \le M_{\delta,\hbar}e^{-\frac{2\pi}{\sqrt{\hbar}}\delta|t|}
\end{equation*}
holds. In addition, the above inequality holds even if we replace the letters $(b_j, c_j)\rightarrow (c_j, b_j), (b_j, c_j) \rightarrow(a_j, b_j)$ and so on.
In fact, for $t\le 0$ it follows by Lemma \ref{upper-bound-of-Phi_b}, Proposition \ref{property-of-Phi_b} (2), $c_j >\delta$, and $\delta<b_j+c_j=\frac{1}{2}-a_j<\frac{1}{2}-\delta$.
For $t\ge 0$, it follows by $b_j>\delta$ and by the fact
\begin{equation*}
\begin{split}
&\left|\Phi_{\mathsf{b}}\left(t+\frac{i}{\sqrt{\hbar}}(b_j+c_j)\right)\right|
=\left|\Phi_{\mathsf{b}}\left(-t+\frac{i}{\sqrt{\hbar}}(b_j+c_j)\right)\right|\left|e^{i\pi(t+\frac{i}{\sqrt{\hbar}}(b_j+c_j))^2}\right| \\
&\le M_{\delta,\hbar}e^{-\frac{2\pi}{\sqrt{\hbar}}(b_j+c_j)t}
\end{split}
\end{equation*}
because of Proposition \ref{property-of-Phi_b} (1), (2).

Therefore, for every $\alpha_Y\in\mathfrak{U}\cap\mathscr{S}_Y$, every $(A,B,D)\in\mathbb{R}^3$,
\begin{equation*}
\begin{split}
&\left|\frac{1}{\Phi_{\mathsf{b}}(A-2c_{\mathsf{b}}(c_2+b_2))}
e^{-4\pi ic_{\mathsf{b}}c_2A}\frac{1}{\Phi_{\mathsf{b}}(B-2c_{\mathsf{b}}(b_4+c_4))}e^{-4\pi ic_{\mathsf{b}}c_4B}\frac{1}{\Phi_{\mathsf{b}}(B-2c_{\mathsf{b}}(b_5+c_5))}\right.
\\
&\times e^{-4\pi ic_{\mathsf{b}}c_5B}\frac{1}{\Phi_{\mathsf{b}}(B-2c_{\mathsf{b}}(a_6+b_6))}
e^{-4\pi ic_{\mathsf{b}}b_6B}\frac{1}{\Phi_{\mathsf{b}}(D-2c_{\mathsf{b}}(b_3+c_3))} e^{-4\pi ic_{\mathsf{b}}b_3D}\\
&\times\left.e^{\pi i(2A^2+B^2-2BD+D^2+2AD)}\times \mbox{\textcircled{\scriptsize 2}}\right|\le (M_{\delta,\hbar})^5 e^{-\frac{2\pi}{\sqrt{\hbar}}\delta(|A|+3|B|+|D|)}
\end{split}
\end{equation*}
holds.

Since the right-hand side of this inequality is integrable in $\mathbb{R}^3$, because of the dominated convergence theorem, it follows that 
\begin{equation*}
\begin{split}
&\lim_{\alpha_Y\rightarrow \tau}\Phi_{\mathsf{b}}\left(\frac{\pi- \omega_{Y,\alpha_Y}(e_0)}{2\pi i\sqrt{\hbar}}\right)Z_\hbar(Y,\alpha_Y)=e^{-\frac{7}{12}\pi i}\int_{\mathbb{R}^3}dAdBdD \frac{1}{\Phi_{\mathsf{b}}(A-2c_{\mathsf{b}}(c_2^\tau+b_2^\tau))}\\
&\quad\times e^{-4\pi ic_{\mathsf{b}}c_2^\tau A}\frac{1}{\Phi_{\mathsf{b}}(B-2c_{\mathsf{b}}(b_4^\tau+c_4^\tau))} e^{-4\pi ic_{\mathsf{b}}c_4^\tau B}\frac{1}{\Phi_{\mathsf{b}}(B-2c_{\mathsf{b}}(b_5^\tau+c_5^\tau))}
e^{-4\pi ic_{\mathsf{b}}c_5^\tau B}\\
&\quad \times \frac{1}{\Phi_{\mathsf{b}}(B-2c_{\mathsf{b}}(a_6^\tau+b_6^\tau))}
e^{-4\pi ic_{\mathsf{b}}b_6^\tau B}\frac{1}{\Phi_{\mathsf{b}}(D-2c_{\mathsf{b}}(b_3^\tau +c_3^\tau))}e^{-4\pi ic_{\mathsf{b}}b_3^\tau D} \\
&\quad\times e^{\pi i(2A^2+B^2-2BD+D^2+2AD)}\times\left.\mbox{\textcircled{\scriptsize 2}}\right|_{\alpha_Y=\tau}.
\end{split}
\end{equation*}

Transformation of variables,
\begin{equation*}
\begin{aligned}
&\Tilde{A}\coloneqq A-2c_{\mathsf{b}}(c_2^\tau+b_2^\tau)=A-2c_{\mathsf{b}}\left(\frac{1}{2}-a_2^\tau\right)\\
&\Tilde{D}\coloneqq D-2c_{\mathsf{b}}(b_3^\tau+c_3^\tau)=D-2c_{\mathsf{b}}\left(\frac{1}{2}-a_3^\tau\right)\\
&\Tilde{B}\coloneqq B-2c_{\mathsf{b}}(c_4^\tau+b_4^\tau)=B-2c_{\mathsf{b}}\left(\frac{1}{2}-a_4^\tau\right)=B-2c_{\mathsf{b}}\left(\frac{1}{2}-a_5^\tau\right)=B-2c_{\mathsf{b}}(b_5^\tau+c_5^\tau)\\
&=B-2c_{\mathsf{b}}\left(\frac{1}{2}-c_6^\tau\right)=B-2c_{\mathsf{b}}(a_6^\tau+b_6^\tau)
\end{aligned}
\end{equation*}
(We used \eqref{da2''}, \eqref{da3''} in the deformation of $\Tilde{B}$.), gives 
\begin{equation*}
\begin{split}
&\quad e^{-4\pi ic_{\mathsf{b}}c_2^\tau A}e^{-4\pi ic_{\mathsf{b}}b_3^\tau D}e^{-4\pi ic_{\mathsf{b}}c_4^\tau B}e^{-4\pi ic_{\mathsf{b}}c_5^\tau B}e^{-4\pi ic_{\mathsf{b}}b_6^\tau B}e^{\pi i\left(2A^2+(B-D)^2+2AD\right)}\\
&=e^{-4\pi ic_{\mathsf{b}}c_2^\tau \left\{\Tilde{A}+2c_{\mathsf{b}}\left(\frac{1}{2}-a_2^\tau \right)\right\}}e^{-4\pi ic_{\mathsf{b}}b_3\left\{\Tilde{D}+2c_{\mathsf{b}}\left(\frac{1}{2}-a_3^\tau\right)\right\}}e^{-4\pi ic_{\mathsf{b}}(c_4^\tau+c_5^\tau+b_6^\tau)\left\{\Tilde{B}+2c_{\mathsf{b}}\left(\frac{1}{2}-a_4^\tau\right)\right\}}\\
&\quad \times e^{\pi i\left[2\left(\Tilde{A}+2c_{\mathsf{b}}\left(\frac{1}{2}-a_2^\tau\right)\right)^2+\left\{\Tilde{B}+2c_{\mathsf{b}}\left(\frac{1}{2}-a_4^\tau\right)-\Tilde{D}-2c_{\mathsf{b}}\left(\frac{1}{2}-a_3^\tau\right)\right\}^2+2\left(\Tilde{A}+2c_{\mathsf{b}}\left(\frac{1}{2}-a_2^\tau\right)\right)\left(\Tilde{D}+2c_{\mathsf{b}}\left(\frac{1}{2}-a_3^\tau\right)\right)\right]}\\
&=e^{2\pi i\Tilde{A}^2}e^{\pi i\Tilde{B}^2}e^{\pi i\Tilde{D}^2}e^{-2\pi i\Tilde{B}\Tilde{D}}e^{2\pi i\Tilde{A}\Tilde{D}}e^{2\pi ic_{\mathsf{b}}\left(-2c_2^\tau+2\left(1-a_2^\tau\right)+1-2a_3^\tau\right)\Tilde{A}}e^{2\pi ic_{\mathsf{b}}\left(-2\left(c_4^\tau+c_5^\tau+b_6^\tau\right)-2(a_4^\tau-a_3^\tau)\right)\Tilde{B}}\\
&\quad \times e^{2\pi ic_{\mathsf{b}}\left(-2b_3^\tau+2\left(a_4^\tau-a_3^\tau\right)+1-2a_2^\tau\right)\Tilde{D}}e^{-4\pi ic_{\mathsf{b}}^2c_2^\tau(1-2a_2^\tau)}e^{-4\pi ic_{\mathsf{b}}^2b_3^\tau(1-2a_3^\tau)}e^{-4\pi ic_{\mathsf{b}}^2(c_4^\tau+c_5^\tau+b_6^\tau)(1-2a_4^\tau)}\\
&\quad \times e^{2\pi ic_{\mathsf{b}}^2(1-2a_2^\tau)^2}e^{4\pi ic_{\mathsf{b}}^2(a_4^\tau-a_3^\tau)^2}e^{2\pi ic_{\mathsf{b}}^2(1-2a_2^\tau)(1-2a_3^\tau)}.
\end{split}
\end{equation*}

Deformation of the coefficients of the first-order terms on $\tilde{A},\tilde{B},\tilde{D}$ in the exponent of $e$ using the equations from \eqref{da1''} to \eqref{da5''} gives
\begin{equation*}
-2c_2^\tau+2(1-a_2^\tau)+1-2a_3^\tau=1   
\end{equation*}
\begin{equation*}
-2(c_4^\tau+c_5^\tau+b_6^\tau+a_4^\tau-a_3^\tau)=-1
\end{equation*}
\begin{equation*}
-2b_3^\tau+2(a_4^\tau-a_3^\tau)+1-2a_2^\tau=0.
\end{equation*}

Therefore,
\begin{equation*}
\begin{split}
&\quad \lim_{\alpha_Y\rightarrow \tau}\Phi_{\mathsf{b}}\left(\frac{\pi-\omega_{Y,\alpha_Y}(e_0)}{2\pi i\sqrt{\hbar}}\right)Z_{\hbar}(Y,\alpha_Y)\\
&=e^{-\frac{7}{12}\pi i}\int_{\mathscr{Y}_\tau}d\Tilde{A}d\Tilde{B}d\Tilde{D}\frac{1}{\Phi_{\mathsf{b}}(\tilde{A})}\frac{1}{\Phi_{\mathsf{b}}^3(\tilde{B})} \frac{1}{\Phi_{\mathsf{b}}(\tilde{D})}e^{2\pi i\Tilde{A}^2}e^{\pi i\Tilde{B}^2}e^{\pi i\Tilde{D}^2}e^{-2\pi i\Tilde{B}\Tilde{D}}e^{2\pi i\Tilde{A}\Tilde{D}}e^{2\pi ic_{\mathsf{b}}\Tilde{A}}\\
&\quad \times e^{-2\pi ic_{\mathsf{b}}\Tilde{B}}\times \mbox{\textcircled{\scriptsize 3}},
\end{split}
\end{equation*}
where
\begin{equation*}
\begin{split}
\mbox{\textcircled{\scriptsize 3}}&\coloneqq e^{4\pi ic_{\mathsf{b}}^2c_1^\tau(c_1^\tau+b_1^\tau)}e^{-\pi ic_{\mathsf{b}}^2\frac{4(c_1^\tau-b_1^\tau)+1}{6}}e^{4\pi ic_{\mathsf{b}}^2c_2^\tau(c_2^\tau+b_2^\tau)}e^{-\pi ic_{\mathsf{b}}^2\frac{4(c_2^\tau-b_2^\tau)+1}{6}}e^{4\pi ic_{\mathsf{b}}^2b_3^\tau(b_3^\tau+c_3^\tau)}\\
&\quad\times e^{-\pi ic_{\mathsf{b}}^2\frac{4(b_3^\tau-c_3^\tau)+1}{6}}e^{4\pi ic_{\mathsf{b}}^2c_4^\tau(c_4^\tau+b_4^\tau)}e^{-\pi ic_{\mathsf{b}}^2\frac{4(c_4^\tau-b_4^\tau)+1}{6}}e^{4\pi ic_{\mathsf{b}}^2c_5^\tau(c_5^\tau+b_5^\tau)}e^{-\pi ic_{\mathsf{b}}^2\frac{4(c_5^\tau-b_5^\tau)+1}{6}}\\
&\quad\times e^{4\pi ic_{\mathsf{b}}^2b_6^\tau(b_6^\tau+a_6^\tau)}e^{-\pi ic_{\mathsf{b}}^2\frac{4(b_6^\tau-a_6^\tau)+1}{6}}e^{-4\pi ic_{\mathsf{b}}^2c_2^\tau(1-a_2^\tau)}e^{-4\pi ic_{\mathsf{b}}^2b_3^\tau(1-2a_3^\tau)}e^{-4\pi ic_{\mathsf{b}}^2(c_4^\tau+c_5^\tau+b_6^\tau)(1-2a_4^\tau)}\\
&\quad\times e^{2\pi ic_{\mathsf{b}}^2(1-2a_2^\tau)^2} e^{4\pi ic_{\mathsf{b}}^2(a_4^\tau-a_3^\tau)^2}e^{2\pi ic_{\mathsf{b}}^2(1-2a_2^\tau)(1-2a_3^\tau)},
\end{split}
\end{equation*}
and
\begin{equation*}
\mathscr{Y}_\tau\coloneqq(\mathbb{R}-2c_{\mathsf{b}}(b_2^\tau+c_2^\tau))\times
(\mathbb{R}-2c_{b}(b_4^\tau+c_4^\tau))\times(\mathbb{R}-2c_{\mathsf{b}}(b_3^\tau+c_3^\tau)).
\end{equation*}

By Proposition \ref{property-of-Phi_b} (1), since 
\begin{equation*}
\frac{e^{i\pi\Tilde{D}^2}}{\Phi_{\mathsf{b}}(\Tilde{D})}=e^{i\pi\frac{1+2c_{\mathsf{b}}^2}{6}}\Phi_{\mathsf{b}}(-\Tilde{D}),
\end{equation*} 
\begin{equation*}
\begin{split}
&\quad\lim_{\alpha_Y\rightarrow \tau}\Phi_{\mathsf{b}}\left(\frac{\pi-\omega_{Y,\alpha_Y}(e_0)}{2\pi i\sqrt{\hbar}}\right)Z_\hbar(Y,\alpha_Y)\\
&=e^{-\frac{5}{12}\pi i}\int_{\mathscr{Y}_\tau}d\Tilde{A}d\Tilde{B}d\Tilde{D}\frac{1}{\Phi_{\mathsf{b}}(\Tilde{A})}\Phi_{\mathsf{b}}(-\Tilde{D})\frac{1}{\Phi_{\mathsf{b}}^3(\Tilde{B})}e^{2\pi i\Tilde{A}^2}e^{\pi i\Tilde{B}^2}e^{2\pi i\Tilde{D}(\Tilde{A}-\Tilde{B})}e^{2\pi ic_{\mathsf{b}}\Tilde{A}}\\
&\quad\times e^{-2\pi ic_{\mathsf{b}}\Tilde{B}}e^{\frac{i\pi}{3}c_{\mathsf{b}}^2}\times\mbox{\textcircled{\scriptsize 3}}.
\end{split}
\end{equation*}

Let
\begin{equation*}
\Tilde{D}^\prime\coloneqq -\Tilde{D},
\end{equation*}
then 
\begin{equation*}
\begin{split}
&\quad \lim_{\alpha_Y\rightarrow \tau}\Phi_{\mathsf{b}}\left(\frac{\pi-\omega_{Y,\alpha_Y}(e_0)}{2\pi i\sqrt{\hbar}}\right)Z_\hbar(Y,\alpha_Y)\\
&=e^{-\frac{5}{12}\pi i}\int_{\mathscr{Y}_\tau^\prime}d\Tilde{A}d\Tilde{B}d\Tilde{D}^\prime\frac{\Phi_{\mathsf{b}}(\Tilde{D}^\prime)}{\Phi_{\mathsf{b}}(\Tilde{A})\Phi_{\mathsf{b}}^3(\Tilde{B})}e^{2\pi i(\Tilde{A}^2+\frac{\Tilde{B}^2}{2}-\Tilde{D}^\prime\Tilde{A}+\Tilde{B}\Tilde{D}^\prime)}e^{-\frac{\pi}{\sqrt{\hbar}}(\Tilde{A}-\Tilde{B})}e^{-\frac{i\pi}{12\hbar}}\times\mbox{\textcircled{\scriptsize 3}},
\end{split}
\end{equation*}
where
\begin{equation*}
\begin{split}
\mathscr{Y}_\tau^\prime&\coloneqq(\mathbb{R}-2c_{\rm b}(b_2^\tau+c_2^\tau))\times
(\mathbb{R}-2c_{b}(b_4^\tau+c_4^\tau))\times(\mathbb{R}+2c_{\rm b}(b_3^\tau+c_3^\tau))\\
&=\left(\mathbb{R}-\frac{i}{2\sqrt{\hbar}}(1-2a_2^{\tau})\right)\times
\left(\mathbb{R}-\frac{i}{2\sqrt{\hbar}}(1-2a_4^{\tau})\right)\times
\left(\mathbb{R}+\frac{i}{2\sqrt{\hbar}}(1-2a_3^{\tau})\right).
\end{split}
\end{equation*}

Since $\tau$ satisfies the equations from \eqref{da1} to \eqref{da7}, if we define $a_i=a_{i+1}^\tau,b_i=b_{i+1}^\tau,c_i=c_{i+1}^\tau\quad ( i=1, \cdots, 5 )$, then we can see that $( a_1, b_1, c_1, \cdots, a_5, b_5, c_5 )$ satisfy the equations from \eqref{it1} to \eqref{it5} in Section \ref{ideal-triangulation}. Therefore, 
\begin{equation*}
\lim_{\alpha_Y\rightarrow \tau}\Phi_{\mathsf{b}}\left(\frac{\pi-\omega_{Y,\alpha_Y}(e_0)}{2\pi i\sqrt{\hbar}}\right)Z_\hbar(Y,\alpha_Y)=e^{-\frac{5}{12}\pi i}e^{-\frac{i\pi}{12\hbar}}\times\mbox{\textcircled{\scriptsize 3}}\times J_{X}(\hbar,0)
\end{equation*}
is derived.
\end{proof}
\section{Proof of Theorem \ref{uemura}}
\label{volume-conjecture-strict-proof}
In this section, we describe the proof of Theorem \ref{uemura}, referring to the policy of the proof for twist knots \cite{MR3945172}.

By Theorem \ref{volume-conjecture-1}, and Theorem \ref{volume-conjecture-2},
let
\begin{equation*}
J_{S^3,7_3}\coloneqq e^{-\frac{i\pi}{3}}J_{X}, 
\end{equation*}
then it is shown that Theorem \ref{uemura} (1) and (2) hold.

Furthermore, transformation of variables ${\sf{x}}=2\pi\sqrt{\hbar}x$, $Y=2\pi\sqrt{\hbar}Y^\prime$, $Z=2\pi\sqrt{\hbar}Z^\prime$,
$W=2\pi\sqrt{\hbar}W^\prime$ gives
\begin{equation*}
\begin{split}
Z_\hbar(X,\alpha_X)&=e^{-\frac{\pi}{3}i}e^{\frac{i}{3}\pi c_{\rm{b}}^2}\times\mbox{\textcircled{\scriptsize 1}}\times
\int_{\mathbb{R}+2\pi i\mu(\alpha_X)}\mathfrak{J}_X(\hbar,\mathsf{x})e^{-\frac{\mathsf{x}}{2\pi\hbar}\lambda(\alpha_X)}d\mathsf{x},
\end{split}
\end{equation*}
where we define
\begin{equation*}
\mathbf{y}\coloneqq\begin{bmatrix}
Y\\
Z\\
W
\end{bmatrix}
\end{equation*}
and 
\begin{equation*}
\begin{split}
\mathfrak{J}_X:(\hbar,\mathsf{x})\mapsto\left(\frac{1}{2\pi\sqrt{\hbar}}\right)^4\int_{\mathscr{Y}_{\alpha_X}}d\mathbf{y}e^{\frac{i\mathbf{y}^{\rm T}Q\mathbf{y}-\frac{3}{2}i\mathsf{x}^2+\mathbf{y}^{\rm T}\mathscr{W}}{2\pi\hbar}}\frac{\Phi_{\mathsf{b}}(\frac{Z}{2\pi\sqrt{\hbar}})}{\Phi_{\mathsf{b}}(\frac{Y}{2\pi\sqrt{\hbar}})\Phi_{\mathsf{b}}(\frac{W}{2\pi\sqrt{\hbar}})\Phi_{\mathsf{b}}(\frac{W-\mathsf{x}}{2\pi\sqrt{\hbar}})\Phi_{\mathsf{b}}(\frac{W+\mathsf{x}}{2\pi\sqrt{\hbar}})}. 
\end{split}
\end{equation*}
Here the integral domain is 
\begin{equation*}
\mathscr{Y}_{\alpha_X}\coloneqq (\mathbb{R}-i\pi(1-2a_1))\times (\mathbb{R}+i\pi(1-2a_2))\times(\mathbb{R}-i\pi(1-2a_3)).
\end{equation*}
Then 
\begin{equation}\label{J-relation2}
J_X(\hbar,x)=2\pi\sqrt{\hbar}\mathfrak{J}_X(\hbar,(2\pi\sqrt{\hbar})x)
\end{equation}
holds.
By the equation \eqref{J-relation2}, Theorem \ref{uemura} (3) holds by proving the following equation.
\begin{equation*}
\lim_{\hbar\rightarrow 0^{+}}2\pi\hbar \log |J_X(\hbar,0)|=\lim_{\hbar\rightarrow 0^{+}}2\pi\hbar\log|\mathfrak{J}_X(\hbar,0)|=-\mathrm{Vol}(S^3\backslash 7_3)
\end{equation*}

We prove this below.

First, we discuss the geometricity of the ideal tetrahedral decomposition, i.e., the property that positive dihedral angles corresponding to the complete hyperbolic structure of the underlying $3$-dimensional hyperbolic manifold can be assigned to these tetrahedra.

In \cite{Thurston}, Thurston gives a method to check whether a given tetrahedral decomposition has geometricity. A system of gluing equations with respect to the complex parameters associated with the tetrahedra is defined there, and it is shown that if this system has a solution, the solution is unique and corresponds to the complete hyperbolic metric of the tetrahedrally decomposed manifold.

However, this system of equations is difficult to solve. In the 1990s, Casson and Rivin devised another method to prove geometricity (for a review paper, see
\cite{MR2796632}). The idea is to focus on the argument part of the complex gluing equation and to use the volume functional property.

For the ideal tetrahedral decomposition $X$ of $S^3\backslash 7^3$ obtained in Section \ref{ideal-triangulation}, we prove the following theorem.
\begin{theorem}\label{geometric}
$X$ is geometric.
\end{theorem}

To prove Theorem \ref{geometric}, we use the method by Futer and Gu\'{e}ritaud (\cite{MR2796632, MR2255497}).
First, we show that the volume functional is not maximal on the boundary of the space of the extended angle structures, which is non-empty by Lemma \ref{non-empty}.
Then Theorem \ref{geometric} can be proven using the result by Casson and Rivin (Theorem \ref{Casson-Rivin}).
\subsection{The volume functional}
Lobachevsky function $\Lambda :\mathbb{R}\rightarrow \mathbb{R}$ is given by
\begin{equation*}
\Lambda(x)=-\int_0^x \log |2\sin t|dt.
\end{equation*}
It is well-defined and continuous on $\mathbb{R}$ with period $\pi$.
If $T$ is a hyperbolic ideal tetrahedron with dihedral angles $a$, $b$, $c$, then its volume is
\begin{equation*}
{\rm Vol}(T)=\Lambda(a)+\Lambda(b)+\Lambda(c).
\end{equation*}

Let $X= ( T_1, \ldots, T_N, \sim )$ be an ideal tetrahedral decomposition and $\mathscr{A}_X$ be the space of its angle structures, then it is a (possibly empty) convex polyhedron in $\mathbb{R}^{3N}$.

The volume functional $\mathscr{V}:\overline{\mathscr{A}_X}\rightarrow \mathbb{R}$ is defined by assigning to (the extended) angle structure $\alpha_X=(2\pi a_1,2\pi b_1,2\pi c_1,\ldots,2\pi a_N,2\pi b_N,2\pi c_N)$ a real number 
\begin{equation*}
\mathscr{V}(\alpha_X)=\Lambda(2\pi a_1)+\Lambda(2\pi b_1)+\Lambda(2\pi c_1)+\cdots+\Lambda(2\pi a_N)+\Lambda(2\pi b_N)+\Lambda(2\pi c_N).    
\end{equation*}

From \cite[Theorem 6.1, 6.6]{MR2255497} and \cite[Lemma 5.3]{MR2796632},
it is known that $\mathscr{V}$ is strictly concave on $\mathscr{A}_X$ and concave on $\overline{\mathscr{A}_X}$.
Furthermore, the following theorem holds as stated in \cite[Theorem 1.2]{MR2796632}.
\begin{theorem}[Casson-Rivin \cite{MR1283870}]\label{Casson-Rivin}
Let $M$ be an orientable 3-manifold whose boundaries consist of tori, and let $X$ be an ideal tetrahedral decomposition of $M$. Then, an angle structure $\alpha$ corresponds to the unique complete hyperbolic metric on the interior of $M$ if and only if $\alpha$ is a critical point of the volume functional.
\end{theorem}
When the last situation of this theorem holds, we say that the ideal tetrahedral decomposition $X$ of the $3$-manifold $M$ is geometric.
\subsection{Thurston's complex gluing equation}\label{complex-gluing-equation}
There exists a bijective map from shape structures $\{(2\pi a,2\pi b,2\pi c)\in(0,\pi)^3\mid a+b+c=\frac{1}{2}\}$ on a vertex ordered tetrahedron $T$ to $\mathbb{R}+i\mathbb{R}_{>0}$.
By this map, if $(2\pi a,2\pi b,2\pi c)\in\mathscr{S}_T$ maps to $z\in\mathbb{R}+i\mathbb{R}_{>0}$, and we call $z$ a complex shape structure.

In addition, we define
\begin{equation*}
z^\prime\coloneqq \frac{1}{1-z},\quad z^{\prime\prime}\coloneqq \frac{z-1}{z}.
\end{equation*}

The complex shape structure $z$ corresponds to the edge with a dihedral angle of $2\pi a$. Furthermore, let $\epsilon(T)$ be the sign of a tetrahedron $T$, then if $\epsilon(T)=1$, $z^\prime$ corresponds to $2\pi c$ and $z^{\prime\prime}$ corresponds to $2\pi b$. If $\epsilon(T)=-1$, $z^\prime$ corresponds to $2\pi b$ and $z^{\prime\prime}$ corresponds to $2\pi c$.

In this section, we define the complex logarithmic function as for $z\in\mathbb{C}^\ast$,
\begin{equation*}
\mathrm{Log}(z)\coloneqq\log |z|+i\arg(z),\quad (\arg(z)\in(-\pi, \pi] ).
\end{equation*}

In addition, let
$y:=\epsilon(T)(\mathrm{Log}(z)-i\pi)\in \mathbb{R}+i\epsilon(T)(-\pi,0)$.

The relations between $(2\pi a,2\pi b,2\pi c)$, $(z,z^\prime,z^{\prime\prime})$, and $y$ for each sign of $\epsilon(T)$ are as follows.

If $\epsilon(T)=1$,
\begin{eqnarray*}
y+i\pi=\mathrm{Log}(z)=\log \left(\frac{\sin{2\pi c}}{\sin{2\pi b}}\right)+2\pi ia,\\ 
-\mathrm{Log}(1+e^y)=\mathrm{Log}(z^\prime)=\log\left(\frac{\sin{2\pi b}}{\sin{2\pi a}}\right)+2\pi ic,\\
\mathrm{Log}(1+e^{-y})=\mathrm{Log}(z^{\prime\prime})=\log\left(\frac{\sin{2\pi a}}{\sin{2\pi c}}\right)+2\pi ib,\\
y=\log\left(\frac{\sin{2\pi c}}{\sin{2\pi b}}\right)-i\pi(1-2a)\in\mathbb{R}-i\pi(1-2 a),\\
z=-e^y\in\mathbb{R}+i\mathbb{R}_{>0}.\\
\end{eqnarray*}

If $\epsilon(T)=-1$, 
\begin{eqnarray*}
-y+i\pi=\mathrm{Log}(z)=\log \left(\frac{\sin{2\pi b}}{\sin{2\pi c}}\right)+2\pi ia,\\ 
-\mathrm{Log}(1+e^{-y})=\mathrm{Log}(z^\prime)=\log\left(\frac{\sin{2\pi c}}{\sin{2\pi a}}\right)+2\pi ib,\\
\mathrm{Log}(1+e^{y})=\mathrm{Log}(z^{\prime\prime})=\log\left(\frac{\sin{2\pi a}}{\sin{2\pi b}}\right)+2\pi ic,\\
y=\log\left(\frac{\sin{2\pi c}}{\sin{2\pi b}}\right)+i\pi(1-2 a)\in\mathbb{R}+i\pi(1-2a),\\
z=-e^{-y}\in\mathbb{R}+i\mathbb{R}_{>0}.\\
\end{eqnarray*}

For a tetrahedral decomposition $X$ and an angle structure $\alpha_X\in\mathscr{A}_X$, the complex weight function $\omega_{X,\alpha_X}^\mathbb{C}:X^1\rightarrow \mathbb{C}$ is defined by assigning to the edge $e\in X^1$ the logarithmic sum of the complex shapes corresponding to the inverse images of $e$ by $\sim$.

Given an ideal tetrahedral decomposition $X$ of a 3-manifold $M$ with torus $S$ as a boundary component, suppose that $S$ is triangulated by triangles generated by truncating the corners of tetrahedra.
Furthermore, let $\sigma$ be an oriented normal closed curve on $S$.
Then $\sigma$ cuts off the corners of triangles on $S$, and let $z_1, z_2, \cdots, z_l$ be the corresponding complex shapes.
Assume $\epsilon_k=1$ if the corner of the triangle is to the left of $\sigma$ and $\epsilon_k=-1$ if it is to the right. Then 
the complex holonomy is defined as 
\begin{equation}
H^{\mathbb{C}}(\sigma)\coloneqq \sum_{k=1}^l \epsilon_k \mathrm{Log}(z_k).
\end{equation}

The angular holonomy is defined by replacing $\mathrm{Log}(z_k)$ with $\arg(z_k)=\Im(\mathrm{Log}(z_k))$.

The complex gluing edge equation is defined as follows,
\begin{equation*}
\forall e\in X^1,\quad \omega_{X,\alpha_X}^\mathbb{C}(e)=2i\pi.
\end{equation*}

On the other hand, the complex completeness equation is defined by the vanishing of the complex holonomy for any closed curve generating the first order homology $H_1(\partial M)$.

Let $M$ be an orientable $3$-manifold whose boundary consists of tori and given an ideal tetrahedral decomposition $X$. The angle structure $\alpha_X\in\mathscr{A}_X$
corresponds to the unique complete hyperbolic metric on the interior of $M$ if and only if  it satisfies the complex gluing edge equation and the complex completeness equation.
\subsection{The classical dilogarithm}
The classical dilogarithm is defined as follows. For $z\in\mathbb{C}\backslash[1,\infty)$, 
\begin{equation*}
\mathrm{Li}_2(z)\coloneqq -\int_0^z\mathrm{Log}(1-u)\frac{du}{u}.
\end{equation*}
The classical dilogarithm satisfies the following properties.
\begin{theorem}\label{classical-dilog}
(1) ( inversion relation )
\begin{equation*}
\forall z\in\mathbb{C}\backslash[1,\infty),\quad \mathrm{Li}_2\left(\frac{1}{z}\right)=-\mathrm{Li}_2(z)-\frac{\pi^2}{6}-\frac{1}{2}\mathrm{Log}(-z)^2.
\end{equation*}
(2) ( integral form )

For any $y\in\mathbb{R}+i(-\pi,\pi)$,
\begin{equation*}
\frac{-i}{2\pi}\mathrm{Li_2}(-e^y)=\int_{v\in\mathbb{R}+i0^+}\frac{\exp(-i\frac{yv}{\pi})}{4v^2\sinh{v}}dv,
\end{equation*}
where $\mathbb{R}+i0^+$ means the deformed path of the real axis in the complex plane which does not pass through the origin, but the upper half plane in the neighborhood of the origin. In this section, we use this notation. For example, the path may be a semicircle on the upper half-plane centered at the origin in the neighborhood of the origin.
\end{theorem}
\subsection{Bloch-Wigner function}
Bloch-Wigner function $D:\mathbb{C}\rightarrow\mathbb{R}$ is defined as follows.
\begin{equation*}
D(z)\coloneqq
\left\{
\begin{array}{lr}
\Im(\mathrm{Li}_2(z))+\arg(1-z)\log|z|& (\text{if}\  z\in\mathbb{C}\backslash\mathbb{R} ) \\
0 & ( \text{otherwise} ).
\end{array}
\right.    
\end{equation*}
Bloch-Wigner function is continuous on $\mathbb{C}$ and real analytic on $\mathbb{C}\backslash\{0,1\}$.
\begin{theorem}[\cite{MR0662760}]\label{bloch-hyperbolic-volume}
Let $T$ be an ideal tetrahedron in $3$-dimensional hyperbolic space $\mathbb{H}^3$, and the complex shape structure is given by $z$. 
Then the hyperbolic volume $\mathrm{Vol}(T)$ is expressed by the following equation.
\begin{equation*}
\mathrm{Vol}(T)=D(z)=D\left(\frac{z-1}{z}\right)=D\left(\frac{1}{1-z}\right).
\end{equation*}
\end{theorem}
We now give the definition of the asymptotic expansion and state the theorem on the asymptotic expansion necessary for the proof of Theorem \ref{uemura} (3).
\begin{definition}
Let $f:\Omega\rightarrow \mathbb{C}$ be a function on a non-bounded domain $\Omega\subset \mathbb{C}$. (Either convergent or divergent) complex power series $\sum_{n=0}^\infty a_nz^{-n}$ is called an asymptotic expansion of $f$, if for any fixed integer $N\ge 0$, 
\begin{equation*}
f(z)=\sum_{n=0}^N a_n z^{-n}+O (z^{-(N+1)})
\end{equation*}
if $z\rightarrow \infty$.
In this case, we denote
\begin{equation*}
f(z)\underset{z\rightarrow\infty}{\cong} \sum_{n=0}^\infty a_nz^{-n}.
\end{equation*}
\end{definition}
By Section 2.4.5 in \cite{Fedoryuk1989}, the following holds.
\begin{theorem}[Fedoryuk]\label{Fedoryuk}
Let $m\ge 1$ be an integer, and $\gamma^m$ be an $m$-dimensional smooth compact real submanifold of $\mathbb{C}^m$ with a connected boundary.
Denote $z=(z_1,\ldots,z_m)\in\mathbb{C}^m$, $dz=dz_1\cdots dz_m$.
Let $z\mapsto f(z)$ and $z\mapsto S(z)$ be analytic complex-valued functions on the domain $D$ such that $\gamma ^m\subset D\subset \mathbb{C}^m$.
With a parameter $\lambda\in\mathbb{R}$, we define
\begin{equation*}
F(\lambda)=\int_{\gamma^m} f(z)\exp (\lambda S(z))dz.
\end{equation*}

Suppose $\max_{z\in\gamma ^m} \Re S(z)$ is attained only at the point $z^0$, where $z^0$ is an interior point of $\gamma^m$ and a simple saddle point of $S$, i.e., $\nabla S(z^0)=0$ and $\det\mathrm{Hess}(S)(z^0)\neq 0$. 
Then if $\lambda\rightarrow \infty$, the following asymptotic expansion exists.
\begin{equation*}
F(\lambda)\underset{\lambda\rightarrow\infty}{\cong}\left(\frac{2\pi}{\lambda}\right)^{\frac{m}{2}}\frac{\exp(\lambda S(z^0))}{\sqrt{\det{\mathrm{Hess}(S)(z^0)}}}\left[f(z^0)+\sum_{k=1}^\infty c_k\lambda^{-k}\right],
\end{equation*}
where $c_k$ is a complex number, and the choice of the branch of the square root $\sqrt{\det \mathrm{Hess}(S)(z^0)}$ depends on the orientation of $\gamma ^m$.

In particular,
\begin{equation*}
\lim_{\lambda\rightarrow\infty}\frac{1}{\lambda}\log |F(\lambda)|=\Re S(z^0)   
\end{equation*}
holds.
\end{theorem}

By the equations $\eqref{it1'},\eqref{it2''},\eqref{it3'},\eqref{it4'}$, the angle structure $\mathscr{A}_X$ on the ideal tetrahedral decomposition $X$ of $S^3\backslash 7^3$ obtained in Section \ref{ideal-triangulation} is given by 
\begin{equation*}
 \mathscr{A}_X=\{\alpha_X\in\mathscr{S}_X\mid b_1+b_2+a_3=a_4+c_5, a_2=2b_1+c_1, a_2+b_3=c_4+b_5,a_3=b_1+b_2\} .  
\end{equation*}

In general, for a single tetrahedron $T$ constituting a tetrahedral decomposition $X$ with an extended shape structure $\alpha_X\in\overline{\mathscr{S}_X}$, if at least one of the dihedral angles of $T$ is $0$, it is called flat for $\alpha_X$, and if two dihedral angles are $0$ and the third dihedral angle is $\pi$, it is 
called taut for $\alpha_X$. In both cases, the volume of $T$ is $0$.

The following holds in the same way as in Lemma 4.3 in \cite{MR3945172}.
\begin{lemma}\label{flat-taut}
Let $\alpha_X\in\overline{\mathscr{A}_X}\backslash \mathscr{A}_X$ be such that the volume functional on $\overline{\mathscr{A}_X}$ is maximal at $\alpha_X$.
If one dihedral angle of $\alpha_X$ is zero, the other two dihedral angles of the tetrahedron containing this dihedral angle are $0$ and $\pi$. That is, if a tetrahedron is flat for $\alpha_X$, then it is taut.
\end{lemma}
In the four constraints in the definition of $\mathscr{A}_X$, even if we replace the variables like $(a_4,b_4,c_4)\leftrightarrow (c_5,a_5,b_5)$, the constraints are invariant. Therefore, by the concavity of the volume functional and the convexity of $\overline{\mathscr{A}_X}$, there exists an extended shape structure with $(a_4,b_4,c_4)=(c_5,a_5,b_5)$ which has the maximal volume. In this case, the constraints are given by

\begin{empheq}[left={\left\{
\begin{aligned}
b_1+b_2+a_3&=2a_4\\
a_2&=2b_1+c_1\\
a_2+b_3&=2c_4\\
a_3&=b_1+b_2\\
\end{aligned}
\right.\iff\empheqlbrace}]{align}
a_3&=a_4 \label{restriction-1}\\
a_2&=2b_1+c_1\label{restriction-2}\\
a_2+b_3&=2c_4\label{restriction-3}\\
a_3&=b_1+b_2\label{restriction-4}
\end{empheq}
\begin{lemma}\label{volume-maximizer}
Suppose the volume functional on $\overline{\mathscr{A}_X}$ is maximal at $\alpha_X$. Then $\alpha_X$ is not on the boundary $\overline{\mathscr{A}_X}\backslash\mathscr{A}_X$, but necessarily on the interior $\mathscr{A}_X$.
\end{lemma}
\begin{proof}
Since the volume functional takes its maximum value on $(a_4,b_4,c_4)=(c_5,a_5,b_5)$, if it can be shown that the volume functional under the constraints of $(a_4,b_4,c_4)=(c_5,a_5,b_5)$ and the equations \eqref{restriction-1}, \eqref{restriction-2}, \eqref{restriction-3}, \eqref{restriction-4} is maximal, not on the boundary $\overline{\mathscr{A}_X}\backslash\mathscr{A}_X$ but on the interior $\mathscr{A}_X$, the lemma is proven because the volume functional is strictly concave on $\mathscr{A}_X$.

Therefore, by reductio ad absurdum, we derive the contradiction by assuming that the volume functional under the constraints of $(a_4,b_4,c_4)=(c_5,a_5,b_5)$ and the equations \eqref{restriction-1}, \eqref{restriction-2},
 \eqref{restriction-3}, \eqref{restriction-4} takes the maximum value on the boundary $\overline{\mathscr{A}_X}\backslash\mathscr{A}_X$. 

In this case, at least one of $T_1,\ldots,T_4(T_5)$ is a taut tetrahedron by Lemma \ref{flat-taut}.

(i) In the case $T_1$ is taut.

\quad (a) If $(a_1,b_1,c_1)=(\frac{1}{2},0,0)$, by the equation \eqref{restriction-2}, $a_2=0$. Therefore $T_2$ is taut, so $b_2=0$ or $b_2=\frac{1}{2}$.

\quad\quad ($\alpha$) If $b_2=0$, by the equations \eqref{restriction-1} and \eqref{restriction-4}, $a_3=a_4=0$.
Since all tetrahedra are taut, the volume is $0$, which violates the maximality.

\quad\quad ($\beta$) If $b_2=\frac{1}{2}$, $a_3=a_4=\frac{1}{2}$ by the equations \eqref{restriction-1} and \eqref{restriction-4}.
Since all tetrahedra are taut also, in this case, the volume is $0$, which violates the maximality.

\quad (b) If $(a_1,b_1,c_1)=(0,0,\frac{1}{2})$, $a_2=\frac{1}{2}$ by the equation \eqref{restriction-2}.
Therefore, $b_2=c_2=0$. By the equations \eqref{restriction-1} and \eqref{restriction-4}, $a_3=a_4=0$. Therefore, $T_1,\ldots,T_4(T_5)$
are all taut, and the volume is $0$, which violates the maximality.

\quad (c) If $(a_1,b_1,c_1)=(0,\frac{1}{2},0)$, this situation does not exist because $a_2=1>\frac{1}{2}$ by the equation \eqref{restriction-2}.

(ii) In the case $T_2$ is taut.

\quad (a) If $(a_2,b_2,c_2)=(0,0,\frac{1}{2})$, by the equation \eqref{restriction-2}, $2b_1+c_1=0$. Therefore $b_1=c_1=0$, and $T_1$ is taut, so this case is attributed to (i).

\quad (b) If $(a_2,b_2,c_2)=(0,\frac{1}{2},0)$, as in (a), by the equation \eqref{restriction-2}, $b_1=c_1=0$
and $T_1$ is taut. Then this case is attributed to (i).

(iii) In the case $T_3$ is taut.

\quad (a) If $(a_3,b_3,c_3)=(0,0,\frac{1}{2})$ or $(0,\frac{1}{2},0)$, by the equation \eqref{restriction-4}, $b_1=b_2=0$ and this case is attributed to (i).

\quad (b) If $(a_3,b_3,c_3)=(\frac{1}{2},0,0)$, $a_4=\frac{1}{2}$ by the equation \eqref{restriction-1}. Therefore $b_4=c_4=0$. 
By the equation \eqref{restriction-3}, $a_2=0$. By the equation \eqref{restriction-2}, $b_1=c_1=0$, so this case is attributed to (i).

(iv) In the case $T_4(T_5)$ is taut.

\quad (a) If $(a_4,b_4,c_4)=(0,0,\frac{1}{2})$ or $(0,\frac{1}{2},0)$, by the equation \eqref{restriction-1}, $a_3=0$, and by the equation \eqref{restriction-4}, $b_1=b_2=0$. Therefore, this case is attributed to (i).

\quad (b) If $(a_4,b_4,c_4)=(\frac{1}{2},0,0)$, by the equation \eqref{restriction-3}, $a_2=b_3=0$. By the equation \eqref{restriction-2}, $b_1=c_1=0$. Therefore, this case is attributed to (i).

Therefore, the volume functional can be maximal on the boundary if $(a_2,b_2,c_2)=(\frac{1}{2},0,0)$ and $T_1$, $T_3$ and $T_4(T_5)$ are non-flat tetrahedra.

Then 
\begin{empheq}[left={\empheqlbrace}]{align}
a_3&=a_4 \label{restriction-1'}\\
\frac{1}{2}&=2b_1+c_1\label{restriction-2'}\\
\frac{1}{2}+b_3&=2c_4\label{restriction-3'}\\
a_3&=b_1 \label{restriction-4'}.
\end{empheq}

By the equation \eqref{restriction-2'}, $\frac{1}{2}=b_1+(\frac{1}{2}-a_1)$, then $a_1=b_1$.

Let $a_3=b_1=a_4=a_1=\frac{s}{2\pi}$. Then $c_1=\frac{1}{2}-\frac{s}{\pi}$. Let $b_4=\frac{v}{2\pi}$, then $c_4=\frac{1}{2}-\frac{s}{2\pi}-\frac{v}{2\pi}$.
By the equation \eqref{restriction-3'}, $b_3=\frac{1}{2}-\frac{s}{\pi}-\frac{v}{\pi}$. $c_3=\frac{1}{2}-a_3-b_3=\frac{s}{2\pi}+\frac{v}{\pi}$.

Since $0<a_i<\frac{1}{2}$, $0<b_i<\frac{1}{2}$, and $0<c_i<\frac{1}{2}$ for $i=1$, $3$, and $4$, it follows that $0<s$, $0<v$ and $s+v<\frac{\pi}{2}$.

Let $t>0$ be a sufficiently small real number and
\begin{eqnarray*}
a_1^t=a_1+\frac{t}{2\pi},\quad b_1^t=b_1,\quad c_1^t=c_1-\frac{t}{2\pi}\\
a_2^t=a_2-\frac{t}{2\pi},\quad b_2^t=b_2,\quad c_2^t=c_2+\frac{t}{2\pi}\\
a_3^t=a_3,\quad b_3^t=b_3+\frac{t}{2\pi},\quad c_3^t=c_3-\frac{t}{2\pi}\\
a_4^t=a_4,\quad b_4^t=b_4,\quad c_4^t=c_4,
\end{eqnarray*}
then $(a_i^t,b_i^t,c_i^t)\quad(i=1,\cdots,4)$ satisfy the constraints from \eqref{restriction-1} to \eqref{restriction-4}.

Calculation of the derivative of the volume functional $\mathscr{V}$ on the extended angle structure given by $(a_i^t,b_i^t,c_i^t)\quad(i=1,\cdots,4)$ gives, noting that since $b_2^t=b_2=0$, the volume of $T_2$ is zero independently of $t$, 
\begin{eqnarray*}
\left.\frac{\partial \mathscr{V}}{\partial t}\right|_{t=0}&=&\Lambda^\prime(2\pi a_1)-\Lambda^\prime(2\pi c_1)
+\Lambda^\prime(2\pi b_3)-\Lambda^\prime(2\pi c_3)\\
&=&-\log{\sin{2\pi a_1}}+\log{\sin{2\pi c_1}}-\log{\sin{2\pi b_3}}+\log{\sin{2\pi c_3}}\\
&=&-\log{\sin{s}}+\log{\sin{2s}}-\log{\sin{2(s+v)}}+\log{\sin{(s+2v)}}.
\end{eqnarray*}

Then 
\begin{equation*}
\exp{\left(\left.-\frac{\partial \mathscr{V}}{\partial t}\right |_{t=0}\right)}=\frac{\sin{2(s+v)}}{2\cos{s}\sin{(s+2v)}}.
\end{equation*}

Fix $s$. Let $f(v)=\frac{\sin{2(s+v)}}{\sin{(s+2v)}}$, then
\begin{eqnarray*}
f^\prime(v)&=&\frac{2\cos{2(s+v)}\sin{(s+2v)}-2\sin{2(s+v)}\cos{(s+2v)}}{\sin^2{(s+2v)}}\\
&=&\frac{2\{\sin{(s+2v-2(s+v))}\}}{\sin^2{(s+2v)}}=\frac{2\sin{(-s)}}{\sin^2{(s+2v)}}<0.
\end{eqnarray*}

Thus, $f(v)$ strictly monotonically decreases and 
\begin{equation*}
\lim_{v\rightarrow 0^+}\frac{\sin{2(s+v)}}{2(\cos{s})\sin{(s+2v)}}=\lim_{v\rightarrow 0^+}\frac{\sin{2s}}{2\cos{s}\sin{s}}=1.
\end{equation*}

Therefore, $\exp{\left(\left.-\frac{\partial\mathscr{V}}{\partial t}\right |_{t=0}\right)}<1$, that is $\left.\frac{\partial\mathscr{V}}{\partial t}\right |_{t=0}>0$, then the volume functional $\mathscr{V}$ is not maximal at $t=0$, which is a contradiction.
\end{proof}
\begin{proof}[{\bf Proof of Theorem \ref{geometric}}]
By Lemma \ref{non-empty}, $\mathscr{A}_X$ is non-empty. Therefore, since a continuous function on a non-empty compact set has a maximum value, the volume functional $\mathscr{V}:\overline{\mathscr{A}_X}\rightarrow \mathbb{R}$ is maximal at some $\alpha_X\in\overline{\mathscr{A}_X}$.  
Since Lemma \ref{volume-maximizer}, $\alpha_X \notin\overline{\mathscr{A}_X}\backslash\mathscr{A}_X$, that is, $\alpha_X\in\mathscr{A}_X$, it is shown that $X$ is geometric by Theorem \ref{Casson-Rivin}. 
\end{proof}
\subsection{Cusp triangulation and the complex gluing equation}
By removing the neighborhood of each vertex of tetrahedra in Figure \ref{fig:ideal_triangulation}, we cut off the corners of tetrahedra. Then let $\nu(7_3)$ be a tubular neighborhood of the knot $7_3$ in $S^3$, and we obtain a triangulation of the boundary torus $\partial \nu(7_3)$. This triangulation of the torus is shown in Figure \ref{fig:torus_triangulation}.
\begin{figure}[tbh]
\centering
\includegraphics[width=\textwidth]{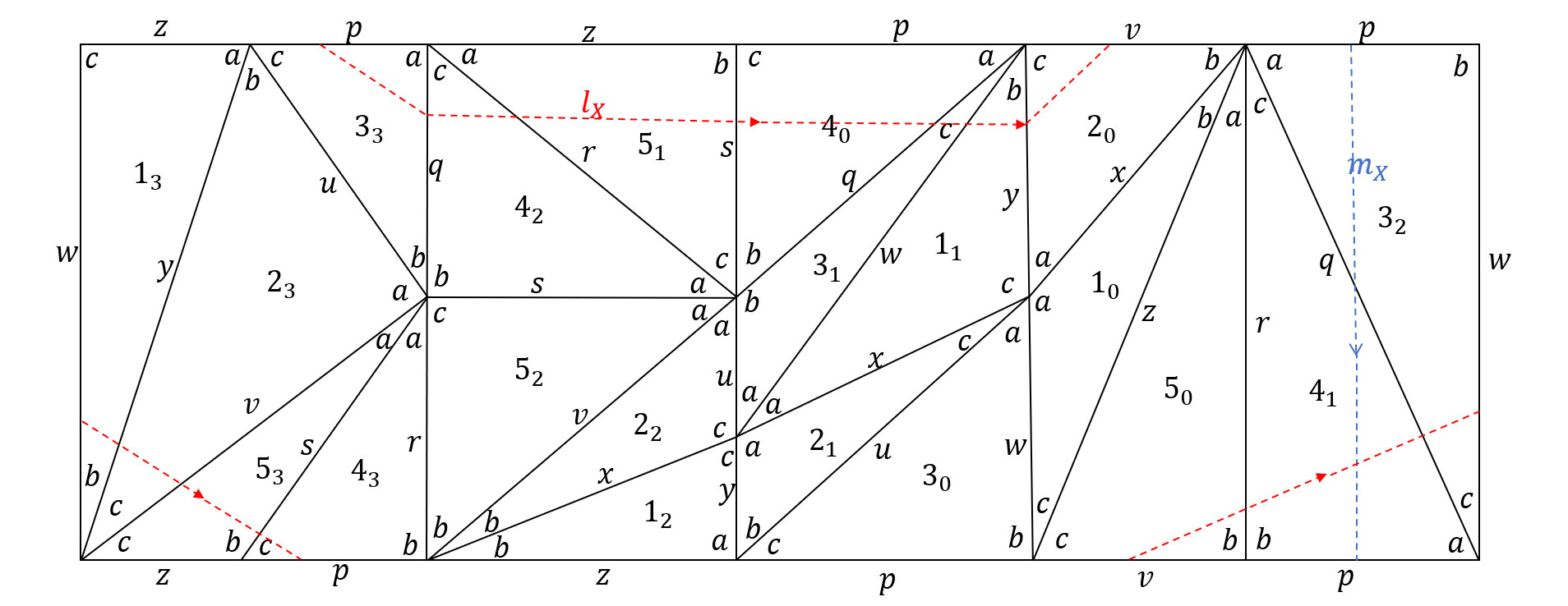}
\caption{A triangulation of the boundary torus of the tubular neighborhood of the knot $7_3$.}
\label{fig:torus_triangulation}
\end{figure}
Here, $i_j$ ($i=1,\ldots,5$, $j=0,\ldots,3$) inside the triangle in Figure \ref{fig:torus_triangulation} denotes the triangle created by truncating the neighborhood of the vertex $j$ of the tetrahedron $T_i$ in Figure \ref{fig:ideal_triangulation}.
The letter marked near each edge of each triangle indicates the variable of Figure \ref{fig:ideal_triangulation} that represents the face of the tetrahedron that includes the edge. The edges with the same letter on the top, bottom, left, and right of the rectangle in the figure are considered to be identical to each other.
The symbols $a,b$, and $c$ on the corners of the triangle indicate the dihedral angles, and for a triangle with $i_j$ marked inside, it means $2\pi a_i, 2\pi b_i, 2\pi c_i$, respectively. Closed curves representing the torus meridian $m_X$ and longitude $l_X$ are drawn by dashed lines in the figure.
Let $z_i$ be a complex shape structure of the tetrahedron $T_i$, then the complex gluing equations are given as follows. Here, $e_1,\ldots,e_5\in X^1$ are shown in Figure \ref{fig:ideal_triangulation}, \ref{fig:ideal_triangulation_edge}.
\begin{eqnarray}
\omega_{X,\alpha}^\mathbb{C}(e_1)=2i\pi &\iff&\mathrm{Log} z_1^{\prime\prime}+\mathrm{Log} z_2^\prime +\mathrm{Log} z_3 +\mathrm{Log} z_4^{\prime\prime} +\mathrm{Log} z_4^\prime \nonumber\\
& & +\mathrm{Log} z_5 +\mathrm{Log} z_5^{\prime\prime} =2i\pi, \label{e1-complex-gluing}\\
\omega_{X,\alpha}^\mathbb{C}(e_2)=2i\pi&\iff&\mathrm{Log} z_1^{\prime\prime} +\mathrm{Log} z_1^\prime +\mathrm{Log} z_2^{\prime\prime} +\mathrm{Log} z_3^{\prime\prime} +\mathrm{Log} z_3^\prime \nonumber\\
& &+\mathrm{Log} z_4+\mathrm{Log} z_5^\prime =2i\pi, \label{e2-complex-gluing}\\
\omega_{X,\alpha}^\mathbb{C}(e_3)=2i\pi&\iff&\mathrm{Log} z_2+\mathrm{Log}z_3^{\prime\prime}+\mathrm{Log} z_4+\mathrm{Log} z_4^{\prime\prime}\nonumber\\
& &+\mathrm{Log} z_5+\mathrm{Log} z_5^\prime=2i\pi,\label{e3-complex-gluing}\\
\omega_{X,\alpha}^\mathbb{C}(e_4)=2i\pi&\iff&\mathrm{Log} z_1+\mathrm{Log} z_1^\prime +\mathrm{Log} z_2+\mathrm{Log} z_2^{\prime\prime} +\mathrm{Log} z_3 =2i\pi, \label{e4-complex-gluing}\\
\omega_{X,\alpha}^\mathbb{C}(e_5)=2i\pi&\iff&\mathrm{Log} z_1 +\mathrm{Log} z_2^\prime +\mathrm{Log} z_3^\prime +\mathrm{Log} z_4^\prime +\mathrm{Log} z_5^{\prime\prime} =2i\pi.\label{e5-complex-gluing}
\end{eqnarray}

The complex completeness equations regarding the meridian $m_X$ and longitude $l_X$ in Figure \ref{fig:torus_triangulation} are
\begin{equation}
H^{\mathbb{C}}(m_X)=0\iff-\mathrm{Log} z_3+\mathrm{Log} z_4=0\iff z_3=z_4\label{meridian-completeness}
\end{equation}
\begin{eqnarray}
& & H^{\mathbb{C}}(l_X)=0\nonumber\\
&\iff&-\mathrm{Log} z_1^{\prime\prime}-\mathrm{Log} z_2^{\prime\prime}+\mathrm{Log} z_5-\mathrm{Log}z_4^\prime+\mathrm{Log} z_3+\mathrm{Log} z_4^\prime
-\mathrm{Log} z_5^\prime\nonumber\\
&&\quad-\mathrm{Log} z_4^{\prime\prime}+\mathrm{Log} z_3^\prime+\mathrm{Log} z_1^{\prime\prime}+\mathrm{Log} z_2^{\prime\prime}
-\mathrm{Log} z_5^{\prime\prime}+\mathrm{Log} z_4^\prime-\mathrm{Log} z_3^\prime=0\nonumber\\
&\iff&2\mathrm{Log} z_5+\mathrm{Log} z_3+\mathrm{Log} z_4^\prime-\mathrm{Log}z_4^{\prime\prime}=i\pi,\label{completeness-longitude}
\end{eqnarray}
where, in the derivation of the equation \eqref{completeness-longitude}, we used $\mathrm{Log} z_i+\mathrm{Log} z_i^\prime+\mathrm{Log} z_i^{\prime\prime}=i\pi\quad (i=1,\ldots,5)$.

Under the assumption that the equation \eqref{meridian-completeness} holds, that is, $z_3=z_4$, the equation \eqref{e1-complex-gluing} is transformed as follows, 
\begin{align}
&\mathrm{Log} z_1^{\prime\prime}+\mathrm{Log} z_2^\prime+\mathrm{Log} z_3+\mathrm{Log} z_3^{\prime\prime}+\mathrm{Log} z_3^\prime 
+\mathrm{Log} z_5+\mathrm{Log} z_5^{\prime\prime}=2i\pi\nonumber\\
\iff&\mathrm{Log} z_1^{\prime\prime}+\mathrm{Log} z_2^{\prime}+i\pi+i\pi-\mathrm{Log} z_5^\prime=2i\pi\nonumber\\
\iff&\mathrm{Log} z_1^{\prime\prime}+\mathrm{Log} z_2^\prime-\mathrm{Log} z_5^\prime=0\tag{$\ref{e1-complex-gluing}^\prime$}.\label{e1-complex-gluing'}
\end{align}

The equation \eqref{e2-complex-gluing} is transformed as 
\begin{align}
&i\pi-\mathrm{Log}z_1+\mathrm{Log} z_2^{\prime\prime}+i\pi+\mathrm{Log}z_5^\prime =2i\pi\nonumber\\
\iff&\mathrm{Log} z_1-\mathrm{Log} z_2^{\prime\prime}-\mathrm{Log}z_5^\prime=0.\tag{$\ref{e2-complex-gluing}^\prime$}\label{e2-complex-gluing'}
\end{align}

The equation \eqref{e3-complex-gluing} is transformed as
\begin{align}
&\mathrm{Log} z_2+i\pi-\mathrm{Log} z_3^\prime+\mathrm{Log} z_3^{\prime\prime}+i\pi-\mathrm{Log} z_5^{\prime\prime}=2i\pi\nonumber\\
\iff&\mathrm{Log} z_2-\mathrm{Log} z_3^\prime+\mathrm{Log} z_3^{\prime\prime}-\mathrm{Log} z_5^{\prime\prime}=0.\tag{$\ref{e3-complex-gluing}^\prime$}\label{e3-complex-gluing'}
\end{align}

The equation \eqref{e4-complex-gluing} is transformed as
\begin{align}
&i\pi-\mathrm{Log} z_1^{\prime\prime}+i\pi-\mathrm{Log} z_2^\prime+\mathrm{Log} z_3=2i\pi\nonumber\\
\iff&\mathrm{Log} z_3=\mathrm{Log} z_1^{\prime\prime}+\mathrm{Log} z_2^\prime.\tag{$\ref{e4-complex-gluing}^\prime$}\label{e4-complex-gluing'}
\end{align}

The equation \eqref{e5-complex-gluing} is transformed as 
\begin{align}
\mathrm{Log} z_1 +\mathrm{Log} z_2^\prime +2\mathrm{Log} z_3^\prime +\mathrm{Log} z_5^{\prime\prime} =2i\pi.\tag{$\ref{e5-complex-gluing}^\prime$}\label{e5-complex-gluing'}
\end{align}

The equation \eqref{completeness-longitude} is transformed as 
\begin{align}
&2\mathrm{Log} z_5+i\pi-\mathrm{Log} z_3^{\prime\prime} -\mathrm{Log} z_3^{\prime\prime}=i\pi\nonumber\\
\iff&\mathrm{Log} z_5=\mathrm{Log} z_3^\prime
\iff z_5=z_3^{\prime\prime}\iff z_3=z_5^\prime.\tag{$\ref{completeness-longitude}^\prime$}\label{completeness-longitude'}
\end{align}

If the equation \eqref{completeness-longitude'} holds, $z_3^\prime=z_5^{\prime\prime}$, thus the equation \eqref{e5-complex-gluing'} is transformed as
\begin{align}
\mathrm{Log} z_1+\mathrm{Log} z_2^{\prime} +3\mathrm{Log} z_3^\prime=2i\pi\tag{$\ref{e5-complex-gluing}^{\prime\prime}$},\label{e5-complex-gluing''}
\end{align}
and the equation \eqref{e1-complex-gluing'} is transformed as
\begin{align*}
\mathrm{Log}z_3=\mathrm{Log}z_1^{\prime\prime}+\mathrm{Log}z_2^\prime,
\end{align*}
and the equation \eqref{e4-complex-gluing'} is derived. In addition, the equation \eqref{e3-complex-gluing'} is transformed as 
\begin{align}
2\mathrm{Log} z_3^\prime=\mathrm{Log} z_2+\mathrm{Log} z_3^{\prime\prime},\tag{$\ref{e3-complex-gluing}^{\prime\prime}$}\label{e3-complex-gluing''}
\end{align}
and the equation \eqref{e2-complex-gluing'} is transformed as 
\begin{align}
\mathrm{Log} z_3=\mathrm{Log} z_1-\mathrm{Log} z_2^{\prime\prime}.\tag{$\ref{e2-complex-gluing}^{\prime\prime}$}\label{e2-complex-gluing''}
\end{align}

On the other hand, transforming the equation \eqref{e5-complex-gluing''} using the equation \eqref{e2-complex-gluing''},
\begin{align*}
&\mathrm{Log} z_3+\mathrm{Log} z_2^{\prime\prime}+\mathrm{Log} z_2^\prime+3\mathrm{Log} z_3^\prime=2i\pi\\
\iff&i\pi-\mathrm{Log} z_2+i\pi-\mathrm{Log} z_3^{\prime\prime}+2\mathrm{Log} z_3^\prime=2i\pi\\
\iff&2\mathrm{Log} z_3^\prime=\mathrm{Log} z_2+\mathrm{Log} z_3^{\prime\prime},
\end{align*}
and the equation \eqref{e3-complex-gluing''} is derived.

Therefore, the equations from \eqref{e1-complex-gluing} to \eqref{completeness-longitude} are satisfied if and only if the equations \eqref{meridian-completeness}, \eqref{completeness-longitude'}, \eqref{e4-complex-gluing'}, \eqref{e2-complex-gluing''}, \eqref{e5-complex-gluing''} are satisfied.

By Theorem \ref{geometric}, $X$ has a unique complex shape structure corresponding to the complete hyperbolic metric, and this complex shape structure is the unique $z=(z_1,z_2,z_3,z_4,z_5)\in(\mathbb{R}+i\mathbb{R}_{>0})^5$ which satisfies the equations \eqref{meridian-completeness}, \eqref{completeness-longitude'}, \eqref{e4-complex-gluing'}, \eqref{e2-complex-gluing''}, and \eqref{e5-complex-gluing''}.  

Let
\begin{equation*}
\mathscr{U}\coloneqq (\mathbb{R}+i(-\pi,0))\times(\mathbb{R}+i(0,\pi))\times(\mathbb{R}+i(-\pi,0)).  
\end{equation*} 

Let $\alpha_X^0\in\mathscr{A}_X$ be the angle structure corresponding to the unique complete hyperbolic metric, which exists by Theorem \ref{geometric}, and let
\begin{equation*}
\mathscr{Y}^0\coloneqq\mathscr{Y}_{\alpha_X^0}.  
\end{equation*}
The potential function $S:\mathscr{U}\rightarrow \mathbb{C}$ is defined as follows, 
\begin{equation}\label{definition-of-S}
S(\mathbf{y})\coloneqq i\mathbf{y}^{\mathrm{T}}Q\mathbf{y}+\mathbf{y}^{\mathrm{T}}\mathscr{W}+i\mathrm{Li}_2(-e^{Y})+3i\mathrm{Li}_2(-e^{W})-i\mathrm{Li}_2(-e^{Z}).
\end{equation}
\subsection{The property of the potential function $S$ on $\mathscr{U}$.}
\begin{lemma}[ \cite{MR3945172}, Lemma 7.2 ]\label{invertible}
Let $m\ge 1$ be an integer, and let $S_1$ and $S_2\in M_m(\mathbb{R})$ be real square matrices of order $m$ such that $S_1$ is a symmetric positive definite matrix and $S_2$ is a symmetric matrix.
Then the complex symmetric matrix $S_1+iS_2$ is invertible.
\end{lemma}
Hereafter, let $y_1\coloneqq Y$, $y_2\coloneqq Z$, and $y_3\coloneqq W$.
\begin{lemma}\label{hessian}
For any $\mathbf{y}\in\mathscr{U}$, the holomorphic Hessian of $S$ is given by
\begin{equation*}
\mathrm{Hess}(S)(\mathbf{y})=\left(\frac{\partial^2S}{\partial y_j\partial y_k}\right)_{j,k\in \{1,2,3\}}(\mathbf{y})=2iQ+i
\begin{pmatrix}
\frac{-1}{1+e^{-y_1}} &  &  \\
 & \frac{1}{1+e^{-y_2}} & \\
  &  &  \frac{-3}{1+e^{-y_3}}
\end{pmatrix}.
\end{equation*}

Furthermore, for any $\mathbf{y}\in\mathscr{U}$, the determinant of $\mathrm{Hess}(S)(\mathbf{y})$ is not $0$.
\end{lemma}
\begin{proof}
The first equation is derived by the second derivative of $S$ and by the fact that for any $y\in\mathbb{R}\pm i(0,\pi)$
\begin{equation*}
\frac{\partial \mathrm{Li}_2(-e^y)}{\partial y}=-\mathrm{Log}(1+e^y)
\end{equation*}
holds.

Let $\mathbf{y}\in\mathscr{U}$, then $Q$ is a real symmetric matrix, so $\Im (\mathrm{Hess}(S)(\mathbf{y}))$ is a symmetric matrix.
Furthermore,
\begin{equation*}
\Re(\mathrm{Hess}(S)(\mathbf{y}))=\begin{pmatrix}
-\Im \left(\frac{-1}{1+e^{-y_1}}\right) & & \\
 & -\Im \left(\frac{1}{1+e^{-y_2}}\right) & \\
 &   &  -\Im \left(\frac{-3}{1+e^{-y_3}}\right) 
\end{pmatrix},
\end{equation*}
and all the diagonal components are negative.
In fact, let $y_1=a+bi\quad(a,b\in\mathbb{R})$, then $b\in(-\pi,0)$ and 
\begin{eqnarray*}
-\Im \left(\frac{-1}{1+e^{-y_1}}\right)&=&\Im \frac{1}{1+e^{-a}(\cos{b}-i\sin{b})}\\
&=&\frac{e^{-a}\sin{b}}{(1+e^{-a}\cos{b})^2+(e^{-a}\sin{b})^2}<0.
\end{eqnarray*} The same applies to the other diagonal components.
Therefore, since $-\mathrm{Hess}(S)(\mathbf{y})$ is invertible by Lemma \ref{invertible}, $\mathrm{Hess}(S)(\mathbf{y})$ is also invertible and the lemma is proven.
\end{proof}
On the other hand, for a tetrahedron $T$ with a shape structure, there exists the following diffeomorphism $\psi_T$.
\begin{equation*}
\psi_T:\mathbb{R}+i\mathbb{R}_{>0}\rightarrow\mathbb{R}-i\epsilon(T)(0,\pi),\quad z\mapsto\epsilon(T)(\mathrm{Log}(z)-i\pi)
\end{equation*}
The inverse map is given by
\begin{equation*}
\psi_T^{-1}:\mathbb{R}-i\epsilon(T)(0,\pi)\rightarrow \mathbb{R}+i\mathbb{R}_{>0},\quad y\mapsto -\exp(\epsilon(T)y).
\end{equation*}
\begin{lemma}\label{completeness-bijection}
Let us consider the following diffeomorphism
\begin{equation*}
\psi\coloneqq\left(\prod_{T\in\{T_1,T_2,T_3\}}\psi_T \right):(\mathbb{R}+i\mathbb{R}_{>0})^{3}\rightarrow \mathscr{U}.    
\end{equation*}
Here, we note that $\epsilon(T_1)=\epsilon(T_3)=1$, and $\epsilon(T_2)=-1$ according to Figure \ref{fig:ideal_triangulation}.

The map $\psi$ induces the bijective map between $\{\mathbf{z}=(z_1,z_2,z_3)\in(\mathbb{R}+i\mathbb{R}_{>0})^3\mid\eqref{e4-complex-gluing'}\land\eqref{e2-complex-gluing''}\land\eqref{e5-complex-gluing''}\}$ and $\{\mathbf{y}\in\mathscr{U}\mid\nabla S(\mathbf{y})=0\}$.
In particular, $S$ has the unique critical point on $\mathscr{U}$ corresponding to the unique complete hyperbolic structure $\mathbf{z}^0$ on the geometric ideal tetrahedral decomposition $X$. Here, we define $z_4^0=z_3^0$ and $z_5^0={z_3^{\prime\prime}}^0$. 
\end{lemma}
\begin{proof}
For any $\mathbf{y}\in\mathscr{U}$,
\begin{equation*}
\nabla S(\mathbf{y})=\begin{pmatrix}
\partial_{y_1} S(\mathbf{y})\\
\partial_{y_2} S(\mathbf{y})\\
\partial_{y_3} S(\mathbf{y})
\end{pmatrix}
=2iQ\mathbf{y}+\mathscr{W}+i\begin{pmatrix}
-\mathrm{Log}(1+e^{y_1})\\
\mathrm{Log}(1+e^{y_2})\\
-3\mathrm{Log}(1+e^{y_3})
\end{pmatrix}.
\end{equation*}
Let $y_1=\psi_{T_1}(z_1)$, $y_3=\psi_{T_3}(z_3)$, then 
\begin{eqnarray*}
\mathrm{Log}(z_1)=y_1+i\pi,&\quad\mathrm{Log}(z_1^\prime)=-\mathrm{Log}(1+e^{y_1}),&\quad\mathrm{Log}(z_1^{\prime\prime})=\mathrm{Log}(1+e^{-y_1})\\
\mathrm{Log}(z_3)=y_3+i\pi,&\quad\mathrm{Log}(z_3^\prime)=-\mathrm{Log}(1+e^{y_3}),&\quad \mathrm{Log}(z_3^{\prime\prime})=\mathrm{Log}(1+e^{-y_3})
\end{eqnarray*}
and let $y_2=\psi_{T_2}(z_2)$, then 
\begin{equation*}
\mathrm{Log}(z_2)=-y_2+i\pi,\quad \mathrm{Log}(z_2^\prime)=-\mathrm{Log}(1+e^{-y_2}),\quad \mathrm{Log}(z_2^{\prime\prime})=\mathrm{Log}(1+e^{y_2}).
\end{equation*}
Since
\begin{eqnarray*}
\nabla S(\mathbf{y})&=i\begin{pmatrix}
2 & -1 & 0\\
-1 & 0 &1 \\
0 & 1 & 1
\end{pmatrix}\begin{pmatrix}
y_1\\
y_2\\
y_3
\end{pmatrix}
+\begin{pmatrix}
-\pi\\
0\\
\pi
\end{pmatrix}
+i\begin{pmatrix}
-\mathrm{Log}(1+e^{y_1})\\
\mathrm{Log}(1+e^{y_2})\\
-3\mathrm{Log}(1+e^{y_3})
\end{pmatrix},
\end{eqnarray*}

\begin{eqnarray*}
\nabla S(\psi(\mathbf{z}))&=i\begin{pmatrix}
2\mathrm{Log}z_1+\mathrm{Log}z_2+\mathrm{Log}z_1^\prime-2i\pi\\
-\mathrm{Log}z_1+\mathrm{Log}z_3+\mathrm{Log}z_2^{\prime\prime}\\
-\mathrm{Log}z_2+\mathrm{Log}z_3+3\mathrm{Log}z_3^\prime-i\pi
\end{pmatrix}.
\end{eqnarray*}
Let 
\begin{equation*}
A=\begin{pmatrix}
1 & 1 & \\
 & 1 & \\
 & -1 & 1
\end{pmatrix}.    
\end{equation*}
Since $|A|=1$, $A$ is a regular matrix and
\begin{eqnarray*}
A\cdot(\nabla S(\psi(\mathbf{z}))&=&i\begin{pmatrix}
\mathrm{Log}z_1+\mathrm{Log}z_1^\prime+\mathrm{Log}z_2+\mathrm{Log}z_2^{\prime\prime}+\mathrm{Log}z_3-2i\pi\\
-\mathrm{Log}z_1+\mathrm{Log}z_3+\mathrm{Log}z_2^{\prime\prime}\\
\mathrm{Log}z_1-\mathrm{Log}z_2-\mathrm{Log}z_2^{\prime\prime}+3\mathrm{Log}z_3^\prime-i\pi
\end{pmatrix}\\
&=&i\begin{pmatrix}
\mathrm{Log}z_3-\mathrm{Log}z_1^{\prime\prime}-\mathrm{Log}z_2^\prime\\
-\mathrm{Log}z_1+\mathrm{Log}z_3+\mathrm{Log}z_2^{\prime\prime}\\
\mathrm{Log}z_1+\mathrm{Log}z_2^\prime+3\mathrm{Log}z_3^\prime-2i\pi
\end{pmatrix}.
\end{eqnarray*}
In this deformation, we used the fact that $\mathrm{Log}z_j+\mathrm{Log}z_j^\prime+\mathrm{Log}z_j^{\prime\prime}=i\pi\quad(j=1,2,3)$.

Therefore, since satisfying $\mathbf{z}\in(\mathbb{R}+i\mathbb{R}_{>0})^3$ and $\eqref{e4-complex-gluing'}\land\eqref{e2-complex-gluing''}\land\eqref{e5-complex-gluing''}$ is equivalent to satisfying $\psi(\mathbf{z})\in\mathscr{U}$ and $\nabla S(\psi(\mathbf{z}))=\mathbf{0}$, the lemma is proven.
\end{proof}
Let
\begin{equation*}
\mathscr{Y}^0=\mathscr{Y}_{\alpha_X^0}=(\mathbb{R}-i\pi(1-2a_1^{0}))\times(\mathbb{R}+i\pi(1-2a_2^{0}))\times(\mathbb{R}-i\pi(1-2a_3^{0})),
\end{equation*}
where $\alpha_X^0\in\mathscr{A}_X$ is the complete hyperbolic angle structure corresponding to the complete hyperbolic complex shape structure $\mathbf{z}^0$.

Let $\mathbf{y}\in\mathscr{Y}^0$ be parametric represented by 
\begin{equation*}
\mathbf{y}=\begin{pmatrix}
y_1\\
y_2\\
y_3
\end{pmatrix}
=\begin{pmatrix}
x_1+id_1^0\\
x_2+id_2^0\\
x_3+id_3^0
\end{pmatrix}
=\mathbf{x}+i\mathbf{d}^0,
\end{equation*}
where $d_1^0=-\pi(1-2a_1^{0})<0$, $d_2^0=\pi(1-2a_2^{0})>0$, $d_3^0=-\pi(1-2a_3^{0})<0$. In addition, let $\mathbf{y}^0\coloneqq\psi(z_1^0,z_2^0,z_3^0)\in\mathscr{Y}^0$. $\mathscr{Y}^0=\mathbb{R}^3+i\mathbf{d}^0\subset \mathbb{C}^3$ is a $\mathbb{R}$-affine subspace of $\mathbb{C}^3$.
\subsection{The concavity of $\Re S$ on each contour $\mathscr{Y}_\alpha$}
\begin{lemma}\label{RS-concave}
For any $\alpha\in\mathscr{A}_X$, the function $\Re S:\mathscr{Y}_\alpha\rightarrow \mathbb{R}$ is strictly concave on $\mathscr{Y}_\alpha$.
\end{lemma}
\begin{proof}
Let $\alpha\in\mathscr{A}_X$. 
Since $\Re S:\mathscr{Y}_\alpha\rightarrow\mathbb{R}$ is twice continuously differentiable as a function of three real variables, 
it is sufficient to show that the Hessian $(\Re S|_{\mathscr{Y}_\alpha})^{\prime\prime}$ is negative definite for all $\mathbf{x}+i\mathbf{d}\in\mathscr{Y}_\alpha$.
Since this Hessian is equal to the real part of the holomorphic Hessian of $S$, by the calculation of Lemma \ref{hessian}, for all $\mathbf{x}\in\mathbb{R}^3$, the following holds.
\begin{eqnarray*}
(\Re S|_{\mathscr{Y}_\alpha})^{\prime\prime}(\mathbf{x}+i\mathbf{d})&=&\Re(\mathrm{Hess}(S)(\mathbf{x}+i\mathbf{d}))\\
&=&\begin{pmatrix}
-\Im\left(\frac{-1}{1+e^{-x_1-id_1}}\right) & & \\
 & -\Im\left(\frac{1}{1+e^{-x_2-id_2}}\right) & \\
 & & -\Im (\frac{-3}{1+e^{-x_3-id_3}})
\end{pmatrix}.
\end{eqnarray*}
Since $d_1,d_3\in(-\pi,0)\quad d_2\in(0,\pi)$, all diagonal components of this matrix are negative.
In fact, 
\begin{equation*}
-\Im\left(\frac{-1}{1+e^{-x_1-id_1}}\right)=\frac{e^{-x_1}\sin{d_1}}{(1+e^{-x_1}\cos{d_1})^2+(e^{-x_1}\sin{d_1})^2}<0,
\end{equation*}
\begin{equation*}
-\Im\left(\frac{1}{1+e^{-x_2-id_2}}\right)=\frac{-e^{-x_2}\sin{d_2}}{(1+e^{-x_2}\cos{d_2})^2+(e^{-x_2}\sin{d_2})^2}<0
\end{equation*}
and the same holds for the other diagonal component. 

Therefore, $(\Re S|_{\mathscr{Y}_\alpha})^{\prime\prime}$ is negative definite at any point on $\mathscr{Y}_\alpha$, and $\Re S|_{\mathscr{Y}_\alpha}$ is strictly concave.
\end{proof}
\subsection{Properties of $\Re S$ on the contour $\mathscr{Y}^0$ corresponding to the complete hyperbolic structure}
The following lemma holds as well as Lemma 7.6 in \cite{MR3945172}.
\begin{lemma}\label{strict-max-of-ReS}
The function $\Re S:\mathscr{Y}^0\rightarrow\mathbb{R}$ has a strictly global maximum at $\mathbf{y}^0\in\mathscr{Y}^0$.
\end{lemma}
\begin{proof}
Since the holomorphic gradient of $S:\mathscr{U}\rightarrow \mathbb{C}$ vanishes at $\mathbf{y}^0$ by Lemma \ref{completeness-bijection}, the real gradient of $\Re S|_{\mathscr{Y}^0}$, which is the real part of the holomorphic gradient of $S$, also vanishes at $\mathbf{y}^0$. Therefore $\mathbf{y}^0$ is the critical point of $\Re S|_{\mathscr{Y}^0}$.
Furthermore, since $\Re S|_{\mathscr{Y}^0}$ is strictly concave by Lemma \ref{RS-concave}, $\Re S|_{\mathscr{Y}^0}$ takes the global maximum at $\mathbf{y}^0$.
\end{proof}
\begin{lemma}\label{rewrite-S}
The function $S:\mathscr{U}\rightarrow \mathbb{C}$ can be rewritten as
\begin{equation*}
S(\mathbf{y})=i\mathrm{Li}_2(-e^{y_1})+i\mathrm{Li}_2(-e^{-y_2})+3i\mathrm{Li}_2(-e^{y_3})+i\mathbf{y}^{\rm T}Q\mathbf{y}+i\frac{y_2^2}{2}+\mathbf{y}^{\rm T}\mathscr{W}+i\frac{\pi^2}{6}.
\end{equation*}
\end{lemma}
\begin{proof}
It is proven by transforming $S(\mathbf{y})$ using the inversion relation formula of Theorem \ref{classical-dilog} (1) where $z=-e^{y_2}$.
\end{proof}
\begin{lemma}\label{relation-between-S-and-Vol}
\begin{equation*}
\Re(S)(\mathbf{y}^0)=-\mathrm{Vol}(S^3\backslash 7_3)    
\end{equation*}
holds.
\end{lemma}
\begin{proof}
By Lemma \ref{rewrite-S}, for all $\mathbf{y}\in\mathscr{U}$,
\begin{equation*}
\Re(S)(\mathbf{y})=-\Im(\mathrm{Li}_2(-e^{y_1}))-\Im(\mathrm{Li}_2(-e^{y_2}))-3\Im(\mathrm{Li}_2(-e^{y_3}))-\Im\left(\mathbf{y}^{\rm T}Q\mathbf{y}+\frac{y_2^2}{2}\right)+
\Re(\mathbf{y}^{\rm T}\mathscr{W}) .
\end{equation*}

For $z\in\mathbb{R}+i\mathbb{R}_{>0}$, the hyperbolic volume of the ideal hyperbolic tetrahedron with the complex shape $z$ can be expressed using the Bloch-Wigner function $D$ as 
$D(z)=\Im(\mathrm{Li}_2(z))+\arg(1-z)\log|z|$.

For $z_1=-e^{y_1}=-e^{x_1+id_1}$, by the relations in Section \ref{complex-gluing-equation}, $\arg(1-z_1)\log|z_1|=-2\pi c_1 x_1$.

Similarly, for $z_2=-e^{-y_2}=-e^{-x_2-id_2}$, $\arg(1-z_2)\log|z_2|=2\pi b_2 x_2$ and for $z_3=-e^{y_3}=-e^{x_3+id_3}$, $\arg(1-z_3)\log|z_3|=-2\pi c_3 x_3$.

Therefore, for $\mathbf{y}\in\mathscr{U}$,
\begin{equation*}
\Re(S)(\mathbf{y})=-D(z_1)-2\pi c_1 x_1-D(z_2)+2\pi b_2 x_2-3D(z_3)-6\pi c_3 x_3-2\mathbf{x}^{\rm T}Q\mathbf{d}-x_2d_2+\mathbf{x}^{\rm T}\mathscr{W} . 
\end{equation*}

Since $\mathbf{z}^0$ is the complex shape structure corresponding to the complete hyperbolic structure on the ideal tetrahedral decomposition $X$, and  $z_3^0=z_4^0$, $z_5^0={z_3^{\prime\prime}}^0$,
using the properties of the Bloch-Wigner function described in Theorem \ref{bloch-hyperbolic-volume}, we can obtain
\begin{eqnarray*}
-\mathrm{Vol}(S^3\backslash 7_3)&=& -D(z_1^0)-D(z_2^0)-D(z_3^0)-D(z_4^0)-D(z_5^0)\\
&=&-D(z_1^0)-D(z_2^0)-D(z_3^0)-D(z_3^0)-D({z_3^{\prime\prime}}^0)\\
&=&-D(z_1^0)-D(z_2^0)-3D(z_3^0).
\end{eqnarray*}

Therefore, let
\begin{equation*}
\mathcal{J}:=\begin{pmatrix}
-2\pi c_1^{0}\\
2\pi b_2^{0}\\
-6\pi c_3^{0}\\
\end{pmatrix}
+\mathscr{W}
-2Q\mathbf{d}^0+\begin{pmatrix}
0\\
-d_2^0\\
0
\end{pmatrix},
\end{equation*}
then it is sufficient to prove $(\mathbf{x}^0)^{\rm T}\cdot\mathcal{J}=0$.

Since $d_2^0=\pi(1-2a_2^{0})=2\pi(b_2^{0}+c_2^{0})$, $2\pi b_2^{0}-d_2^0=-2\pi c_2^{0}$. Thus,
\begin{eqnarray*}
\mathcal{J}=-\begin{pmatrix}
2\pi c_1^{0}\\
2\pi c_2^{0}\\
6\pi c_3^{0}
\end{pmatrix}
+\begin{pmatrix}
-\pi\\
0\\
\pi
\end{pmatrix}+\begin{pmatrix}
-2 & 1 & \\
1 & 0 & -1 \\
 & -1 & -1 
\end{pmatrix}
\begin{pmatrix}
d_1^0\\
d_2^0\\
d_3^0
\end{pmatrix}.
\end{eqnarray*}

Since $d_1^0=-\pi(1-2a_1^{0})$, $d_2^0=\pi(1-2a_2^{0})$, $d_3^0=-\pi(1-2a_3^0)$,
\begin{eqnarray*}
\mathcal{J}&=&\begin{pmatrix}
-2\pi c_1^{0}-\pi +2\pi(1-2a_1^{0})+\pi-2\pi a_2^{0}\\
-2\pi c_2^{0}-\pi+2\pi a_1^{0}+\pi-2\pi a_3^{0}\\
-6\pi c_3^{0}+\pi-\pi+2\pi a_2^{0}+\pi-2\pi a_3^{0}
\end{pmatrix}
=\begin{pmatrix}
2\pi-2\pi c_1^{0}-4\pi a_1^{0}-2\pi a_2^{0}\\
-2\pi c_2^{0}+2\pi a_1^{0}-2\pi a_3^{0}\\
\pi-6\pi c_3^{0}+2\pi a_2^{0}-2\pi a_3^{0}
\end{pmatrix}\\
&=&\begin{pmatrix}
4\pi(b_1^0+c_1^{0})-2\pi c_1^{0}-2\pi a_2^{0}\\
-2\pi c_2^{0}+2\pi a_1^{0}-2\pi a_3^{0}\\
\pi-6\pi c_3^{0}+2\pi a_2^{0}-2\pi a_3^{0}
\end{pmatrix}
=\begin{pmatrix}
2\pi (2b_1^{0}+c_1^{0}-a_2^{0})\\
2\pi (-c_2^{0}+a_1^{0}-a_3^{0})\\
2\pi(\frac{1}{2}-3c_3^{0}+a_2^{0}-a_3^{0})

\end{pmatrix}.
\end{eqnarray*}

$z_3^0=z_4^0$ leads to $a_3^{0}=a_4^{0}$, $b_3^{ 0}=b_4^{ 0}$, $c_3^{ 0}=c_4^{ 0}$ and $z_5^0={z_3^{\prime\prime}}^0$ leads to $a_3^0=c_5^0$, $c_3^0=b_5^0$, $b_3^0=a_5^0$.

Therefore, by the constraints in the definition of $\mathscr{A}_X$,
\begin{equation*}
\left\{
\begin{aligned}
b_1^{0}+b_2^{0}&=a_3^{0}\\
a_2^{0}&=2b_1^{0}+c_1^{0}\\
a_2^{0}+b_3^{0}&=2c_3^{0}\\
\end{aligned}
\right.
\end{equation*}
holds.

Then \begin{eqnarray*}
-c_2^{0}+a_1^{0}-a_3^{0}&=&-c_2^{0}+a_1^{0}-(b_1^{0}+b_2^{0})\quad(\because b_1^{0}+b_2^{0}=a_3^{0})\\
&=&-\left(\frac{1}{2}-a_2^{0}\right)+a_1^{0}-b_1^{0}\\
&=& -\frac{1}{2}+2b_1^{0}+c_1^{0}+a_1^{0}-b_1^{0} \quad (\because a_2^{0}=2b_1^{0}+c_1^{0})\\
&=&0,
\end{eqnarray*}
\begin{eqnarray*}
\frac{1}{2}-3c_3^{0}+a_2^{0}-a_3^{0}&=&\frac{1}{2}-c_3^{0}-(a_2^{0}+b_3^{0})+a_2^{0}-a_3^{0}\quad (\because a_2^{0}+b_3^{0}=2c_3^{0})\\
&=&\frac{1}{2}-c_3^{0}-b_3^{0}-a_3^{0}=0,
\end{eqnarray*}
thus, $\mathcal{J}=0$. The above proves the lemma.
\end{proof}
Thereafter, let $r_0>0$ and $\gamma =\left\{\mathbf{y}\in\mathscr{Y}^0\mid\|\mathbf{y}-\mathbf{y}^0\|\le r_0\right\}$ be the $3$-dimensional ball containing $\mathbf{y}^0$ in $\mathscr{Y}^0$.

\subsection{Asymptotic expansion of the integral on $\mathscr{Y}^0$}
The following is proven in a similar way to the proof of Proposition 7.9. in \cite{MR3945172}.
\begin{prop}\label{rho}

There exists a constant $\rho\in\mathbb{C}^\ast$ such that the following holds.

If $\lambda\rightarrow \infty$,
\begin{equation*}
\int_{\gamma} d\mathbf{y} e^{\lambda S(\mathbf{y})}=\rho \lambda^{-\frac{3}{2}}\exp(\lambda S(\mathbf{y^0}))(1+o_{\lambda\rightarrow\infty}(1)).
\end{equation*}

In particular,
\begin{equation*}
\frac{1}{\lambda}\log\left|\int_{\gamma}d\mathbf{y}e^{\lambda S(\mathbf{y})}\right|\underset{\lambda\rightarrow\infty}{\longrightarrow} \Re S(\mathbf{y}^0)=-\mathrm{Vol}(S^3\backslash 7_3).
\end{equation*}
\end{prop}
\begin{proof}
We apply the saddle point method of Theorem \ref{Fedoryuk}. In Theorem \ref{Fedoryuk}, let $m=3$, $\gamma^3=\gamma$, $z=\mathbf{y}$, $z^0=\mathbf{y}^0$, $D=\mathscr{U}$, $f=1$ and $S$ is defined by the equation \eqref{definition-of-S}. We check the conditions for the application of Theorem \ref{Fedoryuk}.

1) $\mathbf{y}^0$ is the interior point of $\gamma$ by the configuration.

2) $\max_{\gamma} \Re S$ is attained only at $\mathbf{y}^0$ by Lemma \ref{strict-max-of-ReS}.

3) By Lemma \ref{completeness-bijection}, $\nabla S(\mathbf{y}^0)=0$.

4) By Lemma \ref{hessian}, $\det\mathrm{Hess}(\mathbf{y}^0)\neq 0$.

Therefore, by setting $\rho\coloneqq \frac{(2\pi)^\frac{3}{2}}{\sqrt{\det\mathrm{Hess}(S)(\mathbf{y}^0)}}\in\mathbb{C}^\ast$, the first statement is proven by Theorem \ref{Fedoryuk}.

The second statement follows by Lemma \ref{relation-between-S-and-Vol}.
\end{proof}
Next, we calculate the remainder term, that is, the upper bound of the integral on $\mathscr{Y}^0\backslash\gamma$, which excludes the compact ball from the unbounded integral domain.

The following is proven in the same way as the proof of Lemma 7.10 in \cite{MR3945172}.
\begin{lemma}\label{upper-bound}
There exist constants $A,B >0$ such that the following holds. For all $\lambda > A$,
\begin{equation*}
\left|\int_{\mathscr{Y}^0\backslash \gamma}d\mathbf{y}e^{\lambda S(\mathbf{y})}\right|\le Be^{\lambda M},
\end{equation*}
where $M\coloneqq \max_{\partial\gamma}\Re S$.
\end{lemma}
\begin{proof}
First, we perform a variable transformation to $3$-dimensional spherical coordinates.

\begin{equation*}
\mathbf{y}\in\mathscr{Y}^0\backslash\gamma\iff r\Vec{\eta}\in(r_0,\infty)\times\mathbb{S}^{2}
\end{equation*}

Then for all $\lambda>0$, 
\begin{equation*}
\int_{\mathscr{Y}^0\backslash\gamma}d\mathbf{y}e^{\lambda S(\mathbf{y})}=\int_{\mathbb{S}^2}d \mathrm{vol}_{\mathbb{S}^2}\int_{r_0}^\infty r^2e^{\lambda S(r\Vec{\eta})}dr.
\end{equation*}

Therefore, for all $\lambda >0$,
\begin{equation*}
\left|\int_{\mathscr{Y}^0\backslash\gamma}d\mathbf{y}e^{\lambda S(\mathbf{y})}\right|\le 4\pi \sup_{\Vec{\eta}\in\mathbb{S}^2}\int_{r_0}^\infty 
r^2 e^{\lambda \Re(S)(r\Vec{\eta})}dr. 
\end{equation*}
Let $\Vec{\eta}\in\mathbb{S}^2$ and $f=f_{\Vec{\eta}}\coloneqq (r\mapsto \Re(S)(r\Vec{\eta}))$ be the function which restricts $\Re (S)$ on $(r_0,\infty)\Vec{\eta}$.
We find an upper bound for $\int_{r_0}^\infty r^2 e^{\lambda f(r)}$.
By Lemma \ref{RS-concave}, since $\Re(S)$ is strictly concave and $f$ is its restriction on the convex set, $f$ is also strictly concave on $(r_0,\infty)$ (even on $[0,\infty)$).

The function which represents the slope $N:[r_0,\infty)\rightarrow \mathbb{R}$ is defined as $N(r)\coloneqq \frac{f(r)-f(r_0)}{r-r_0}$ for $r>r_0$, $N(r_0):=f^\prime(r_0)$.

The function $N$ is of class $C^1$ and satisfies $N^\prime(r):=\frac{f^\prime(r)-N(r)}{r-r_0}$ for $r>r_0$.
Since $f$ is strictly concave, for any $r\in(r_0,\infty)$, 
$f^\prime(r)<N(r)$, and $N$ decreases on this interval.
Therefore,
\begin{equation*}
\int_{r_0}^\infty r^2 e^{\lambda f(r)}dr=e^{\lambda f(r_0)}\int_{r_0}^\infty r^2e^{\lambda N(r)(r-r_0)}dr\le e^{\lambda f(r_0)}\int_{r_0}^\infty r^2 e^{\lambda N(r_0)(r-r_0)}dr.
\end{equation*}
Note that $N(r_0)=f^\prime(r_0)<0$ by Lemma \ref{RS-concave}, \ref{strict-max-of-ReS}. 

By repeating the integration by parts, we obtain 
\begin{equation*}
\int_{r_0}^\infty r^2 e^{\lambda N(r_0)(r-r_0)}dr=\frac{1}{(\lambda N(r_0))^3}\sum_{k=0}^{2}(-1)^{1-k}\frac{2!}{k!}(\lambda N(r_0))^kr_0^k.
\end{equation*}

Furthermore, since $N(r_0)=f^\prime(r_0)=\langle(\nabla \Re(S)(r_0\Vec{\eta}));\Vec{\eta}\rangle$ and $S$ is holomorphic, $(\Vec{\eta}\mapsto N(r_0)=f^\prime_{\Vec{\eta}}(r_0))$ is a continuous map from $\mathbb{S}^2$ to $\mathbb{R}_{<0}$. Thus, there exist some constants $m_1, m_2>0$ such that  
for all $\Vec{\eta}\in\mathbb{S}^2$, $0<m_1\le |N(r_0)|\le m_2$.

Therefore, the following inequality holds for all $\lambda > \frac{1}{m_1r_0}$.
\begin{eqnarray*}
\int_{r_0}^\infty r^{2}e^{\lambda f(r)}dr &\le& e^{\lambda f(r_0)}\frac{1}{(\lambda N(r_0))^3}\sum_{k=0}^{2}(-1)^{1-k}\frac{2!}{k!}(\lambda N(r_0))^kr_0^k\\
&\le & e^{\lambda f(r_0)}\left|\frac{1}{(\lambda N(r_0))^3}\sum_{k=0}^{2}(-1)^{1-k}\frac{2!}{k!}(\lambda N(r_0))^kr_0^k\right|\\
&\le & e^{\lambda f(r_0)}\frac{1}{|\lambda N(r_0)|^3}\sum_{k=0}^{2}2!|\lambda N(r_0)r_0|^k\\
&\le & e^{\lambda f(r_0)}\frac{3!|\lambda N(r_0)r_0|^3}{|\lambda N(r_0)|^3}=3!r_0^3e^{\lambda f(r_0)}.
\end{eqnarray*}
For all $\Vec{\eta}\in\mathbb{S}^2$, there is a constant $C>0$ independent of $\lambda$ and $\Vec{\eta}$ such that for all $\lambda>\frac{1}{m_1r_0}$, $\int_{r_0}^\infty r^2 e^{\lambda f_{\Vec{\eta}}(r)}dr\le Ce^{\lambda f_{\Vec{\eta}}(r_0)}$. Thus, let $M\coloneqq \max_{\partial\gamma}\Re S$, then for all $\lambda >\frac{1}{m_1r_0}$,
\begin{equation*}
\left|\int_{\mathscr{Y}^0\backslash \gamma}d\mathbf{y}e^{\lambda S(\mathbf(y)}\right|\le 4\pi \sup_{\Vec{\eta}\in\mathbb{S}^2}\int_{r_ 0}^\infty r^2e^{\lambda \Re(S)(r\Vec{\eta})}dr\le 4\pi Ce^{\lambda M}
\end{equation*}
holds.
Thus, by setting $A:=\frac{1}{m_1r_0}$ and $B:=4\pi C$, the lemma is proven. 
\end{proof}
The following is proven in a similar way to the proof of Proposition 7.11 in \cite{MR3945172}.
\begin{prop}
For $\rho\in\mathbb{C}^\ast$, which is identical to the one in Proposition \ref{rho},
if $\lambda\rightarrow\infty$, 
\begin{equation*}
\int_{\mathscr{Y}^0}d\mathbf{y} e^{\lambda S(\mathbf{y})}=\rho\lambda^{-\frac{3}{2}}\exp{(\lambda S(\mathbf{y}^0))}(1+o_{\lambda\rightarrow\infty}(1)).
\end{equation*}
In particular, 
\begin{equation*}
\frac{1}{\lambda}\log{\left|\int_{\mathscr{Y}^0}d\mathbf{y} e^{\lambda S(\mathbf{y})}\right|}\underset{\lambda\rightarrow\infty}{\longrightarrow}\Re S(\mathbf{y}^0)=-\mathrm{Vol}(S^3\backslash 7_3).
\end{equation*}
\end{prop}
\begin{proof}
As in the proof of Proposition \ref{rho}, the second statement follows by the first.

We prove the first statement.

By Lemma \ref{upper-bound}, for all $\lambda > A$,
$\left|\int_{\mathscr{Y}^0\backslash\gamma}d\mathbf{y}e^{\lambda S(\mathbf{y})}\right|\le Be^{\lambda M}$
holds. 

By Lemma \ref{RS-concave}, \ref{strict-max-of-ReS}, since $M<\Re (S)(\mathbf{y}^0)$,
\begin{equation*}
\int_{\mathscr{Y}^0\backslash \gamma}d\mathbf{y} e^{\lambda S(\mathbf{y})}=o_{\lambda\rightarrow\infty}(\lambda^{-\frac{3}{2}}\exp{(\lambda S(\mathbf{y}^0))}).   
\end{equation*}

Therefore, the first statement follows by Proposition \ref{rho} and the following equation.
\begin{equation*}
\int_{\mathscr{Y}^0} d\mathbf{y}e^{\lambda S(\mathbf{y})}=\int_{\gamma} d\mathbf{y}e^{\lambda S(\mathbf{y})}+\int_{\mathscr{Y}^0\backslash \gamma}d\mathbf{y}e^{\lambda S(\mathbf{y})}.
\end{equation*}
\end{proof}
\subsection{Extension of the asymptotic expansion to the quantum dilogarithm}
For $\mathsf{b}>0$, a new potential function $S_{\mathsf{b}}:\mathscr{U}\rightarrow\mathbb{C}$, which is a holomorphic function of three complex variables, is defined by
\begin{equation*}
S_{\sf b}(\mathbf{y}):=i\mathbf{y}^{\rm T}Q\mathbf{y}+\mathbf{y}^{\rm T}\mathscr{W}+2\pi\mathsf{b}^2\mathrm{Log}\left(\frac{\Phi_{\mathsf{b}}(\frac{y_2}{2\pi\mathsf{b}})}{\Phi_{\sf b}(\frac{y_1}{2\pi\mathsf{b}})\Phi_{\mathsf{b}}(\frac{y_3}{2\pi\mathsf{b}})^3}\right).
\end{equation*}

\begin{lemma}[ \cite{MR3945172}, Lemma 7.12 ]
For all $\mathsf{b}\in(0,1)$, for all $y\in\mathbb{R}+i(0,\pi)$,
\begin{equation*}
\Re\left(\mathrm{Log}\left(\Phi_{\sf b}\left(\frac{-\overline{y}}{2\pi\mathsf{b}}\right)\right)-\left(\frac{-i}{2\pi\mathsf{b}^2}\mathrm{Li}_2(-e^{-\overline{y}})\right)\right)
=\Re\left(\mathrm{Log}\left(\Phi_{\sf b}\left(\frac{y}{2\pi\mathsf{b}}\right)\right)-\left(\frac{-i}{2\pi\mathsf{b}^2}\mathrm{Li}_2(-e^y)\right)\right).
\end{equation*}
\end{lemma}
\begin{lemma}[ \cite{MR3945172}, Lemma 7.13 ]\label{upperbound-of-dif}
For all $\delta>0$, there exists a constant $B_{\delta}>0$ such that the following holds.
For all $\mathsf{b}\in(0,1)$, for all $y\in\mathbb{R}+i[\delta,\pi-\delta]$,
\begin{equation*}
\left|\Re\left(\mathrm{Log}\left(\Phi_{\mathsf{b}}\left(\frac{y}{2\pi\mathsf{b}}\right)\right)-\left(\frac{-i}{2\pi\mathsf{b}^2}\mathrm{Li}_2(-e^y)\right)\right)\right|\le B_{\delta}\mathsf{b}^2.
\end{equation*}

Furthermore, $B_{\delta}$ is expressed in the form $B_\delta=\frac{C}{\delta}+C^\prime$ for $C$, $C^\prime>0$.
\end{lemma}
\begin{lemma}[ \cite{MR3945172}, Lemma 7.14 ]\label{upperbound-of-dif-negative}
For all $\delta > 0$, there exists a constant $B_\delta>0$ (which is identical to the one in Lemma \ref{upperbound-of-dif}) such that the following hods.

For all $\mathsf{b}\in(0,1)$, for all $y\in\mathbb{R}-i[\delta,\pi-\delta]$, 
\begin{equation*}
\left|\Re\left(\mathrm{Log}\left(\Phi_{\mathsf{b}}\left(\frac{y}{2\pi\mathsf{b}}\right)\right)-\left(\frac{-i}{2\pi\mathsf{b}^2}\mathrm{Li}_2(-e^y)\right)\right)\right|\le B_\delta \mathsf{b}^2.
\end{equation*}
\end{lemma}
The following holds in a similar way to the proof of Proposition 7.15 in \cite{MR3945172}.
\begin{prop}\label{limit-Sb}
There exists a constant $\rho^\prime\in\mathbb{C}^\ast$ such that the following holds.
If $\mathsf{b}\rightarrow 0^+$,
\begin{eqnarray*}
\int_{\mathscr{Y}^0}d\mathbf{y}e^{\frac{1}{2\pi\mathsf{b}^2}S_{\mathsf{b}}(\mathbf{y})}&=&\int_{\mathscr{Y}^0}d\mathbf{y}e^{\frac{i\mathbf{y}^{\rm T}Q\mathbf{y}+\mathbf{y}^{\rm T}\mathscr{W}}{2\pi\mathsf{b}^2}}\frac{\Phi_{\mathsf{b}}\left(\frac{y_2}{2\pi\mathsf{b}}\right)}{\Phi_{\mathsf{b}}\left(\frac{y_1}{2\pi\mathsf{b}}\right)\Phi_{\mathsf{b}}\left(\frac{y_3}{2\pi\mathsf{b}}\right)^3}\\
&=& e^{\frac{1}{2\pi\mathsf{b}^2}S(\mathbf{y}^0)}\left(\rho^\prime\mathsf{b}^3\left(1+o_{\mathsf{b}\rightarrow 0^+}(1)\right)+O_{\mathsf{b}\rightarrow 0^{+}}(1)\right).
\end{eqnarray*}
In particular, 
\begin{equation*}
2\pi\mathsf{b}^2\log\left|\int_{\mathscr{Y}^0}d\mathbf{y}e^{\frac{1}{2\pi\mathsf{b}^2}S_{\mathsf{b}}(\mathbf{y})}\right|\underset{\mathsf{b}\rightarrow 0^+}{\longrightarrow}\Re S(\mathbf{y}^0)=-\mathrm{Vol}(S^3\backslash 7_3).
\end{equation*}
\end{prop}
\begin{proof}
The second statement follows by the first statement and the fact that 
\begin{equation*}
\rho^\prime \mathsf{b}^3(1+o_{\mathsf{b}\rightarrow 0^+}(1))+O_{\mathsf{b}\rightarrow 0^+}(1)
\end{equation*}
is a polynomial in $\mathsf{b}$ if $\mathsf{b}\rightarrow 0^+$.

We prove the first statement.
The integral over $\mathscr{Y}^0$ is divided into the integral over the compact domain $\gamma$ and the integral over the non-bounded domain $\mathscr{Y}^0\backslash\gamma$.

Note that there exists a constant $\delta$ such that for all $\mathbf{y}=(y_1,y_2,y_3)\in\mathscr{Y}^0$, $\Im(y_1),\Im(y_3)\in[-(\pi-\delta),-\delta]$, $\Im(y_2)\in[\delta,\pi-\delta]$.
Let $(\zeta_1,\zeta_2,\zeta_3)\coloneqq (-1,1,-3)$. By Lemma \ref{upperbound-of-dif}, \ref{upperbound-of-dif-negative},
\begin{eqnarray}
& &\left|\Re\left(\frac{1}{2\pi\mathsf{b}^2}S_{\mathsf{b}}(\mathbf{y})-\frac{1}{2\pi\mathsf{b}^2}S(\mathbf{y})\right)\right|\nonumber\\
&=&\left|\Re\left(
\sum_{j=1}^3\zeta_j\left(\mathrm{Log}\left(\Phi_{\mathsf{ b}}\left(\frac{y_j}{2\pi\mathsf{b}}\right)\right)-\left(\frac{-i}{2\pi\mathsf{b}^2}\mathrm{Li}_2(-e^{y_j})\right)\right)\right)\right|\nonumber\\
&\le&\sum_{j=1}^3|\zeta_j|\left|\Re\left(\mathrm{Log}\left(\Phi_{\mathsf{b}}\left(\frac{y_j}{2\pi\mathsf{b}}\right)\right)-\left(\frac{-i}{2\pi\mathsf{b}^2}\mathrm{Li}_2(-e^{y_j})\right)\right)\right|\nonumber\\
&\le&5B_\delta \mathsf{b}^2.\label{upperbound-5Bb^2}
\end{eqnarray}
Focusing on the compact domain $\gamma$, we prove 
\begin{equation*}
\int_{\gamma} d\mathbf{y}e^{\frac{1}{2\pi\mathsf{b}^2}S_{\mathsf{b}}(\mathbf{y})}=e^{\frac{1}{2\pi\mathsf{b}^2}S(\mathbf{y}^0)}\left(\rho^\prime\mathsf{b}^3(1+o_{\mathsf{b}\rightarrow 0^+}(1))+O_{\mathsf{b}\rightarrow 0^+}(1)\right).
\end{equation*}
By Proposition \ref{rho}, setting $\lambda=\frac{1}{2\pi\mathsf{b}^2}$, $\rho^\prime\coloneqq\rho(2\pi)^{\frac{3}{2}}$, it is sufficient to prove 
\begin{equation*}
\int_\gamma d\mathbf{y}e^{\frac{1}{2\pi\mathsf{b}^2}S(\mathbf{y})}\left(e^{\frac{1}{2\pi\mathsf{b}^2}(S_{\mathsf{b}}(\mathbf{y})-S(\mathbf{y}))}-1\right)=e^{\frac{1}{2\pi\mathsf{b}^2}S(\mathbf{y}^0)}O_{\mathsf{b}\rightarrow 0^+}(1).
\end{equation*}
This equality follows by the upper bound $5B_\delta\mathsf{b}^2$ of the inequality \eqref{upperbound-5Bb^2}, the compactness of $\gamma$, and Lemma \ref{strict-max-of-ReS}.

Finally, for the integral on the unbounded domain $\mathscr{Y}^0\backslash\gamma$, we prove
\begin{equation*}
\int_{\mathscr{Y}^0\backslash\gamma}d\mathbf{y} e^{\frac{1}{2\pi\mathsf{b}^2}S_{\mathsf{b}}(\mathbf{y})}=e^{\frac{1}{2\pi\mathsf{b}^2}S(\mathbf{y}^0)}O_{\mathsf{b}\rightarrow 0^+}(1).
\end{equation*}

By the proof of Lemma \ref{upper-bound}, for all $\mathsf{b}<(2\pi A)^{-\frac{1}{2}}$,
\begin{equation*}
\int_{\mathscr{Y}^0\backslash\gamma}d\mathbf{y}e^{\frac{1}{2\pi\mathsf{b}^2}\Re(S)(\mathbf{y})}\le Be^{\frac{1}{2\pi\mathsf{b}^2}M}.
\end{equation*}

Furthermore, for all $\mathsf{b}\in(0,1)$ and $\mathbf{y}\in\mathscr{Y}^0\backslash\gamma$
\begin{equation*}
e^{\frac{1}{2\pi\mathsf{b}^2}\Re(S_{\mathsf{b}}(\mathbf{y})-S(\mathbf{y}))}\le e^{5B_\delta\mathsf{b}^2}.
\end{equation*}

Let $v\coloneqq\frac{\Re(S)(\mathbf{y}^0)-M}{2}$. Then for all $\mathsf{b}>0$ smaller than both $(2\pi A)^{-\frac{1}{2}}$ and $\left(\frac{v}{10\pi B_\delta}\right)^{\frac{1}{4}}$,
\begin{eqnarray*}
\left|\int_{\mathscr{Y}^0\backslash\gamma}d\mathbf{y}e^{\frac{1}{2\pi\mathsf{b}^2}S_{\mathsf{b}}(\mathbf{y})}\right|&=&\left|\int_{\mathscr{Y}^0\backslash\gamma}
d\mathbf{y}e^{\frac{1}{2\pi\mathsf{b}^2}S(\mathbf{y})}e^{\frac{1}{2\pi\mathsf{b}^2}(S_{\mathsf{b}}(\mathbf{y})-S(\mathbf{y}))}\right|  \\
&\le & \int_{\mathscr{Y}^0\backslash\gamma}d\mathbf{y}e^{\frac{1}{2\pi\mathsf{b}^2}\Re(S)(\mathbf{y})}e^{\frac{1}{2\pi\mathsf{b}^2}\Re(S_{\mathsf{b}}(\mathbf{y})-S(\mathbf{y}))}\\
&\le & Be^{\frac{M}{2\pi\mathsf{b}^2}}e^{5B_\delta\mathsf{b}^2}\le Be^{\frac{1}{2\pi\mathsf{b}^2}(M+v)}\\
&=&e^{\frac{1}{2\pi\mathsf{b}^2}S(\mathbf{y}^0)}O_{\mathsf{b}\rightarrow0^+}(1).
\end{eqnarray*}
The above proves the first statement.
\end{proof}
\subsection{From $\mathsf{b}$ to $\hbar$}
For all $\mathsf{b}>0$, $\hbar:=\mathsf{b}^2(1+\mathsf{b}^2)^{-2}>0$ 
was defined.

For $\mathsf{b}>0$, we define a new potential function $S_{\mathsf{b}}^\prime:\mathscr{U}\rightarrow\mathbb{C}$, which is a holomorphic function of three complex variables as follows.
\begin{equation*}
S_{\mathsf{b}}^\prime (\mathbf{y})\coloneqq i\mathbf{y}^{\rm T}Q\mathbf{y}+\mathbf{y}^{\rm T}\mathscr{W}+2\pi\hbar\mathrm{Log}\left(\frac{\Phi_{\mathsf{b}}(\frac{y_2}{2\pi\sqrt{\hbar}})}{\Phi_{\mathsf{b}}(\frac{y_1}{2\pi\sqrt{\hbar}})\Phi_{\mathsf{b}}(\frac{y_3}{2\pi\sqrt{\hbar}})^3}\right).
\end{equation*}
Note that
\begin{equation*}
\left|\mathfrak{J}_X(\hbar,0)\right|=\left|\left(\frac{1}{2\pi\sqrt{\hbar}}\right)^4 \int_{\mathscr{Y}^0}d\mathbf{y} e^{\frac{1}{2\pi\hbar}S_{\mathsf{b}}^\prime(\mathbf{y})}\right|.
\end{equation*}

Let $c_\delta\coloneqq \sqrt{\frac{\delta}{2(\pi-\delta)}}$, then the following holds by Lemma 7.17 of \cite{MR3945172} and its proof.

\begin{lemma}[ \cite{MR3945172}, Lemma 7.17 ]\label{upperbound-of-dif-b}
For all $\delta\in(0,\frac{\pi}{2})$, there exists a constant $C_\delta$ such that the following holds. For all $\mathsf{b}\in(0,c_\delta)$, and for all $y\in\mathbb{R}+i([-(\pi-\delta),-\delta]\cup[\delta,\pi-\delta])$,
\begin{equation*}
\left|\Re\left(\left(\frac{-i}{2\pi\mathsf{b}^2}\mathrm{Li}_2\left(-e^{y(1+\mathsf{b}^2)}\right)\right)-\left(\frac{-i}{2\pi\mathsf{b}^2}(1+\mathsf{b}^2)^2\mathrm{Li}_2(-e^y)\right)\right)\right|\le C_\delta.
\end{equation*}
\end{lemma}

The following holds in a similar way to the proof of Proposition 7.18 in \cite{MR3945172}.
\begin{prop}\label{limit-Sb-prime}
For $\rho^\prime\in\mathbb{C}^\ast$ defined by Proposition \ref{limit-Sb}, the following equation holds if $\hbar\rightarrow 0^+$.
\begin{eqnarray*}
\int_{\mathscr{Y}^0}d\mathbf{y} e^{\frac{1}{2\pi\hbar}S_{\mathsf{b}}^\prime(\mathbf{y})}&=&\int_{\mathscr{Y}^0}d\mathbf{y} e^{\frac{i\mathbf{y}^{\rm T}Q\mathbf{y}+\mathbf{y}^{\rm T}\mathscr{W}}{2\pi\hbar}}\frac{\Phi_{\mathsf{b}}(\frac{y_2}{2\pi\sqrt{\hbar}})}{\Phi_{\mathsf{b}}(\frac{y_1}{2\pi\sqrt{\hbar}})
\Phi_{\mathsf{b}}(\frac{y_3}{2\pi\sqrt{\hbar}})^3}\\
&=& e^{\frac{1}{2\pi\hbar}S(\mathbf{y}^0)}\left(\rho^\prime\hbar^{\frac{3}{2}}(1+o_{\hbar\rightarrow 0^+}(1))+O_{\hbar\rightarrow 0^+}(1)\right).
\end{eqnarray*}
In particular,
\begin{equation*}
(2\pi\hbar)\log\left|\int_{\mathscr{Y}^0}d\mathbf{y} e^{\frac{1}{2\pi\hbar}S_{\mathsf{b}}^\prime(\mathbf{y})}\right|\underset{\hbar\rightarrow 0^+}{\longrightarrow}
\Re S(\mathbf{y}^0)=-\mathrm{Vol}(S^3\backslash 7_3).
\end{equation*}
\end{prop}
\begin{proof}
The second statement follows from the first statement and Lemma \ref{relation-between-S-and-Vol}.

We prove the first statement.

Take $\delta>0$ so that the absolute values of the imaginary part of the coordinates of any $\mathbf{y}\in\mathscr{Y}^0$ are in $[\delta,\pi-\delta]$.
Let $(\zeta_1,\zeta_2,\zeta_3)\coloneqq (-1,1,-3)$.

Then for all $\mathbf{y}\in\mathscr{Y}^0$ and all $\mathsf{b}\in(0,c_\delta)$, by Lemma \ref{upperbound-of-dif},
\ref{upperbound-of-dif-negative}, and \ref{upperbound-of-dif-b},
\begin{eqnarray*}
& &\left|\Re\left(\frac{1}{2\pi\hbar}S_{\mathsf{b}}^\prime(\mathbf{y})-\frac{1}{2\pi\hbar}S(\mathbf{y})\right)\right|\\
&=&\left|\Re\left(
\sum_{j=1}^3\zeta_j\left(\mathrm{Log}\left(\Phi_{\mathsf{b}}\left(\frac{y_j}{2\pi\sqrt{\hbar}}\right)\right)-\left(\frac{-i}{2\pi\hbar}\mathrm{Li}_2(-e^{y_j})\right)\right)\right)\right|\\
&\le& \sum_{j=1}^3|\zeta_j|\left|\Re\left(\mathrm{Log}\left(\Phi_{\mathsf{b}}\left(\frac{y_j(\mathsf{b}^2+1)}{2\pi\mathsf{b}}\right)\right)-\left(\frac{-i}{2\pi\mathsf{b}^2}\mathrm{Li}_2\left(-e^{y_j(1+\mathsf{b}^2)}\right)\right)\right)\right|\\
& &+ \sum_{j=1}^3|\zeta_j|\left|\Re\left(\left(\frac{-i}{2\pi\mathsf{b}^2}\mathrm{Li}_2\left(-e^{y_j(1+\mathsf{b}^2)}\right)\right)-\left(\frac{-i}{2\pi\mathsf{b}^2}
(1+\mathsf{b}^2)^2\mathrm{Li}_2(-e^{y_j})\right)\right)\right|\\
&\le& 5B_{\frac{\delta}{2}}\mathsf{b}^2+5C_\delta\le 5(B_{\frac{\delta}{2}}+C_\delta)
\end{eqnarray*}
holds.

Let $\lambda=\frac{1}{2\pi\hbar}$, and $\hbar$ is taken to be small enough such that 
\begin{equation*}
0<\mathsf{b}<\min \left\{c_\delta, (2\pi A)^{-\frac{1}{2}}, \left(\frac{v}{10\pi(B_{\frac{\delta}{2}}+C_\delta)}\right)^{\frac{1}{2}}\right\},
\end{equation*}
then the remainder of the proof proceeds in a similar way to the proof of Proposition \ref{limit-Sb}.
\end{proof}
By Proposition \ref{limit-Sb-prime}, since $\lim_{\hbar\rightarrow 0^+}2\pi\hbar\log|\mathfrak{J}_X(\hbar, 0)|=-\mathrm{Vol}(S^3\backslash 7_3)$, Theorem \ref{uemura} (3) is proven.

\appendix
\section{The approximate evaluation regarding Theorem\ref{uemura} (3)}
\label{volume-conjecture}
In this appendix, we verify that Theorem \ref{uemura} (3) holds by numerical analysis based on the saddle point method. Therefore, this is not a rigorous proof.
\begin{equation*}
\begin{split}
J_X(\hbar,0)&=
\int_{\mathscr{Y}^\prime} d\mathbf{y}^\prime e^{2i\pi({\mathbf{y}^\prime}^{\rm T}Q\mathbf{y}^\prime)}e^{\frac{1}{\sqrt{\hbar}}{\mathbf{y}^\prime}^{\rm T}\mathscr{W}}\frac{\Phi_{\mathsf{b}}(Z^\prime)}{\Phi_{\mathsf{b}}(Y^\prime)\Phi_{\mathsf{b}}^3(W^\prime)}\\
&=\int_{\mathscr{Y}^\prime} d\mathbf{y}^\prime e^{2i\pi({Y^\prime}^2-Y^\prime Z^\prime +Z^\prime W^\prime +\frac{1}{2}{W^\prime}^2)}e^{\frac{\pi}{\sqrt{\hbar}}(-Y^\prime +W^\prime)}\frac{\Phi_{\mathsf{b}}(Z^\prime)}{\Phi_{\mathsf{b}}(Y^\prime)\Phi_{\mathsf{b}}^3(W^\prime)},
\end{split}
\end{equation*}
where
\begin{eqnarray*}
\mathscr{Y}^\prime
&=&\left(\mathbb{R}-\frac{i}{2\sqrt{\hbar}}(1-2a_1)\right)\times\left(\mathbb{R}+\frac{i}{2\sqrt{\hbar}}(1-2a_2)\right)\times\left(\mathbb{R}-\frac{i}{2\sqrt{\hbar}}(1-2a_3)\right)
\end{eqnarray*}
and 
\begin{equation*}
\mathbf{y}^\prime=\begin{bmatrix}
Y^\prime\\
Z^\prime\\
W^\prime
\end{bmatrix},\quad
\mathscr{W}=\begin{bmatrix}
-\pi\\
0\\
\pi
\end{bmatrix},\quad
Q=\begin{bmatrix}
1 & -\frac{1}{2} & 0\\
-\frac{1}{2} & 0 & \frac{1}{2}\\
0 & \frac{1}{2} &\frac{1}{2}
\end{bmatrix}.
\end{equation*}

By Sections 13.1, 13.2 in \cite{MR3227503},
\begin{eqnarray*}
\int_{\mathbb{R}+\frac{i}{2\sqrt{\hbar}}(1-2a_2)} dZ^\prime 
e^{2\pi iZ^\prime (-Y^\prime +W^\prime)}\Phi_{\mathsf{b}}(Z^\prime)
&=&\frac{e^{-\frac{i}{12}\pi(1-4c_{\mathsf{b}}^2)}e^{2\pi ic_{\mathsf{b}}(-Y^\prime+W^\prime)}}{\Phi_{\mathsf{b}}(Y^\prime -W^\prime -c_{\mathsf{b}})}\\
&=& \frac{e^{-\frac{i}{12}\pi(1+\frac{1}{\hbar})}e^{-\frac{\pi}{\sqrt{\hbar}}(-Y^\prime+W^\prime)}}{\Phi_{\mathsf{b}}(Y^\prime -W^\prime -c_{\mathsf{b}})}.
\end{eqnarray*}

Therefore, in order to show that Theorem \ref{uemura} (3) holds, it is sufficient to prove
\begin{equation}
\lim_{\hbar\rightarrow 0^+}2\pi\hbar\log\left|\int_{\mathscr{Y}}dxdy\frac{e^{i\pi(2x^2+y^2)}}{\Phi_{\mathsf{b}}(x)\Phi_{\mathsf{b}}(x-y-c_{\mathsf{b}})\Phi_{\mathsf{b}}^3(y)}\right|=-\rm{Vol}(S^3\backslash 7_3),
\label{volume-limit}
\end{equation}

where
\begin{eqnarray*}
\mathscr{Y}
&\coloneqq &\left(\mathbb{R}-\frac{i}{2\sqrt{\hbar}}(1-2a_1)\right)
\times\left(\mathbb{R}-\frac{i}{2\sqrt{\hbar}}(1-2a_3)\right).
\end{eqnarray*}

We show the above equation by the approximate evaluation based on the saddle point method. 

Transformation of variables such that $x=\frac{x^\prime}{2\pi\mathsf{b}},y=\frac{y^\prime}{2\pi\mathsf{b}}$ leads to 
\begin{equation}
\begin{split}
&\int_{\mathscr{Y}}dxdy\frac{e^{i\pi(2x^2+y^2)}}{\Phi_{\mathsf{b}}(x)\Phi_{\mathsf{b}}(x-y-c_{\mathsf{b}})\Phi_{\mathsf{b}}^3(y)}\\
&=\frac{1}{(2\pi\mathsf{b})^2}\int_{\mathscr{Y}^\prime}dx^\prime dy^\prime\frac{e^{i\pi \left(2\frac{{x^\prime}^2}{(2\pi\mathsf{b})^2}+\frac{{y^\prime}^2}{(2\pi\mathsf{b})^2}\right)}}{\Phi_{\mathsf{b}}(\frac{x^\prime}{2\pi\mathsf{b}})\Phi_{\mathsf{b}}(\frac{x^\prime-y^\prime}{2\pi\mathsf{b}}-c_{\mathsf{b}})\Phi_{\mathsf{b}}^3(\frac{y^\prime}{2\pi\mathsf{b}})},
\end{split}
\label{2-variable-volume-integral}
\end{equation}
where 
\begin{eqnarray*}
\mathscr{Y}^\prime
&\coloneqq &\left(\mathbb{R}-i\pi (1+\mathsf{b}^2)(1-2a_1)\right)
\times\left(\mathbb{R}-i\pi (1+\mathsf{b}^2)(1-2a_3)\right).
\end{eqnarray*}

$c_{\mathsf{b}}=\frac{i(\mathsf{b}+\mathsf{b}^{-1})}{2}=\frac{i\pi(\mathsf{b}^2+1)}{2\pi\mathsf{b}}$ and 
\begin{equation}
\begin{split}
\int_{\mathscr{Y}}dxdy\frac{e^{i\pi(2x^2+y^2)}}{\Phi_{\mathsf{b}}(x)\Phi_{\mathsf{b}}(x-y-c_{\mathsf{b}})\Phi_{\mathsf{b}}^3(y)}&=\frac{1}{(2\pi\mathsf{b})^2}\int_{\mathscr{Y}^\prime}dxdy\frac{e^{i\pi \left(2\frac{x^2}{(2\pi\mathsf{b})^2}+\frac{y^2}{(2\pi\mathsf{b})^2}\right)}}{\Phi_{\mathsf{b}}(\frac{x}{2\pi\mathsf{b}})\Phi_{\mathsf{b}}(\frac{x-y-i\pi(\mathsf{b}^2+1)}{2\pi\mathsf{b}})\Phi_{\mathsf{b}}^3(\frac{y}{2\pi\mathsf{b}})}
\end{split}.
\label{volume-function1}
\end{equation}
By the equation (A.21) in \cite{MR3700060}, since 
\begin{equation*}
\Phi_{\mathsf{b}}\left(\frac{x}{2\pi{\mathsf{b}}}\right)=\exp\left(\frac{1}{2\pi i\mathsf{b}^2}{\rm{Li}}_2\left(-e^x\right)\right)\left(1+O({\mathsf{b}}^2)\right),
\end{equation*} 
\begin{align}
&\eqref{volume-function1}\notag\\
&=\frac{1}{(2\pi\mathsf{b})^2}\int_{\mathscr{Y}^\prime}dxdy\frac{e^{i\pi \left(2\frac{x^2}{(2\pi{\mathsf{b}})^2}+\frac{y^2}{(2\pi\mathsf{b})^2}\right)}}{\exp{\left(\frac{1}{2\pi i\mathsf{b}^2}{\rm{Li}}_2(-e^x)\right)}\exp{\left(\frac{1}{2\pi i\mathsf{b}^2}({\rm{Li}}_2(e^{x-y})+O(\mathsf{b}^2))\right)}\exp\left(\frac{3}{2\pi i{\mathsf{b}}^2}{\rm Li}_2(-e^y)\right)(1+O({\mathsf{b}}^2))}\notag\\
&=\frac{1}{(2\pi{\mathsf{b}})^2}\int_{\mathscr{Y}^\prime}dxdy e^{\frac{1}{2\pi i{\mathsf{b}}^2}\left(-{\rm{Li}}_2(-e^x)-{\rm{Li}}_2(e^{x-y})-3{\rm Li}_2(-e^y)-x^2-\frac{y^2}{2}+O({\mathsf{b}}^2)\right)}(1+O({\mathsf{b}}^2)).\label{volume-function2}
\end{align}
Let
\begin{equation*}
V(x,y)=-{\rm{Li}}_2(-e^x)-{\rm Li}_2(e^{x-y})-3{\rm Li}_2(-e^y)-x^2-\frac{y^2}{2}.
\end{equation*}
Since 
\begin{equation*}
\frac{\partial}{\partial x}{\rm Li}_2(-e^x)=-\log(1+e^x),
\end{equation*}
\begin{equation*}
\begin{split}
\frac{\partial}{\partial x}V(x,y)&=\log(1+e^x)+\log(1-e^{x-y})-2x\\
&=\log{\frac{(1+e^2)(1-e^{x-y})}{e^{2x}}},
\end{split}
\end{equation*}
and 
\begin{equation*}
\begin{split}
\frac{\partial}{\partial y}V(x,y)&=-\log{(1-e^{x-y})}+3\log{(1+e^y)}-y\\
&=\log{\frac{(1+e^y)^3}{(1-e^{x-y})e^y}}.
\end{split}
\end{equation*}

The stationary points such that
\begin{equation*}
\frac{\partial}{\partial x}V(x,y)=\frac{\partial}{\partial y}V(x,y)=0
\end{equation*}
satisfy
\begin{align}
e^{2x}=(1+e^x)(1-e^{x-y}),\label{saddle-point-cond1}\\
(1-e^{x-y})e^{y}=(1+e^y)^3.\label{saddle-point-cond2}
\end{align}

By the equation $\eqref{saddle-point-cond1}$, 
\begin{equation*}
e^y=-\frac{1+e^x}{e^x-e^{-x}-1}.
\end{equation*}

Substituting this into the equation $\eqref{saddle-point-cond2}$,
\begin{equation*}
(e^{2x}-e^x-1)^2=\frac{(1+2e^x)^3}{e^{3x}}.    
\end{equation*}

Let $e^x=t$, then  
\begin{equation*}
t^3(t^2-t-1)^2=(1+2t)^3.
\end{equation*}

Using Mathematica, we solve the equation numerically with respect to $t$,
\begin{equation*}
\begin{split}
t&=-1,\quad -0.566231,\quad 2.712568,\quad -0.446038-0.121232i,\\
&\quad-0.446083+0.121232i,\quad 0.872869-1.511780i,\quad0.872869+1.511780i.
\end{split}
\end{equation*}

However, if $t=-1$, $e^y=0$, and there is no such complex solution $y$. 

Denoting the value of $V(x,y)$ at the stationary point using $t$, we obtain
\begin{equation*}
\begin{split}
f(t):=&-\left\{{\rm Log}(t)\right\}^2-\frac{1}{2}\left\{{\rm Log}\left(\frac{-1-t}{-1-\frac{1}{t}+t}\right)\right\}^2-{\rm Li}_2(-t)\\
&\quad\quad -{\rm Li}_2\left(\frac{1+t-t^2}{1+t}\right)-3{\rm Li}_2\left(\frac{1+t}{-1-\frac{1}{t}+t}\right).
\end{split}
\end{equation*}
\begin{equation*}
\frac{\partial^2}{\partial x^2}V(x,y)=\frac{e^x}{1+e^x}-\frac{e^{x-y}}{1-e^{x-y}}-2,
\end{equation*}
\begin{equation*}
\frac{\partial^2}{\partial y^2}V(x,y)=-\frac{e^{x-y}}{1-e^{x-y}}+3\frac{e^y}{1+e^y}-1,
\end{equation*}
\begin{equation*}
\frac{\partial^2}{\partial y\partial x}V(x,y)=\frac{\partial^2}{\partial x\partial y}V(x,y)=\frac{e^{x-y}}{1-e^{x-y}}.
\end{equation*}

Let $h(t),i(t),j(t)$ be the value of $\frac{\partial^2}{\partial x^2}V(x,y),\frac{\partial^2}{\partial y^2}V(x,y),\frac{\partial^2}{\partial y\partial x}V(x,y)$ at the stationary point, respectively.
In the neighborhood of the stationary point $(x_0,y_0)$, Taylor expansion gives 

\begin{equation*}
V(x,y)=V(x_0,y_0)+\frac{1}{2}\left\{h(t)(x-x_0)^2+2j(t)(x-x_0)(y-y_0)+i(t)(y-y_0)^2\right\}\cdots.
\end{equation*}

In general, suppose that $A$ is a complex symmetric quadratic matrix and $B$ is a positive definite matrix if $A=B+iC$\quad($B,C$ are real matrices). Then 
\begin{equation*}
\int_{\mathbb{R}^2}\exp(-\frac{1}{2}x^tAx) dx=\frac{2\pi}{\sqrt{\det A}}
\end{equation*}
and the integral on the left-hand side converges.

Therefore, if the integral contour is changed to the contour $x=x_0+\mathbb{R},y=y_0+\mathbb{R}$ which passes through the stationary point $(x_0,y_0)$ and is parallel to the real axes and 
both eigenvalues of the matrix
\begin{equation*}
S(t)=
\begin{pmatrix}
\Im h(t) & \Im j(t) \\
\Im j(t) & \Im i(t) \\
\end{pmatrix}
\end{equation*}
are negative,
let
\begin{equation*}
A(t)=\begin{pmatrix}
 h(t) &  j(t) \\
 j(t) &  i(t) \\
\end{pmatrix},
\end{equation*}
then
\begin{equation*}
\int_{(x_0+\mathbb{R})\times(y_0+\mathbb{R})}dxdye^{\frac{1}{2\pi i{\rm{b}}^2}V(x,y)}=e^{\frac{1}{2\pi i{\mathsf{b}}^2}f(t)}\mathsf{b}^2\left(\frac{-4\pi^2 i}{\sqrt{\det A(t)}}+O({\mathsf{b}})\right).
\end{equation*}

In this case, the left-hand side of the equation $\eqref{volume-limit}$ is equal to 
\begin{equation*}
\Im f(t).
\end{equation*}

By the property of Proposition \ref{property-of-Phi_b} (3), if $0<a_1,a_3,a_1-a_3< \frac{1}{2}$, the integrand of the equation \eqref{2-variable-volume-integral} decreases rapidly at infinity, so by the holomorphicity of the integrand and the Bochner-Martinelli formula, the value of the equation \eqref{2-variable-volume-integral} is invariant if we change the contour with $a_1,a_3$ satisfy the above condition. 
In particular, if $\alpha_X\in\mathscr{A}_X$, by the equations \eqref{it4'} and \eqref{it2''}, $a_1-a_3=a_1-(b_1+b_2)=a_1-b_1-(\frac{1}{2}-a_2-c_2)=a_1-b_1-\frac{1}{2}+2b_1+c_1+c_2=-\frac{1}{2}+(a_1+b_1+c_1)+c_2=c_2$, so, certainly, this condition is satisfied.

In particular, if the imaginary parts of $e^{x_0},e^{y_0}$ are negative and the imaginary part of $e^{x_0-y_0}$ is positive, this condition is satisfied. So, we check the existence of such stationary points $(x_0,y_0)$.

Let $t_3=-0.446038-0.121232i$, then 
\begin{align*}
h(t_3)=-3.67157+1.42489i,\\
i(t_3)=-3.0171-0.961697i,\\
j(t_3)=0.948895-1.80189i
\end{align*}
and the eigenvalues of $S(t_3)$ are 
\begin{equation*}
2.39278,-1.9296.
\end{equation*}

Let $t_5=0.872869-1.511780i$. 
Since $e^y=-\frac{t^2+t}{t^2-t-1}$, if 
$t=t_5$, $e^y=-0.537981 - 1.04357 i$.
Since $e^{x-y}=-\frac{t^2-t-1}{t+1}$, 
if $t=t_5$, $e^{x-y}=0.803839 + 1.25082i$.
\begin{align*}
h(t_5)=-0.445662-1.04125i,\\
i(t_5)=1.81348-3.1839i,\\
j(t_5)=-0.87763+0.780285i,
\end{align*}
and the eigenvalues of  $S(t_5)$ are 
\begin{equation*}
-0.787211,-3.43793.
\end{equation*}

Therefore, the stationary point $(x_0,y_0)$ such that the eigenvalues of $S(t)$ are both negative, the imaginary parts of $e^{x_0},e^{y_0}$ are negative, and the imaginary part of $e^{x_0-y_0}$ is positive exists if $t=t_5$. In this case, 
\begin{equation*}
f(t_5)=2.884158080-4.592125697i
\end{equation*}
and  
\begin{equation*}
\Im f(t_5)=-4.592125697. 
\end{equation*} 

On the other hand, the hyperbolic volume of the complementary space of the knot $7_3$ in $S^3$ is 4.592125697$\cdots$, which numerically confirms the equation $\eqref{volume-limit}$.
\bibliographystyle{alpha} 
\bibliography{reference1}
\end{document}